\documentclass{birkjour}

\usepackage{hyperref}
\newenvironment{changemargin}[2]{\begin{list}{}{%
\setlength{\topsep}{0pt}%
\setlength{\leftmargin}{0pt}%
\setlength{\rightmargin}{0pt}%
\setlength{\listparindent}{\parindent}%
\setlength{\itemindent}{\parindent}%
\setlength{\parsep}{0pt plus 1pt}%
\addtolength{\leftmargin}{#1}%
\addtolength{\rightmargin}{#2}%
}\item }{\end{list}}

\let\oldtocsection=\tocsection

\let\oldtocsubsection=\tocsubsection

\let\oldtocsubsubsection=\tocsubsubsection

\renewcommand{\tocsection}[2]{\hspace{0em}\oldtocsection{#1}{#2}}
\renewcommand{\tocsubsection}[2]{\hspace{1em}\oldtocsubsection{#1}{#2}}
\renewcommand{\tocsubsubsection}[2]{\hspace{2em}\oldtocsubsubsection{#1}{#2}}

  \newtheorem{thm}{Theorem}[section]
  \newtheorem{cor}[thm]{Corollary}
  \newtheorem{lem}[thm]{Lemma}
  \newtheorem{prop}[thm]{Proposition}
  \theoremstyle{definition}
    \newtheorem{defn}[thm]{Definition}
  \theoremstyle{remark}
  \newtheorem{rem}[thm]{Remark}

\usepackage{amsmath,amsfonts,amssymb,amsthm,enumerate,ulem}

 \newtheorem*{conj}{Conjecture}

\usepackage{graphicx}
\usepackage{enumitem}
\usepackage{float}
\usepackage{color}
\usepackage[mathscr]{eucal}
\usepackage[sectionbib]{chapterbib}

\begin{document}
%
\pagenumbering{roman}
\title[Motion of an immersed particle and  point vortex dynamics]{Motion of a  particle immersed in a two dimensional incompressible perfect fluid \textit{and} point vortex dynamics \\ \, \\ Franck Sueur\footnote{Institut de Math\'ematiques de Bordeaux, 
UMR CNRS 5251, Universit\'e de Bordeaux, Franck.Sueur@math.u-bordeaux.fr}}
\maketitle
\setcounter{page}{5}

\begin{abstract}
In these notes, we expose some  recent  works by the author  in collaboration with  Olivier Glass, Christophe Lacave and  Alexandre Munnier, 
establishing point vortex dynamics as zero-radius limits of motions of a rigid body immersed  in a two dimensional  incompressible perfect fluid in several inertia regimes. 
\end{abstract}

\pagenumbering{arabic}\setcounter{page}{1}

\section{Introduction}

The aim of these notes is to present some recent and forthcoming works, in particular  \cite{GLS,GLS2,GMS} and \cite{GLMS},
obtained  by the author  in collaboration with  Olivier Glass, Christophe Lacave and  Alexandre Munnier 
establishing point vortex dynamics as zero-radius limits of motions of a rigid body immersed  in a two dimensional  incompressible perfect fluids in several inertia regimes.

The rigid body is assumed to be only accelerated by the force exerted by the fluid pressure on its boundary according to  the Newton equations, the fluid velocity and pressure being given by the incompressible Euler equations. 
The equations at stake  then read:
\begin{gather}
\label{intrau1}
\displaystyle \frac{\partial u }{\partial t}+(u  \cdot\nabla)u   + \nabla \pi =0 \quad  \text{ and } \quad \operatorname{div} u   = 0  ,
\\\ \label{intrau2} m h'' (t) =  \int_{\partial  \mathcal{S} (t)} \pi n \, ds \quad \text{ and } \quad \mathcal{J} r' (t) =  \int_{\partial  \mathcal{S} (t)} (x-h(t))^{\perp} \cdot \pi n \, ds .
\end{gather}
Here 
\begin{itemize}
\item $u=(u_1,u_2)$ and $\pi$ respectively denote the velocity and pressure fields (the fluid  is supposed to be  homogeneous, of density $1$ in order to simplify the notations),
\item $m >0$ and $\mathcal{J} >0$ denote respectively the mass and the momentum of inertia of the body,
\item  $h'(t)$ is the velocity of the center of mass  $h (t)$ of the body and $r(t)$ is the angular velocity.
The body  rigidly moves so that at times $t$ it occupies a domain $\mathcal{S}(t)$ which is isometric to its initial position  $\mathcal{S}_0$
which is supposed to be a simply connected smooth compact subset of $\mathbb{R}^2$. 
Indeed,  there exists a rotation matrix
\begin{equation}
\label{2drot}
R(\theta(t)):=\begin{pmatrix} \cos \theta (t) & -\sin \theta(t) \\ \sin \theta (t) & \cos \theta(t)\end{pmatrix}
\end{equation}
such that
\[
\mathcal{S}(t) = \{ h(t)+ R(\theta(t)) x , \ x \in \mathcal{S}_{0} \} .
\]
Furthermore, the angle satisfies
$\theta'(t)=r(t) .$
\item when $x=(x_1,x_2)$ the notation $x^\perp $ stands for $x^\perp =( -x_2 , x_1 )$, 
\item $n$ denotes  the unit normal vector pointing outside the fluid domain, which of course depends on the solid position. 
\end{itemize}

We assume that the boundary of the solid is impermeable so that the 
fluid cannot penetrate into the solid and we assume that there is no cavitation as well. 
The natural boundary condition at the fluid-solid interface is therefore 
\begin{equation}
\label{bc-intro}
u  \cdot n =   \Big( h'(t) + r(t)(x-h(t))^{\perp} \Big)  \cdot n   \quad \text{for}\ \  x\in \partial \mathcal{S}  (t) . 
\end{equation}

Let us emphasize that this condition extends the usual condition $u  \cdot n = 0$ on a fixed boundary and involves only the normal part of the fluid velocity. 
As usual with perfect fluid no pointwise boundary condition needs to be prescribed for the tangential part of the fluid velocity. 
However because the domain occupied by the solid is a hole in the fluid domain there is a global condition on the tangential part of the fluid velocity involving the circulation 
 $ \gamma$  defined as 
$$\gamma =  \int_{ \partial \mathcal{S}_0} u  \cdot  \tau \, ds $$
where $u$ is the fluid velocity and $\tau$ the the unit counterclockwise tangential vector so that $n = \tau^\perp$.
Indeed when considering a fluid velocity $u$ which has a good enough regularity, what we will always do in these notes, 
 the so-called Kelvin's theorem applies and $\gamma$ is preserved over time.
 The circulation somehow encodes the amount of vorticity hidden in the particle from the fluid viewpoint. 
Indeed by Green Theorem the circulation can be recast as the integral over $\mathcal{S}_0$ of the vorticity  $ \operatorname{curl} \overline{u}= \partial_1  \overline{u}_2 -  \partial_2  \overline{u}_1$ of any smooth vector field  $\overline{u} $ in $\mathcal{S}_0$ such that $\overline{u} \cdot  \tau = u  \cdot  \tau $ on $ \partial \mathcal{S}_0$.
Therefore the limit where the body  radius converges to zero corresponds to a singular perturbation problem (in space) for the fluid velocity when this latter has to accommodate with a condition of nonzero $ \gamma$ around a shrinking solid. 
Indeed it is well understood since the works \cite{ILN} and \cite{Lopes} that for a solid obstacle held fixed in a perfect incompressible fluid,  with a nonzero given circulation  and with possibly nonzero vorticity in the fluid,  in the limit where the obstacle shrinks into a fixed pointwise particle, 
 the Euler equation driving the fluid evolution has to be modified: in the Biot-Savart law providing the fluid velocity generated by the fluid vorticity, a Dirac mass 
 at the fixed position of the pointwise obstacle with an amplitude equal to the circulation has to be added to the fluid vorticity.  
 We will refer to the background fluid velocity  in the sequel and we will denote it by $ u_\text{bd} $.
 The genuine fluid vorticity $\omega$ (that is without the Dirac mass) is convected by  the background fluid velocity $ u_\text{bd}$.
In the case of a moving body one may wonder if the divergence of the fluid velocity has to be modified  in the zero-radius limit as well in order to accommodate with the non-homogeneous condition \eqref{bc-intro}. In the sequel we will consider some cases where the solid radius $\varepsilon$ shrinks to $0$ with  $(h' , \varepsilon \theta' )$ bounded so that the limit fluid velocity 
is divergence free including in the region where the solid has disappeared.\footnote{Let us observe that since 
$$ \int_{ \partial \mathcal{S}(t)}   \Big( h'(t) + r(t)(x-h(t))^{\perp} \Big)  \cdot n    \, ds = 0 ,$$ see \eqref{Stokes0} below, should the behaviour of the solid velocity with $\varepsilon$ be worse the resulting singularity would be rather a dipole than a Dirac mass.}
Still the analysis of the dynamics of immersed rigid particles requires a more precise analysis, in particular because it is driven by the fluid pressure, a quantity which depends in a non linear and non local way on the fluid velocity. Hence to understand the limit dynamics one has to precisely evaluate the pressure field on the boundary of the solid, that is, where the singularity is concentrated.

One main goal of this line of research  is  prove another derivation of the point vortex dynamics as motions of  immersed particles. 
  In particular we will consider, in some appropriate settings, the limit of the dynamics of an immersed rigid body  when its size and its mass go to zero and recover the equation of a point vortex. Let us recall that the point vortex system is a classical model  which goes back to Helmholtz \cite{Helmholtz}, Kirchhoff \cite{Kirchhoff}, Poincar\'e \cite{Poincare}, Routh \cite{Routh}, Kelvin \cite{Kelvin}, and Lin \cite{Lin1,Lin2}. In these works it was thought  as an idealized fluid model where the vorticity of an ideal incompressible two-dimensional fluid is discrete.  Although it does not constitute a solution of the incompressible Euler equations, even in the sense of distributions,  point vortices can be viewed as Dirac masses obtained as  limits of concentrated smooth vortices which evolve according to the Euler equations. In particular in the case of  a single vortex moving in a bounded and simply-connected domain this was proved by Turkington in \cite{Turk} and an extension to the case of several vortices, including in  the case where there is also a part of the vorticity which is absolutely continuous with respect to Lebesgue measure (the so-called wave-vortex system),  was given by Marchioro and Pulvirenti, see \cite{MP}. 
Let us also mention that Gallay has recently proven in \cite{Gallay} that the point vortex system can also be obtained as vanishing viscosity limits of concentrated smooth vortices evolving according to the incompressible Navier-Stokes equations.
Our main goal is to prove that this classical point vortex motion  can also be viewed as the limit of the dynamics of a solid, shrinking into a pointwise massless particle with fixed circulation, in free motion.
Indeed, our analysis also covers the case where the mass is kept fixed positive in the limit, one then obtains a second-order ordinary differential equation for the particle's position, that we will refer to as a ``massive'' point vortex system.

Let us precise more the two kinds of inertia regimes we are going to consider in the small radius limit  with the following definitions. 

\begin{defn}[Massive and massless particles] \label{massiveP}
We define 
\begin{itemize}
\item a massive particle as the limit of a rigid body
when its radius $\varepsilon$ goes to $0$ with its mass  $m^\varepsilon$ and its momentum of inertia $ \mathcal J^\varepsilon$  
 respectively satisfying   $m^\varepsilon =  m $ and $ \mathcal J^\varepsilon=\varepsilon^2\mathcal J,$
 \item a massless particle as the limit of a rigid body 
when its radius $\varepsilon$ goes to $0$ with its mass  $m^\varepsilon$ and its momentum of inertia $ \mathcal J^\varepsilon$  
 respectively satisfying  $m^\varepsilon= \varepsilon^{\alpha} \,  m $ and $\mathcal J^\varepsilon= \varepsilon^{\alpha+2} \,  {\mathcal J}$,
\end{itemize}
where  $ \alpha > 0$, $m>0$  and  $\mathcal J>0$ are fixed  independent of $\varepsilon$.
\end{defn}

Five remarks are in order: 
\begin{itemize}
\item it is understood that we consider a self-similar shrinking of the rigid body into its center of mass. Choosing the origin $0$ of the frame as the center of mass of $\mathcal S_0$ it means that we will as initial domain, 
for every $\varepsilon \in (0,1]$, 
\begin{equation} 
\label{DomInit}
\mathcal S^\varepsilon_0 :=\varepsilon\mathcal S_0 ,
\end{equation}
\item one observes, of course, that the case of a massive particle corresponds to the limit case $ \alpha = 0$,
\item the scaling of  momentum of inertia $ \mathcal J^\varepsilon$ may look surprising at first sight, but it is quite natural since it corresponds to a second order moment whereas the mass is a zeroth order moment of the body's density, 
\item  it is understood that the circulation $\gamma^\varepsilon $ around the body  satisfies   $\gamma^\varepsilon = \gamma$, where $ \gamma$ is fixed. The amount of circulation is therefore supposed to be independent of the size of the body in our problem. Moreover we assume that $\gamma \neq 0$ in the case of a massless particle.
\item the case where  $\mathcal S_0$ is a homogeneous disk is the most simple whereas the case where $\mathcal S_0$ is a non-homogeneous disk involves some adapted tools in particular in order to deal with the case where  $ \alpha \geq 2$. We refer to \cite{GMS} for a detailed treatment  of these cases and we will consider only here the case where  $\mathcal S_0$ is not a  disk.
\end{itemize}

Let us have a deeper look at Newton's equations \eqref{intrau2}
and anticipate that in the case of a massless particle  the prefactor $m$ and $\mathcal J$ in front of the second-order time derivative  converge to zero in the zero radius limit so that one faces a  singular perturbation problem in time of a non linear dynamics.
We will make use of geodesic and gyroscopic features of the system in order to overcome this difficulty.
Let us mention from now on that the use of the geodesic structure  is more subtle that one may expect at first glance. 
Indeed  the full system ``fluid + rigid body'' is conservative and enjoy a geodesic structure as a whole in the sense that if on a time interval  $(0,T)$
the initial and final configurations are prescribed, then 
the PDE's system  ``fluid + rigid body'' is satisfied on $(0,T)$ if and only if  the couple of flow maps associated with the fluid and  solid velocities 
 is a critical point of the  action obtained by time integration of the total kinetic energy, cf \cite{GS-Arnold}.
This gives some credit to the belief that  the energy conservation drives the dynamics of the system, still  some important transfers of energy from one phase to another may occur  and this lead to a lack of  bound of the solid velocity in the case of a light body.
Since the fluid velocity corresponding to a point vortex is not square integrable   a renormalization of the energy is necessary in the zero radius limit.
Indeed one main feature of the point vortex equation is that the self-induced velocity of the vortex is discarded, or more precisely the self-induced velocity as if the point vortex was alone in the plane. 
Let us therefore introduce $K_{ \mathbb{R}^2} [\cdot] $ the Biot-Savart law in the full plane that is the operator which maps  any reasonable scalar function
 $\omega$ to the unique vector field $K_{ \mathbb{R}^2} [\omega] $ vanishing at infinity and satisfying  
$\operatorname{div} K_{ \mathbb{R}^2} [\omega]  = 0 $ and $\operatorname{curl}  K_{ \mathbb{R}^2} [\omega] = \omega$  in $ \mathbb{R}^2$, which is given by the convolution formula 
\begin{equation} \label{KR2}
K_{\mathbb{R}^{2}}[\omega]:=\frac{1}{2 \pi} \int_{ \mathbb{R}^{2}} \frac{(x-y)^{\perp}}{|x-y|^{2}} \omega (t,y) \, dy .
\end{equation}
We believe that the following statement is true in a very general setting.
\ \par \
\begin{conj}[$\mathcal C$]
A massive particle immersed in a two dimensional  incompressible perfect fluid moves according to Newton's law with a gyroscopic force 
  orthogonally proportional to its relative velocity and  proportional  to the  circulation around  the body. 
    A massless particle  immersed in a two dimensional  incompressible perfect fluid with nonzero circulation moves as a point vortex, its 
 vortex strength being given by the circulation. 
    More precisely, the  position $h(t)$ of a massive (respectively massless) particle satisfies  the equation
\begin{equation}
  \label{natoo}
m h'' = \gamma \big(h' -u_\text{d} (h)\big)^\perp \quad \text{ (resp. } h' = u_\text{d} (h) ),
\end{equation}
with\footnote{The index ``d'' of $u_\text{d} $ can be either interpreted as drift or desingularized as it is obtained from $ u_\text{bd}$ by removing its orthoradial singular part in $h$.}
   $u_\text{d} (h) = \big( u_\text{bd} - K_{ \mathbb{R}^2} [\gamma \delta_h] \big) (h) $
   where  
   $ u_\text{bd}$ denotes the background fluid velocity hinted above. On the other hand the genuine fluid vorticity $\omega$ is convected by  the background fluid velocity $ u_\text{bd}$.
  \end{conj}
\ \par \
\ \par \
Let us mention that in \cite[Chapter 3]{Friedrichs} Friedrichs already evoked a similar conjecture with a massive point vortex system in the case of two point vortices in the whole plane under the terminology of {\it bound vortices} (as opposed to {\it free vortices}). The gyroscopic force appearing in the right hand side of the first equation in \eqref{natoo} is a generalisation of the Kutta-Joukowski lift force which attracted a huge interest at the beginning of the 20th century during the first mathematical investigations in aeronautics, and which will be recalled in Section \ref{chap1}.

Observe that the conjecture above does not mention any sensitivity to the body's shape. Indeed it is expected that a more accurate description of the asymptotic behaviour 
thanks to a multi-scale asymptotic expansion of the body's dynamics in the limit $\varepsilon \rightarrow 0$
 will reveal an influence of the body's shape on some corrector terms which appears as  subprincal  in the limit $\varepsilon \rightarrow 0$ for coarse topologies. 
 The case where the circulation is assumed to vanish with  $\varepsilon$ as $\varepsilon \rightarrow 0$ is  another setting where such a dependence with respect to the  body's shape should appear. 
\ \par \
\ \par \
Indeed this conjecture has already been proved in a few cases and this is precisely the goal of these notes to give an account of these results.

\begin{itemize}
 
 \item \textbf{In Section \ref{chap1},}  we will start with a review of the case, well-known since more than one century,  of  the motion of one single rigid body immersed in an  irrotational  fluid filling the rest of the plane.
 In this setting the equations  at stake are the incompressible Euler equations  \eqref{Euler1} on the fluid domain $\mathcal{F} (t) := \mathbb{R}^2  \setminus \mathcal{S}(t) $, 
the Newton equations \eqref{Solide2}, the interface condition \eqref{bc-intro}  and the following condition of decay at infinity: $ \lim_{|x|\to \infty} |u(t,x)| =0 $.
Regarding the initial conditions we observe that there is no loss of generality in assuming that the center of mass (respectively rotation angle) of the solid
coincides at the initial time with the origin $(0,0)$ (resp. $0$) and we therefore prescribe some initial position and velocity of the solid of the form $(h,h' ,\theta,\theta' )  (0)= (0 , \ell_0 , 0 ,  r_0 )$. On the other hand we prescribe the initial value of the velocity $u |_{t= 0} = u_0 $ in the initial domain $ \mathcal{F}_0 = \mathcal{F} (0) = \mathbb{R}^2 \setminus {\mathcal{S}}_{0} $ occupied by the fluid. Of course since we aim at considering smooth solutions we assume that the fluid and solid initial data are compatible. We may 
therefore consider that the solid translation and rotation velocities $ \ell_0$ and $  r_0$ are  arbitrarily given and that the fluid velocity $u_0 $ is the unique vector field compatible in the sense of the following definition.

\begin{defn}[Compatible initial fluid velocity] \label{CompData}
Given the initial domain $ \mathcal{S}_0$ occupied by the body, $ \ell_0$ and $  r_0$ respectively in $ \mathbb{R}^2$ and $ \mathbb{R}$, and $\gamma$ in $\mathbb{R}$, 
we say that a vector field $u_0$ on the closure of $\mathcal{F}_0 =\mathbb{R}^2 \setminus  \mathcal{S}_0$ with values in $ \mathbb{R}^2$
is compatible if it 
is the unique vector field satisfying 
 the following div/curl type system:
\begin{gather*}
 \operatorname{div} u_0 = 0     \text{ and } 
 \operatorname{curl} u_0  =  0  \text{ in }   \mathcal{F}_0  , \, 
 \\  u_0 \cdot n =   \Big( \ell_0 +  r_0 x^{\perp} \Big)  \cdot n \text{ for }  x \in \partial \mathcal{S}_0  , \, \quad 
  \int_{ \partial \mathcal{S}_0} u_0  \cdot  \tau \, ds=  \gamma ,
\\ \lim_{x\to \infty} u_0 = 0 .
 \end{gather*}  
\end{defn}

  Indeed the zero vorticity condition propagates from $t=0$ according to Helmholtz's third theorem so that at any time $t >0$ 
  the fluid velocity $u(t,\cdot)$ can indeed be recovered from the solid's dynamics  by an elliptic-type problem similar to the one given above for the initial data. 
  Since time appears only as a parameter rather than in the differential operators, the fluid state may be seen as only solving an auxiliary steady problem rather than an evolution equation. 
  The Newton equations can therefore be rephrased as a second-order differential equation whose coefficients are determined by the auxiliary fluid problem. 
In particular the prefactor of the translation and angular accelerations is the sum of  the inertia of the solid and of  the so-called ``added inertia'' which is a symmetric  positive-semidefinite  matrix depending only on the body's shape and which encodes the amount of  incompressible fluid that the  rigid body  has also to accelerate
around itself. Remarkably enough in the case where the circulation is $0$ it turns out that the solid equations can be recast as a geodesic equation associated with the metric given by the total inertia.  Unlike the geodesic structure of the full system ``fluid + rigid body'' hinted above,  the configuration manifold only encodes here  the  solid's dynamics and is therefore of  finite dimensions. This echoes that the equations of motion of point vortices embedded in incompressible flow are usually thought as a reduction of an infinite-dimensional dynamical system, namely the incompressible Euler equation, to a finite-dimensional system.
Another celebrated  feature of the body's dynamics  is due to a  gyroscopic force, proportional to the circulation around the body, known as the
Kutta-Joukowski lift force. 
In order to make  these features appear, cf. Theorem \ref{pasdenom}, and to make as explicit as possible the quantities genuinely involved in this ODE two approaches were followed in the literature: the first one dates back to  Blasius, Kutta, Joukowski, Chaplygin and Sedov, cf. for instance \cite{Sedov}, and relies on complex analysis whereas the second one is real-analytic and was initiated by Lamb, cf. \cite{Lamb}.  We will report here these two methods.\footnote{On the one hand the presentation of the complex-analytic method is extracted from the use we made of it in our first investigations of  the rotational case, cf. Section \ref{chap3} which reports the results of \cite{GLS,GLS2}. On the other hand the  presentation of the real-analytic method is  extracted from the use we made of it in our investigation of the case where  fluid-solid system occupies a bounded domain, cf. Section \ref{chap2} which reports the results of   \cite{GMS}.
Arguably the length comparison and the temporary occurrence of  Archimedes' type quantities (like the volume of the body, its geometric center..) in some intermediate computations leading to  Lemma \ref{PropCia} (where they cancel out) emphasize the superiority of Lamb's method for our purposes. 
Indeed in our forthcoming paper  \cite{GLMS}  Lamb's approach is extended to tackle the  general case where several bodies move in a bounded rotational perfect flow when  some of the
rigid bodies shrink to pointwise particles, some of them with constant mass, the others with vanishing
mass. Still the complex-analytic method is known to be useful to deal with the case of a body whose boundary has singularities thanks to conformal mapping. 
It could be that it appears relevant as well to investigate the motion of a rigid curve resulting from an anisotropic shrinking. In this direction let us mention the paper \cite{Cricri} which deals with the influence of a fixed curve on the fluid around.}
A trivial consequence of this reformulation is that a global-in-time smooth solution to the Cauchy problem exists and is unique.
Therefore in Section \ref{chap1} (cf. below Proposition \ref{CP}) we will prove the following classical result.
\begin{thm}
\label{thm-intro-sec1-CP}
Given the initial domain $\mathcal{S}_0$ occupied by the body, the  initial solid translation and rotation velocities $( \ell_0 , r_0 )$ in $\mathbb{R}^2 \times \mathbb{R}  $, the circulation $ \gamma$ in $ \mathbb{R}$, 
and  $u_0$ the associated compatible initial fluid velocity (according to Definition \ref{CompData}), there exists a unique smooth global-in-time solution to the problem compound of  the incompressible Euler equations  \eqref{intrau1} on the fluid domain, of the Newton equations \eqref{intrau2}, of the interface condition \eqref{bc-intro}, of the  condition  at infinity: $ \lim_{|x|\to \infty} |u(t,x)| =0 $, 
and of the initial conditions $(h,h' ,\theta,\theta' )  (0)= (0 , \ell_0 , 0 ,  r_0 )$ and $u |_{t= 0} = u_0 $.
\end{thm}

Moreover the structure of the reduced ODE hinted above allows to investigate the zero-radius limit quite easily and to obtain the following result.

\begin{thm}
\label{thm-intro-sec1-VO}
Let be given a rescaled initial domain $\mathcal{S}_0$ occupied by the body, some  initial solid translation and rotation velocities $( \ell_0 , r_0 )$  in $\mathbb R^2 \times \mathbb R$ and a  circulation  $\gamma $ in $ \mathbb R$ in the case of a  massive particle and in  $ \mathbb R^*$ in the  case of a massless particle, all of them independent of $\varepsilon$.
Let, for each $\varepsilon > 0$, 
 $u_0^\varepsilon$ the associated compatible initial fluid velocity associated (according to Definition \ref{CompData}) with the initial solid domain ${\mathcal S}_0^\varepsilon$ defined by \eqref{DomInit}, $( \ell_0 , r_0 )$  and $ \gamma$;
 and consider the  corresponding solution given by Theorem \ref{thm-intro-sec1-CP}.
 Then in the zero radius limit  $\varepsilon \rightarrow 0$, with the inertia scaling  described in Definition \ref{massiveP}, 
 one respectively obtains for the position $h(t)$ of the pointwise limit particle 
 the  equations $m h''  = \gamma (h^\prime)^\perp $  in the massive limit and $h^\prime = 0 $ in the massless limit. 
\end{thm}
Therefore, in this historical setting, Conjecture
$(\mathcal C)$ is validated with 
 $u_\text{d} = 0$.
 Of course  Theorem \ref{thm-intro-sec1-VO}
 is a quite informal statement put here for sake of exposition, we will provide a rigorous statement in Section \ref{chap1}, cf. Theorem \ref{pasdenom-VO}.

  \item \textbf{In Section \ref{chap2},}  we consider  the case where the fluid-solid system occupies a bounded domain  $\Omega$, still in the irrotational case. We assume that $\Omega$ is a bounded open regular  connected and simply connected  domain $\Omega$ of $\mathbb R^2$ and that the center of mass of the solid coincides at the initial time with the origin  $0 $ which is assumed to be in  $\Omega.$ 
  
  Again the fluid velocity can be recovered from the solid's dynamics by an elliptic-type problem for which time is only a parameter
 and the Newton equations can therefore be rephrased as a second-order differential equation with geodesic and gyroscopic features
 involving some coefficients determined by this auxiliary fluid problem. 
 
 Still some extra difficulties show up in this process. In particular the way the fluid domain depends on the body motion is more intricate and so are the variations of 
 the added inertia and therefore of the metric given by the total inertia. 
 Indeed even in the case of  zero circulation (i.e. when $\gamma = 0$) 
 the reformulation of the system as an geodesic equation  was  proven only recently in \cite{Munnier}. 
 The general case, with nonzero $\gamma $ is obtained in \cite{GLS}. 
 One another main new feature with respect to the unbounded case is that the Kutta-Joukowski lift force is superseded by a more sophisticated force term which has the form of  the Lorentz force in electromagnetism.  
  Indeed the magnetic part of the Lorentz force, being  gyroscopic and proportional
to the circulation around the body, is a quite natural  extension of
the  Kutta-Joukowski lift force of the  unbounded case. Still it depends on the body position in a more subtle way. 
 On the other hand the electric-type force which may seem very
damaging in order to obtain uniform estimates in the zero-radius  limit as it does not disappear in an energy estimate. 
 
 At least for fixed radius we will be able to infer  straightforwardly from this reformulation the 
  local-in-time existence and uniqueness of  smooth solution to the Cauchy problem. 
Unlike the unbounded case of Section  \ref{chap1}, cf. Theorem \ref{thm-intro-sec1-CP},  the result  is only local-in-time since collision of the body with the external may occur  in finite time (at least as far as it concerns smooth solutions), see  \cite{H,HM} for some examples of collisions  of a disk moving in a potential flow (that is in the case where the circulation  $\gamma$ satisfies $\gamma=0$) with the fixed external boundary of the fluid domain. 
Indeed an energy argument, cf. Corollary \ref{bd-loin}, proves that the life-time of such a smooth solution can only be limited by a collision. 
In order to obtain smooth solutions, even for small time, it is necessary to consider some compatible initial data.
  We therefore adapt the notion of compatible initial fluid velocity introduced in Definition \ref{CompData}
  to the bounded case. 
\begin{defn}[Compatible initial fluid velocity] \label{CompDataBd}
Given the open  regular  connected and simply connected bounded cavity $ \Omega$
and the initial regular closed domain $ \mathcal{S}_0 \subset  \Omega$ occupied by the body, $ \ell_0$ and $  r_0$ respectively in $ \mathbb{R}^2$ and $ \mathbb{R}$, and $\gamma$ in $\mathbb{R}$, we say that a vector field $u_0$ on the closure of $\mathcal{F}_0 =  \Omega \setminus  \mathcal{S}_0$ with values in $ \mathbb{R}^2$
is compatible if it  is the unique vector field satisfying the following div/curl type system:
\begin{gather*}
 \operatorname{div} u_0 = 0     \text{ and } 
 \operatorname{curl} u_0  =  0  \text{ in }   \mathcal{F}_0  , \, 
 \\  u_0 \cdot n =   \Big( \ell_0 +  r_0 x^{\perp} \Big)  \cdot n \text{ for }  x \in \partial \mathcal{S}_0  , \, \quad 
  \int_{ \partial \mathcal{S}_0} u_0  \cdot  \tau \, ds=  \gamma ,
\\  u_0 \cdot n =0  \text{ for }  x \in \partial   \Omega.
 \end{gather*}  
\end{defn}

  Therefore in Section \ref{chap2} we will prove the following.
\begin{thm}
\label{thm-intro-sec2-CP}
Given the open  regular  connected and simply connected bounded cavity $ \Omega$, 
the  initial closed domain $\mathcal{S}_0  \subset  \Omega$ occupied by the body, the  initial solid translation and rotation velocities $( \ell_0 , r_0 )$ in $\mathbb{R}^2 \times \mathbb{R}  $, the circulation $ \gamma$ in $ \mathbb{R}$, 
and  $u_0$ the associated compatible initial fluid velocity (according to Definition \ref{CompDataBd}),
there exists a unique smooth local-in-time solution to the problem compound of  the incompressible Euler equations  \eqref{intrau1} on the fluid domain, of the Newton equations \eqref{intrau2}, of the interface condition \eqref{bc-intro}, of the impermeability  condition  $u \cdot n =0 $ on $\partial \Omega$, 
and the initial conditions $(h,h' ,\theta,\theta' )  (0)= (0 , \ell_0 , 0 ,  r_0 )$ and $u |_{t= 0} = u_0 $.
 Moreover the life-time of such a smooth solution can only be limited by a collision. 
\end{thm}
  This result by itself belongs to the mathematical folklore.\footnote{Indeed a stronger result has been obtained in  \cite{GS-!}, where the rotational case (with $ \operatorname{curl} u_0$  in $L^\infty$)  is handled with pure PDE's methods.} Here it will be easily deduced from the normal form hinted above  
   in order to introduce the solutions which will be tackled in the zero radius limit and because this normal form is precisely the first step of our strategy in order to tackle the singular features of the body's dynamics in the zero radius limit.

Then we expand the coefficients of the previous normal form in the zero-radius limit and repeatedly use Lamb's lemma
to reformulate  under an asymptotic normal form closer to the one of the unbounded case where the leading terms of the electric-type force are absorbed in the other terms by a modulation of the unknown. 
 Indeed we will consider in particular as new unknown a quantity obtained by subtracting a drift velocity given by the leading terms of the electric-type potential
 from the translation velocity. 
 This will allow to extend to this case the vanishingly small limit, still  in both  cases of a limit pointwise particle which is massive or massless.
Indeed  we will obtain both the massive point vortex system and the classical point vortex system
 in a cavity as limit equations for respectively a massive and a massless particle, that is 
from the dynamics of a shrinking solid in the inertia regime mentioned above.

Let  us recall that the Kirchhoff-Routh velocity $u_\Omega$ is 
defined as $u_\Omega := \nabla^\perp \psi_\Omega $,
where $ \nabla^\perp := (-\partial_2 , \partial_1 )$ and where the Kirchhoff-Routh stream function $\psi_\Omega$ is defined as $\psi_\Omega (x) :=  \frac{1}{2} \psi^{0} (x,x) $, where $\psi^{0} (h,\cdot)$ is the solution to the Dirichlet problem: $\Delta \psi^{0} (h,\cdot) =  0 $ in $ \Omega$,  $\psi^{0} (h,\cdot) =  - \frac{1}{2\pi} \ln | \cdot-h |$   on  $\partial\Omega $.
Let us now introduce the limit equation for the case of a massive particle.
\begin{alignat}{3} \label{ODE-mass}
m h''  =\gamma \big(h^\prime- \gamma u_\Omega (h)\big)^\perp \, \text{ on }  [0,T), \, \text{ with }
(h , h' ) (0)= (0,\ell_0)  .
\end{alignat}

The existence of the maximal solution  $(h,T)$ follows from classical ODE theory. Moreover it follows from the conservation of the  energy $ \frac{1}{2} m h'  \cdot h' - \gamma^2 \psi_\Omega (h) $
for any $h \in C^\infty([0,T]; \Omega)$  satisfying \eqref{ODE-mass},
and from the continuity of the Kirchhoff-Routh stream function $\psi_\Omega$ in $\Omega$ 
that $T$ is the time of the first collision of $h$  with the outer boundary $\partial \Omega$ of the fluid domain. 
If there is no collision, then $T=+\infty$. 

Let us also recall the point vortex equation:
\begin{alignat}{3} \label{argh}
h^\prime = \gamma u_\Omega (h)  \text{ for }t>0,   \text{ with } h (0)= 0.
\end{alignat}
It is well-known that the solution $h $ is global in time, and in particular that there is no collision of the vortex point with the external boundary $\partial \Omega$.
This follows from the conservation of the energy $ \gamma^2    \psi_\Omega (h)$ 
 for any $h \in C^\infty([0,T]; \Omega)$ satisfying \eqref{argh}, 
and the fact that 
$\psi_\Omega (h) \rightarrow +\infty$ when $h$ comes close to $\partial \Omega$. \par

Next result states the convergence of $h^\varepsilon$ to the solutions to these equations \eqref{ODE-mass} and \eqref{argh}.

\begin{thm} 
\label{thm-intro-sec2-VO}
Let $\mathcal S_0 \subset \Omega$, $( \ell_0 , r_0  ) \in  \mathbb R^2 \times \mathbb R$,   and  $(m ,\mathcal J) \in  (0,+\infty) \times (0,+\infty)$.
Let,  in the case of a massive (respectively massless) particle, 
 $\gamma $ in $ \mathbb R$ (resp. in $ \mathbb R^*$).
Let  $(h,T)$ be the maximal solution to \eqref{ODE-mass}  (resp.  $h $ be the global solution to \eqref{argh}).
For every $\varepsilon \in (0,1]$ small enough to ensure that the set ${\mathcal S}_0^\varepsilon$ defined by \eqref{DomInit} satisfies
   ${\mathcal S}_0^\varepsilon \subset \Omega $, we consider 
 $u_0^\varepsilon$ the associated compatible initial fluid velocity associated (according to Definition \ref{CompDataBd}) with the initial solid domain ${\mathcal S}_0^\varepsilon$ defined by \eqref{DomInit}, $ \Omega$, $( \ell_0 , r_0 )$  and $ \gamma$,
and  we denote $T^\varepsilon$ the life-time of the associated smooth solution $(h^\varepsilon  ,\theta^\varepsilon  ,u^\varepsilon  )$ given by Theorem \ref{thm-intro-sec2-CP}, with the inertia scaling  described in Definition \ref{massiveP} 
and  the initial conditions $(h^\varepsilon ,(h^\varepsilon )' ,\theta^\varepsilon,(\theta^\varepsilon)' )  (0)= (0 , \ell_0 , 0 ,  r_0 )$ and $u^\varepsilon  |_{t= 0} = u^\varepsilon_0 $.
Then in the zero radius limit $\varepsilon\to 0$, there holds 
 $\liminf T^\varepsilon\geqslant T$ (resp.  $\liminf T^\varepsilon \to +\infty$) and
 $h^\varepsilon \relbar\joinrel\rightharpoonup h$ 
 in $W^{2,\infty}([0,T];\mathbb R^2)$ weak-$\star$  (resp. $W^{1,\infty}([0,T];\mathbb R^2)$ weak-$\star$).
\end{thm}
Therefore, in the bounded setting, Conjecture
$(\mathcal C)$ is also true with   $u_\text{d} =  \gamma u_\Omega$ (and a easy byproduct of the analysis is that  $u_\text{bd} = K_{ \Omega} [\gamma \delta_h] $ where $K_{ \Omega}$ denotes the Biot-Savart law associated with the simply connected domain  $\Omega$).
Theorem  \ref{thm-intro-sec2-VO} will be proven as a consequence of  Theorem \ref{THEO-intro} and Theorem  \ref{Bounded-VO}.
This result was obtained in \cite{GMS}.


  \item \textbf{In Section \ref{chap3},}   we will consider the motion of a rigid body immersed in a two dimensional  incompressible perfect fluid with vorticity. 
  In order to focus on the interaction between the  rigid body and the fluid vorticity we go back to the unbounded setting of Section \ref{chap1} so that the fluid-solid system occupies again the whole plane.   We first recall a result obtained in  \cite{GLS} establishing a global in time existence and uniqueness result similar to the celebrated result by Yudovich about the case of a fluid alone. Yudovich's theory relies on the transport of the fluid vorticity in particular to the preservation of  $L^{\infty}$-in space bound of the vorticity when time proceeds and we therefore extend Definition \ref{CompData} to this setting.
%
\begin{defn}[Compatible initial fluid velocity] \label{CompDataY}
Given the initial domain $ \mathcal{S}_0$ occupied by the body, the initial solid velocities  $ \ell_0$ and $  r_0$ respectively in $ \mathbb{R}^2$ and $ \mathbb{R}$, an initial fluid vorticity  $\omega_0$ in $ L_c^{\infty}(\mathbb{R}^{2}\setminus\{0\})$ 
and $\gamma$ in $\mathbb{R}$, 
we say that a vector field $u_0$ on the closure of $\mathcal{F}_0 =\mathbb{R}^2 \setminus  \mathcal{S}_0$ with values in $ \mathbb{R}^2$
is compatible if it 
is the unique vector field in $C^{0}(\overline{\mathcal{F}_0};\mathbb{R}^{2})$
satisfying 
 the following div/curl type system:
\begin{gather*}
 \operatorname{div} u_0 = 0     \text{ and } 
 \operatorname{curl} u_0  =  \omega_0  \text{ in }   \mathcal{F}_0  , \, 
 \\  u_0 \cdot n =   \Big( \ell_0 +  r_0 x^{\perp} \Big)  \cdot n \text{ for }  x \in \partial \mathcal{S}_0  , \, \quad 
  \int_{ \partial \mathcal{S}_0} u_0  \cdot  \tau \, ds=  \gamma ,
\\ \lim_{x\to \infty} u_0 = 0 .
 \end{gather*}  
\end{defn}
We are now ready to state the existence and uniqueness result with bounded vorticity. 
\begin{thm} \label{ThmYudo-intro}
For any $(\ell_0,r_0) \in \mathbb{R}^2 \times \mathbb{R}$, $\omega_0  \in L_c^{\infty}(\overline{{\mathcal F}_0})$, 
there exists a unique solution to  the problem compound of  the incompressible Euler equations  \eqref{intrau1} on the fluid domain, of the Newton equations \eqref{intrau2}, of the interface condition \eqref{bc-intro}, of the  condition  at infinity: $ \lim_{|x|\to \infty} |u(t,x)| =0 $, 
and of the initial conditions $(h,h' ,\theta,\theta' )  (0)= (0 , \ell_0 , 0 ,  r_0 )$ and $u |_{t= 0} = u_0 $, with  $u_0$ the compatible initial velocity associated with  $\ell_0$, $r_0$ and $\omega_0$ by  Definition \ref{CompDataY}.
Moreover $(h, \theta) \in  C^2 (\mathbb{R}^+; \mathbb{R}^2 \times \mathbb{R}) $ and  for all $t>0$, $\omega (t):=\operatorname{curl} u(t) \in L^\infty_c(\overline{{\mathcal F} (t)})$.
\end{thm}
A key of Yudovich's approach is that a $L^{\infty}$-in space bound of the vorticity is enough to control the log-Lipschitz regularity of the fluid velocity.\footnote{Indeed the uniqueness part of the result above has to be understood to hold in the class of solutions with such a regularity.} 
We will see in Section \ref{chap3} that this is still true for the case with an immersed body. Moreover this amount of regularity is sufficient in order to insure convenient a priori bounds regarding the solid motion.  
  We refer to \cite{GS,GS2,Sueur}  for some other results regarding the existence of solutions with less regularity.  Indeed we will slightly modify the proof of Theorem \ref{ThmYudo-intro} given in  \cite{GLS} regarding the  a priori estimates of the rigid body's acceleration. 
 We will use here an argument from  \cite{GS2,Sueur}  which requires pretty much only a $L^2$ type a priori estimate of the fluid velocity. 

This setting will allow to investigate the zero radius limit. In  \cite{GLS} and \cite{GLS2} we have obtained respectively the following results corresponding to  the massive and  massless cases.

\begin{thm}
 \label{thm-intro-sec3-VO}
Let  be given
 a  circulation  $\gamma $ in $ \mathbb R$ in the case of a  massive particle and in  $ \mathbb R^*$ in the  case of a massless particle, 
 $(\ell_0 ,r_0 ) \in \mathbb{R}^3$,  $ \omega_0 $ in  $L_c^{\infty}(\mathbb{R}^{2}\setminus\{0\})$ and 
 consider as initial fluid velocity  $u_0^\varepsilon$ is then defined as the unique vector field  $u_0^\varepsilon$ compatible with $ \mathcal{S}_0^\varepsilon$,  $ \ell_0$,  $  r_0$,  $\gamma$ and $\omega_{0}^{\varepsilon} := \omega_{0|{\mathcal F}_{0}^{\varepsilon}}$.
   For any $\varepsilon \in (0,1]$, let us denote $(h^{\varepsilon},r^{\varepsilon},u^\varepsilon)$ the corresponding solution of the system.
  Then in the zero radius limit  $\varepsilon \rightarrow 0$, with the inertia scaling  described in Definition \ref{massiveP},  one respectively obtains the following equation for the position $h(t)$ of the pointwise limit particle:
 \begin{equation}
  \label{mls1}
 m h''(t) = \gamma \Big(h'(t) - K_{\mathbb{R}^{2}}[\omega(t,\cdot)] (h(t))  \Big)^\perp ,
\end{equation}
   in the massive limit  and 
 \begin{equation}
  \label{mls2}
     h'(t) =  K_{\mathbb{R}^{2}}[\omega(t,\cdot)] (h(t))  ,
     \end{equation}
  in the massless limit.
 Regarding the fluid state, one obtain at the limit the following transport equation for the fluid vorticity:
\begin{gather}
 \label{EulerPoint} 
  \frac{\partial \omega }{\partial t}+ \operatorname{div}
\big(   \omega  K_{\mathbb{R}^{2}}[\omega +  \gamma \delta_{h} ]   \big) = 0 .
\end{gather}
\end{thm}
%

%
Therefore Conjecture
$(\mathcal C)$ is also true with 
 $u_\text{d} = K_{\mathbb{R}^{2}}[\omega]  $ and   $u_\text{bd} =  K_{\mathbb{R}^{2}}[\omega +  \gamma \delta_{h} ] $.

The fluid equation (\ref{EulerPoint}) is  the same whether the body shrinks to a massive or a massless pointwise particle.
Equation (\ref{EulerPoint}) describes the evolution of the vorticity of the fluid: it is transported by a velocity obtained by the usual Biot-Savart law in the plane, but from a vorticity which is the sum of the fluid vorticity and of a point vortex placed at the (time-dependent) position $h(t)$ where the solid shrinks, with a strength equal to the circulation $\gamma$ around the body. 
Equation  \eqref{EulerPoint} and its corresponding  initial condition hold in the sense that  for any test function $\psi\in C^\infty_c([0,T)\times\mathbb{R}^2)$ we have 
\begin{gather*} 
\int_0^\infty \int_{\mathbb{R}^2} \psi_t  \omega \, \, dx\, dt
+\int_0^\infty \int_{\mathbb{R}^2} \nabla_x \psi \cdot  K_{\mathbb{R}^{2}}[\omega +  \gamma \delta_{h} ]  \omega  \, dx\, dt
\\ +\int_{\mathbb{R}^2} \psi(0,x) \omega_0 (x) \, dx=0 .
\end{gather*} 
The uniqueness of the solution to the massive limit system\eqref{mls1}-\eqref{EulerPoint}
 is an interesting question. Observe that a putative uniqueness result would entail the convergence of the whole sequence.
This is the case for the massless limit system  \eqref{mls2}-\eqref{EulerPoint} for which uniqueness does hold according to a result due to 
 Marchioro and Pulvirenti,  cf. \cite{MP} and revisited by Lacave and Miot, cf. \cite{CricriLylyne}.
 
 We will  take advantage of the approach developed in  \cite{GLMS} and exposed in 
Section \ref{chap2} to provide a sketch of a more simple proof of the results claimed in Theorem \ref{thm-intro-sec3-VO} than the ones achieved  in \cite{GLS} for the massive case and most of all in \cite{GLS2} for the massless case. 
In order to do so we will start with an exact reformulation of the body's dynamics for fixed radius into an ODE 
with a geodesic feature, a Lorentz type force of the same form than the one mentioned above in the irrotational case, but with an extra dependence to the fluid vorticity,  and a new term describing a somehow more direct influence of the vorticity but which does not enjoy much structure, cf. Theorem  \ref{enmarche}.
Then we expand the coefficients of the previous ODE in the zero-radius limit
using in particular 
an irrotational approximation of the fluid velocity on the body's boundary in order to use 
Lamb's lemma. This provides in particular the leading terms in the expansion of the Lorentz type force with less effort than by the complex-analytic method used in  \cite{GLS} and \cite{GLS2}.
Another simplification comes from the extra term encoding a  direct influence of the vorticity to the body's dynamics which is more simple to estimate than its counterpart in \cite{GLS2}.\footnote{In particular it avoids again    the temporary occurrence of  Archimedes' type quantities.}
We thus obtain an asymptotic normal form once again close to the one of Section \ref{chap1} 
 where the leading terms of the electric-type force are absorbed in the other terms by a modulation of the solid translation velocity by a drift velocity of the particle under the influence of the fluid vorticity. 
 This will allow to extend to this case the vanishingly small limit, still  in both  cases of a limit pointwise particle which is massive or massless.
Indeed Theorem \ref{thm-intro-sec3-VO} will be proven in Section
\ref{sec3-VO} as a consequence of 
Theorem \ref{yudo-VO}.
\end{itemize}

We aim at extending this analysis to the case where both interactions of a body with an external boundary and with the fluid vorticity are considered in the same time,  for  several  massive and massless particles, which will provide a positive answer to  Conjecture $(\mathcal C)$ in a wide setting, cf. the ongoing work  \cite{GLMS}.


\section{Case of an unbounded irrotational flow}
\label{chap1}
In this section we assume that the system ``fluid + solid'' is unbounded so that the domain occupied by the fluid  at  time $t$ is 
$\mathcal{F} (t) := \mathbb{R}^2  \setminus \mathcal{S}(t) $
  starting from the initial domain 
 $\mathcal{F}_{0}  := \mathbb{R}^2 \setminus {\mathcal{S}}_{0} $.
The equations at stake  then read :
\begin{gather}  \label{Euler1}
\displaystyle \frac{\partial u }{\partial t}+(u  \cdot\nabla)u   + \nabla \pi =0  \quad \text{ and } \quad
\operatorname{div} u   = 0 \quad \text{for}\ x\in \mathcal{F}(t) ,
\\ m h'' (t) =  \int_{\partial  \mathcal{S} (t)} \pi n \, ds \quad \text{ and } \quad \mathcal{J}  \theta'' (t) =  \int_{\partial  \mathcal{S} (t)} (x-h(t))^{\perp} \cdot \pi n \, ds ,  \label{Solide2} 
\\ u  \cdot n =   \Big( h'(t) + r(t)(x-h(t))^{\perp} \Big)  \cdot n   \quad \text{for}\ \  x\in \partial \mathcal{S}  (t),   
\\  \lim_{|x|\to \infty} |u(t,x)| =0 , \label{Euler3b}
\\ u |_{t= 0} = u_0  \quad \text{for} \  x\in  \mathcal{F}_0   \quad \text{ and } \quad 
(h,h' ,\theta,\theta' )  (0)= (0 , \ell_0 , 0 ,    r_0 ).  \label{Solideci}
\end{gather}

\subsection{Reduction to an ODE. Statement of Theorem \ref{pasdenom}}

In the irrotational case, the system above can be recast as an ODE whose unknowns are the degrees of freedom of the solid, namely $h$ and  $\theta$.
In particular the motion of the fluid is completely determined by the solid position and velocity.
In order to state this, let us introduce the variable 
$q:=(h , \theta)\in\mathbb R^3.$
Since the fluid and solid domains only depend on $t$ through the solid position, we will rather denote them respectively 
$\mathcal{F}(q) $ and $\mathcal{S}(q) $ instead of $\mathcal{F} (t) $ and $\mathcal{S}(t) $.

Let us gather the mass and moment of inertia  of the solid into the following  matrix:
\begin{equation} \label{DefMG}
{\mathcal M}_{g} := \begin{pmatrix}
	m & 0 & 0 \\
	0 & m & 0 \\
	0 & 0 & {\mathcal J}
\end{pmatrix} .
\end{equation}
Observe that ${\mathcal M}_{g}$ is diagonal and in the set $S^{++}_3 (\mathbb{R})$  of the real symmetric positive definite $3\times 3$ matrices.
As already mentioned in the introduction the reformulation relies on the phenomenon of added mass, which, loosely speaking, measures how much the  surrounding fluid resists the acceleration as the body moves through it. This will be encoded by a matrix $\mathcal{M}_a$ in the set $S^{+}_3 (\mathbb{R})$  of the real symmetric positive-semidefinite $3\times 3$ matrices. The index $a$ refers to ``added'', by opposition to the genuine inertia ${\mathcal M}_{g}$.
This matrix  $\mathcal{M}_a$ depends on  the shape of the domain occupied by the solid and therefore on the solid position. 
Still since the system ``fluid + solid'' occupies the full plane the added inertia is invariant by translation and therefore only depends on $\theta$.
In order to measure its variations let us denote by 
$\mathcal{BL} (\mathbb{R}^3 \times \mathbb{R}^3 ; \mathbb{R}^3 )$  the space of bilinear mappings from  $ \mathbb{R}^3 \times \mathbb{R}^3 $ to $ \mathbb{R}^3$.
\begin{defn}[a-connection] \label{Christ}
Given a $C^{\infty}$ mapping
$ \theta \in \mathbb R \mapsto {\mathcal M}_{a,\,  \theta} \in S^{+}_3 (\mathbb{R})$, 
we say that the $C^{\infty}$ mapping
$ \theta \in \mathbb R \mapsto  \Gamma_{a,\theta}  \in  \mathcal{BL} (\mathbb{R}^3 \times \mathbb{R}^3 ; \mathbb{R}^3 ) $  is the  associated  a-connection  if for any 
$ p \in\mathbb R^3$,
\begin{subequations} 
\begin{equation}
\langle\Gamma_{a,\theta}  ,p,p\rangle :=\left(\sum_{1\leqslant i,j\leqslant 3} (\Gamma_{a,\theta} )^k_{i,j} p_i p_j  \right)_{1\leqslant k\leqslant 3}\in\mathbb R^3 ,
\end{equation}
with  for every $i,j,k \in \{1,2,3\}$  and for any $q=(h , \theta)$, 
\begin{equation} 
 \label{stanford}
(\Gamma_{a,\theta} )^k_{i,j} (q)  :=  \frac12
\Big(  ({\mathcal M}_{a, \theta})_{k,j}^{i} + ({\mathcal M}_{a, \theta})_{k,i}^{j} - ({\mathcal M}_{a, \theta})_{i,j}^{k}  \Big) (q) ,
\end{equation}
where $({\mathcal M}_{a, \theta})_{i,j}^{k} $ denotes the partial derivative with respect to $ q_{k}$ of the entry of indexes $(i,j)$ of the matrix ${\mathcal M}_{a, \theta}$, that is
\begin{equation} 
({\mathcal M}_{a, \theta})_{i,j}^{k} := \frac{\partial ({\mathcal M}_{a, \theta})_{i,j}}{\partial q_{k}} .
\end{equation}
\end{subequations}
\end{defn}

As already mentioned  in the introduction another celebrated  feature of the body's dynamics in the case of an unbounded irrotational flow is  the
Kutta-Joukowski force.  
This force  also depends on  the shape of the domain occupied by the solid and in particular on the solid position through $\theta$ only. 
Since this force is  gyroscopic, i.e. orthogonal to the velocity $q'$ (which gathers both  translation and rotation velocities), and proportional to the circulation around the body, it will be encoded by a vector $B_{\theta}$ in $\mathbb{R}^3$.

The first main result of this section is the following reformulation of the equations  \eqref{Euler1}-\eqref{Euler3b} into an ODE for  the degrees of freedom of the solid only.

\begin{thm}
\label{pasdenom}
There exists a $C^{\infty}$ mapping
$ \theta \in \mathbb R \mapsto \big({\mathcal M}_{a, \theta} ,  B_{\theta} \big)    \in  S^{+}_3 (\mathbb{R})  \times \mathbb{R}^3$  
depending only on $  \mathcal S_0 $
such that the equations  \eqref{Euler1}-\eqref{Euler3b} are equivalent to the following ODE for  $q=(h , \theta)$:
\begin{equation}  \label{ODE_ext}
 ({\mathcal M}_{g}   +  {\mathcal M}_{a, \theta} )  \, q''  
  + \langle  \Gamma_{a,\theta} ,q',q'\rangle
 =   \gamma q' \times B_{ \theta} ,
\end{equation}
where $ \Gamma_{a,\theta}$ denotes the  a-connection associated with ${\mathcal M}_{a, \theta}$,
the fluid velocity $u$ being given with respect to $q$ and $q' = (h' ,\theta')$ as the unique solution to the following div/curl type system:
\begin{gather*}
 \operatorname{div} u = 0     \text{ and } 
 \operatorname{curl} u  =  0  \text{ in }   \mathcal{F}(q)   , \, 
 \\  u \cdot n =   \Big( h' + \theta'(x-h)^{\perp} \Big)  \cdot n \text{ for }  x \in \partial \mathcal{S} (q)  , \, \quad 
  \int_{ \partial \mathcal{S} (q)} u  \cdot  \tau \, ds=  \gamma ,
\\ \lim_{x\to \infty} u = 0 .
 \end{gather*}  
\end{thm}

Indeed the matrix  $\mathcal{M}_{a,\,  \theta} $,  its associated a-connection $\Gamma_{a,\theta} $ and $B_{ \theta}$ will be given by precise formulas in the next section, cf. \eqref{AddedMass}, \eqref{cla2} and \eqref{DefB}. Let us already mention here that the dependence on $ \theta$ of $ \big({\mathcal M}_{a, \theta} ,  B_{\theta} \big) $ is quite simple since
\begin{equation*}
{\mathcal M}_{a, \theta}  = \mathcal{R}(\theta) {\mathcal M}_{a,0} \mathcal{R}(\theta)^t \text{ and } B_{\theta} = \mathcal{R}(\theta)B_{0} ,
\end{equation*}
where we associate 
 the $3 \times 3$ rotation matrix 
\begin{equation}
 \label{3dR}
\mathcal{R}(\theta) :=
\begin{pmatrix}
{R}(\theta)  & 0 
\\ 0 & 1
\end{pmatrix}   \in{\rm SO}(3) 
\end{equation}
with the $2 \times 2$ rotation matrix
$R(\theta)$ defined in \eqref{2drot}.

\begin{rem}
\label{rem-order}
Let us emphasize that the coefficients $ (\Gamma_{a,\theta} )^{k}_{i,j}$ defined in  \eqref{stanford}
are not the Christoffel symbols associated with ${\mathcal M}_{a, \theta}$ nor ${\mathcal M}_{g}   +  {\mathcal M}_{a,\,  \theta}$. Indeed, one should multiply by $({\mathcal M}_{g}   +  {\mathcal M}_{a,\,  \theta} )^{-1}$ the  column vector of the   $ (\Gamma_{a,\theta} )^{k}_{i,j}$ indexed by $k$ to get the standard Christoffel symbols:
\begin{equation} 
 \label{stan}
\Big( \Gamma^k_{i,j} (q) \Big)_{k \in \{1,2,3\}}   =  ({\mathcal M}_{g}   +  {\mathcal M}_{a,\,  \theta} )^{-1}  \Big( (\Gamma_{a,\theta} )^k_{i,j} (q) \Big)_{k \in \{1,2,3\}}  ,
\end{equation}
for any $i,j \in \{1,2,3\}$.
The reason for deviating from the standard notations is that we want to keep track of the two types of inertia for the subsequent asymptotic analysis of the zero-radius limit. 
Observe that the a-connection only involves the added inertia.
Still it is worth highlighting that  should its right hand side vanish \eqref{ODE_ext} would be the geodesic equation associated
with the metric ${\mathcal M}_{g}   +  {\mathcal M}_{a,\,  \theta} $: 
\begin{equation*}  
(q^k)''   +  \Gamma^k_{i,j}(q)  (q^i)' (q^j)'  =0 ,
\end{equation*}
where the $ \Gamma^k_{i,j} $ are the standard Christoffel symbols defined in  \eqref{stan} and where Einstein summation notation is used. 
\end{rem}

We will prove Theorem \ref{pasdenom} in the sequel but we first deduce and prove the following result.

 \begin{prop}
 \label{CP}
  Given some initial data 
$(q,q')  (0)= (0 , \ell_0 , 0 ,  r_0 )$ there exists a unique global solution  $q\in C^\infty([0,+\infty) ; \mathbb R^3)$ to  \eqref{ODE_ext}. 
Moreover  the quantity 
\begin{equation}
 \label{NR}
\frac12 ( {\mathcal M}_{g}   +  {\mathcal M}_{a, \theta} ) \, q' \cdot q' 
\end{equation}
 is conserved. 
 \end{prop}
  
  Theorem \ref{thm-intro-sec1-CP} is then a consequence of Theorem \ref{pasdenom} and Proposition \ref{CP}.

\begin{rem}
\label{remarep}
The quantity  \eqref{NR} corresponds to the sum of the kinetic energy of the solid associated with its genuine inertia and of the one associated with the added inertia. 
It will become apparent in the sequel, cf. Section \ref{sec-decompo}, that the  kinetic energy  associated with the added inertia of the rigid body can be also interpreted as a renormalization of
 the kinetic energy of the fluid by retaining only the potential contribution  and discarding the term due to the circulation around the body. 
\end{rem}

\begin{proof}
Local existence and uniqueness follow from  classical ODE theory.  Global  existence  would be a consequence of the energy conservation.  
Indeed defining 
 for any $\theta $ in $\mathbb R$, for any $p$ in $\mathbb R^3$,  the matrix 
\begin{equation*} 
S_{a,\theta} (p) := \left(\sum_{1\leq i\leq 3} (\Gamma_{a,\theta} )^k_{i,j} p_i  \right)_{1\leq k,j\leq 3}  \text{ so that  }
\langle \Gamma_{a,\theta} , p, p \rangle = S_{a,\theta} (p)  p ,
\end{equation*}
then, an explicit computation proves that for any  $\theta $ in $\mathbb R$, for any $p$ in $\mathbb R^3$,  
\begin{equation*}
\frac{1}{2}    \frac{\partial  {\mathcal M}_{a, \theta}}{\partial q}    (\theta) \cdot  p
- S_{a,\theta} (p)   \text{ is skew-symmetric.}
\end{equation*}
Therefore for any  $\theta $ in $\mathbb R$, for any $p$ in $\mathbb R^3$,  
\begin{equation*} 
\langle \Gamma_{a,\theta} , p, p \rangle \cdot p = \big( S_{a,\theta} (p) p\big) \cdot p = \frac{1}{2}   \big( \frac{\partial  {\mathcal M}_{a, \theta}}{\partial q}    (\theta) p\big) \cdot  p , 
\end{equation*}
so that taking the inner product of  \eqref{ODE_ext} with $q'$ yields on the one hand 
\begin{equation*} 
 ({\mathcal M}_{g}   +  {\mathcal M}_{a, \theta} )  \, q''   \cdot q' = \frac{1}{2}  \big(    ({\mathcal M}_{g}   +  {\mathcal M}_{a, \theta} )  \, q'   \cdot q'  \big)'
 -  \frac{1}{2}   \big(    \frac{\partial {\mathcal M}_{a, \theta}}{\partial q}   (\theta)   \, q'  \big)  \cdot q' 
\end{equation*}
and on the other hand
\begin{equation*} 
\langle \Gamma_{a,\theta} , q', q' \rangle \cdot q' =  \frac{1}{2}   \big( \frac{\partial  {\mathcal M}_{a, \theta}}{\partial q}    (\theta) q'\big) \cdot  q' , 
\end{equation*}
Therefore when doing the energy estimate the term coming from the a-connection exactly compensates the term coming from the commutation of one time derivative in the acceleration term, and 
 the conservation of  the energy
  follows by observing that the contribution of the right hand side of \eqref{ODE_ext} is $0$.

\end{proof}
%


\subsection{Explicit definition of the ODE coefficients $\mathcal{M}_{a,\,  \theta} $,  $\Gamma_{a,\theta} $ and $B_{ \theta}$}
\label{sec-ex}

In this section we are going to provide some precise formulas for the matrix $\mathcal{M}_{a,\,  \theta} $ (and for its associated a-connection $\Gamma_{a,\theta} $  as well) and for the vector $B_{ \theta}$ thanks to some elementary flows corresponding respectively to potential and circulatory type flows. 
Indeed we will see in 
Section \ref{sec-decompo} (in a different frame, moving with the body)  that the real flow can be decomposed thanks to these elementary flows.

\subsubsection{Kirchhoff  potentials}
The following so-called Kirchhoff potentials $\Phi:=(\Phi_{i})_{i=1,2, 3}$ will play a major role. 
They are defined as the solutions of the following problems:
\begin{equation}\label{def Phi}
-\Delta \Phi_i = 0 \quad   \text{in } \mathcal{F}_{0} ,\quad
\Phi_i \longrightarrow 0 \quad  \text{when }  x \rightarrow  \infty, \quad
\frac{\partial \Phi_i}{\partial n}=K_i  \quad  \text{on }   \partial \mathcal{F}_{0}  ,
\end{equation}
where 
$(K_{1},\, K_{2}, \, K_{3}) :=(n_1,\, n_2 ,\, x^\perp \cdot n)$.
The compatibility condition for this Neumann Problem is satisfied 
i.e. we check that 
\begin{equation}
\label{Stokes0}
\int_{\partial  \mathcal{S}_{0}  } K_i \, ds = 0, 
\end{equation}
  for $i=1,2,3$, which is obvious by the Stokes formula :
\begin{gather*}
\int_{ \partial  \mathcal{S}_{0} } n_{1}\,ds= -\int_{  \mathcal{S}_{0} } \operatorname{div} (e_1) \,dx=0
  \text{  and similarly  } 
  \int_{ \partial  \mathcal{S}_{0} } n_{2}\,ds = 0 ,
\\  \text{ whereas } 
\int_{ \partial  \mathcal{S}_{0} } x^\perp \cdot n\,ds=  -\int_{  \mathcal{S}_{0} } \operatorname{div} (x^\perp )\,dx=0 .
\end{gather*}
Above $e_1$ and $e_2$ are the unit vectors of the canonical basis of $\mathbb{R}^2$.

We have that for all $i=1,2,3$:
\begin{equation} \label{ComportementPhii}
\Phi_{i}(x) = {\mathcal O}\left(\frac{1}{|x|}\right) \text{ and }
\nabla \Phi_{i}(x) = {\mathcal O}\left(\frac{1}{|x|^{2}}\right) \text{ as } |x| \rightarrow +\infty,
\end{equation}
and consequently that $\nabla \Phi_i$ are in $L^2 (\mathcal{F}_{0})$. \par
For instance in the case where $\mathcal{S}_{0}$  is a disk one has 
$\Phi_{1}(x) = - \frac{ x_{1} }{|x|^{2} } $, $\Phi_{2}(x) = - \frac{ x_{2} }{|x|^{2} } $ and $\Phi_{3}(x) = 0$.
If  $\mathcal{S}_{0}$  is not a disk then these three functions are linearly independent; this can be easily be proved by using 
 a smooth arc length
parameterization of the boundary and the usual Frenet equations, 
 see for instance Lemma 6.1. of  \cite{ort1}. 
%
%
%
\subsubsection{Added inertia}
Let us define the matrices
\begin{equation}
 \label{AddedMass}
\mathcal{M}_a 
:= \begin{pmatrix} m_{i,j} \end{pmatrix}_{i,j \in \{1,2,3\}}
  \text{ and } {\mathcal M}_{a, \theta}  := \mathcal{R}(\theta) {\mathcal M}_{a} \mathcal{R}(\theta)^t ,
\end{equation}
where for $i,j \in \{ 1,2,3\}$
\begin{equation} \label{DefMasse}
m_{i,j}:= \int_{{\mathcal F}_{0}} \nabla \Phi_{i} \cdot \nabla \Phi_{j}  \,dx,
\end{equation}
and
$ \mathcal{R}(\theta) $ is the $3 \times 3$ rotation matrix defined in \eqref{3dR}.
Let us mention from now on that 
 the matrix ${\mathcal M}_{a, \theta} $
 is positive definite if and only $\mathcal{S}_{0}$ 
is not a disk. When $\mathcal{S}_{0}$  is a disk then  ${\mathcal M}_{a, \theta} $ has the form $diag (m_{a} , m_{a} , 0)$ with $m_{a} >0$. 
The case where $\mathcal S_0$ is  disk is therefore peculiar, indeed by combining the translation and rotation equations on observe that 
${\mathcal J} \theta'' = m h'' \cdot (h_{c}-h)^{\perp}$, where  $h_{c}$  denotes of the position  of the center of the disk ${\mathcal S}(q)$, which can be different from $h$ if the body is not homogeneous. 
As a consequence, in this case where ${\mathcal S}_{0}$ is a disk,  a particular reduction of the dynamics is possible and is indeed very helpful in order to tackle the case of massless particles. 
In the sequel we will focus on the case where ${\mathcal S}_{0}$ is not a disk, and we refer to \cite{GMS} for a full treatment of the case where ${\mathcal S}_{0}$ is a disk.

\subsubsection{Added inertia connection}
Let us define for $p = (\ell , r)^t$, 
\begin{eqnarray}
\label{cla2}
 \langle \Gamma_{a,\theta} ,p,p\rangle  &:=&
- \begin{pmatrix} 
 P_{a,\,  \theta} \\ 0
\end{pmatrix}
 \times p
-  r {\mathcal M}_{a,\,  \theta}
\begin{pmatrix}
 0  \\ \ell^\perp   \end{pmatrix} ,
\end{eqnarray}
where $P_{a,\,  \theta} $ are the two first coordinates of  ${\mathcal M}_{a,\,  \theta} \, p $.
A tedious computation reveals that $ \Gamma_{a,\theta}$ is the  a-connection associated with ${\mathcal M}_{a,\,  \theta}$.

\subsubsection{Harmonic field}
\label{HF}
To take the velocity circulation around the body into account, we introduce the following harmonic field: let $H$ be the unique solution vanishing at infinity of 
\begin{gather*}
\operatorname{div} H = 0    \text{ and }   \operatorname{curl} H = 0 \,    \text{ in }   \mathcal{F}_{0}, \, 
H \cdot n = 0 \,    \text{ on }  \partial \mathcal{S}_0, \, 
\int_{\partial \mathcal{S}_0 } H \cdot \tau \, ds = 1 .
\end{gather*}
The vector field $H$ admits a harmonic stream function $ \Psi_{H} (x)$:
\begin{equation*}
H = \nabla^{\perp} \Psi_{H},
\end{equation*}
which vanishes on the boundary $ \partial \mathcal{S}_0$, and behaves like $\frac{1}{2\pi} \ln |x |$ as $x$ goes to infinity. 
One way to get more information on the far-field behaviour of $H$ is to use a little bit of complex analysis. We identify $\mathbb{C}$ and $\mathbb{R}^{2}$ through 
\begin{equation}
\label{bonnet}
(x_{1},x_{2})= x_{1} + i x_{2}=z.
\end{equation}
We also use the notation 
\begin{equation}
\label{chapeau}
\widehat f = f_{1}-if_{2} \text{  for any }  f=(f_{1},f_{2}) .
\end{equation}
The reason of this notation is the following: if $f$ is divergence and curl free if and only if $\widehat f$ is holomorphic.
In particular the function $\widehat{H}$ is holomorphic (as a function of $z=x_{1}+ix_{2}$), and can be decomposed in Laurent Series :
\begin{equation} \label{HSeriesLaurent}
\widehat{H}(z) = \frac{1}{2i\pi z} + {\mathcal O}(1/z^{2}) \ \text{ as } z \rightarrow \infty.
\end{equation}
Coming back to the variable $x\in \mathbb{R}^2$, the previous decomposition implies
\begin{equation} \label{HalInfini}
H(x) = {\mathcal O}\left(\frac{1}{|x|}\right) \text{ and } \nabla H = {\mathcal O}\left(\frac{1}{|x|^{2}}\right).
\end{equation}
Moreover, we deduce from \eqref{HSeriesLaurent} that
\begin{equation} \nonumber 
x^{\perp} \cdot H = \frac{1}{2 \pi} + {\mathcal O}\left(\frac{1}{|x|}\right)
\ \text{ and } \ 
(H)^{\perp} - x^\perp  \cdot\nabla H = {\mathcal O}\left(\frac{1}{|x|^2}\right) .
\end{equation}
\subsubsection{Conformal center}
The harmonic field $H$ allows to define the following geometric constant, known as the conformal center of ${\mathcal S}_{0}$:
\begin{equation} \label{DefXi}
\xi_{1} + i \xi_{2}: = \int_{ \partial  \mathcal{S}_{0}}  z \widehat{H}  \, dz ,
\end{equation}
which depends only on $\mathcal{S}_{0}$. 
In the particular case of a disk, the harmonic field $H $ is given by  $ \frac{1}{2 \pi} \frac{x^{\perp}}{| x |^{2}} $
so that the  conformal center $\xi$ of  a disk is obviously $0$. 
In the general case one proves the following real-analytic characterization. 

\begin{prop} \label{B/L}
There holds:
\begin{equation} \label{eqB/L}
{\xi} =   \int_{\partial  \mathcal{S}_{0}  } \big(  H \cdot  \tau \big) x \, ds  .
\end{equation}
\end{prop}

In order to prove Proposition~\ref{B/L}
 we will use the following result which relates the integral 
  $\int_{ \mathcal{C}} \widehat f   \, dz $ 
associated with a  tangent vector field
 $f$ to its flux and its circulation.
\begin{lem}\label{lem-flux-circu}
Let  $\mathcal{C}$ be a smooth Jordan curve, $f:=(f_1 , f_2)$ a smooth vector fields on $\mathcal{C}$:
\[
\int_{ \mathcal{C}} \widehat f  \, dz =  \int_{ \mathcal{C}} f\cdot \tau \, ds  -i \int_{ \mathcal{C}} f\cdot n \, ds.
\]
\end{lem}
\begin{proof}[Proof of Lemma~\ref{lem-flux-circu}]
Denoting by $\gamma=(\gamma_{1},\gamma_{2})$ a arc-lenght parametrization of $\mathcal{C}$ then $\tau=(\gamma_{1}',\gamma_{2}')/|\gamma'|$,  $n=(-\gamma_{2}',\gamma_{1}')/|\gamma'|$, $ds = |\gamma'(t)|dt$ and $dz = (\gamma_{1}'(t)+i\gamma_{2}'(t)) dt$. Hence the conclusion follows from
\begin{eqnarray*}
\int_{ \mathcal{C}} (f_1 - if_2)  \, dz = \int (f_{1} \gamma_{1}' + f_{2} \gamma_{2}')\, dt -i  \int (-f_{1} \gamma_{2}' + f_{2} \gamma_{1}')\, dt .
\end{eqnarray*}
\end{proof}
\begin{proof}[Proof of Proposition~\ref{B/L}]
Observe that  $z (H_1 - i H_2 ) = f_1 - if_2$ with
$f_1 = x\cdot H$
and $f_2 =x^\perp \cdot H$ so that applying  Lemma~\ref{lem-flux-circu}  we have that 
\begin{equation*}
 \int_{ \partial  \mathcal{S}_{0}}  z \widehat{H}  \, dz    =  \int_{   \partial \mathcal S_0} 
g  \, ds ,   \text{ with } g   :=
\begin{pmatrix} x\cdot H  \\ x^\perp \cdot H  \end{pmatrix} \cdot \tau 
-i \begin{pmatrix} x\cdot H  \\ x^\perp \cdot H  \end{pmatrix} \cdot n .
\end{equation*}
Moreover, for  $x \in   \partial \mathcal S_0$,  we have
\begin{eqnarray*}
g
&=&  (x_{1}H_{1}+x_{2}H_{2})\tau_{1} +  (-x_{2}H_{1}+x_{1}H_{2})\tau_{2} 
\\ && \quad  -i(x_{1}H_{1}+x_{2}H_{2})n_{1} -i  (-x_{2}H_{1}+x_{1}H_{2})n_{2}\\
&=&  x_{1}  (H_{1}\tau_{1}+H_{2}\tau_{2})+x_{2}(H_{2}\tau_{1}-H_{1}\tau_{2})
\\ && \quad  -ix_{1}(H_{1}n_{1}+H_{2}n_{2})-ix_{2}(H_{2}n_{1}-H_{1}n_{2}),
\end{eqnarray*}
and using that 
$(n_{1},n_{2})=(-\tau_{2},\tau_{1})$, we deduce that
%
$g = z (H\cdot \tau) - iz (H\cdot n) .$
%
It is then sufficient to recall that
 $H\cdot n = 0$ to  conclude.
\end{proof}
\subsubsection{Kutta-Joukowski field}
Then the vector field $B_\theta$ is defined by the following formula:
\begin{equation} \label{DefB}
B_\theta :=  \mathcal{R}(\theta) \begin{pmatrix} \xi^{\perp} \\ -1 \end{pmatrix} .
\end{equation}

Observe that the corresponding force in the left hand side of \eqref{ODE_ext} 
is therefore
 \begin{eqnarray*}
q' \times B_{ \theta} 
 =
\begin{pmatrix}
(h')^\perp - \theta' {R}(\theta)  \xi^{\perp}  \\
{R}(\theta) {\xi} \cdot  h'
\end{pmatrix}.
\end{eqnarray*}
%

%
\subsection{Reformulation as an ODE in the body frame. Statement of Theorem \ref{reup}}
\label{Subsec:COF}
In order to transfer the equations in the body frame we apply the following isometric change of variable:
\begin{equation}  \label{chgtvar}
\left\{
\begin{array}{l}
 v (t,x)=R(\theta(t))^T\  u(t,R(\theta(t))x+h(t)), \\
 \tilde{\pi} (t,x)=\pi(t,R(\theta(t))x+h(t)), \\
 {\ell} (t)=R(\theta(t))^T \ h' (t) ,
\end{array}\right.
\end{equation}
where we recall that $R(\theta(t))$ is the $2 \times 2$ rotation matrix defined in \eqref{2drot}
so that the equations  \eqref{Euler1}-\eqref{Solideci}  become
\begin{gather}
\label{Euler11}
\displaystyle \frac{\partial v}{\partial t}
+ \left[(v-\ell-r x^\perp)\cdot\nabla\right]v 
+ r v^\perp+\nabla \tilde{\pi} =0 
\text{ and } \operatorname{div} v = 0
 \quad \text{ for } x\in \mathcal{F}_{0} ,\\
\label{Solide11}
m \ell'(t)=\int_{\partial \mathcal{S}_0} \tilde{\pi} n \ ds-m r \ell^\perp
\text{ and } 
 \mathcal{J} r'(t)=\int_{\partial \mathcal{S}_0} x^\perp \cdot \tilde{\pi} n \ ds ,   \\
\label{Euler13}
v\cdot n = \left(\ell +r x^\perp\right)\cdot n   \quad \text{ for }   x\in \partial \mathcal{S}_0, \\
\label{Solide1ci}
v(0,x)= v_0 (x)  \, \text{ for }x\in \mathcal{F}_{0}  \text{ and } (\ell,r)(0)= (\ell_0,r_0 ),
\end{gather}
 where $r(t) = \theta'(t)$.

In order to state recast the system above as an ODE in the body frame we are going to introduce now a few objects.
Let 
 $\Gamma_{g}: \mathbb{R}^{3} \times \mathbb{R}^{3} \rightarrow \mathbb{R}^{3}$  and $\Gamma_{a}: \mathbb{R}^{3} \times \mathbb{R}^{3} \rightarrow \mathbb{R}^{3}$ be the bilinear symmetric mappings defined, for all ${p} = ( \ell \\ r ) \in \mathbb{R}^{3}$, by
\begin{equation} \label{DefGammag}
  \langle \Gamma_{g}  , {p}, p \rangle = m r \begin{pmatrix} \ell^{\perp} \\ 0 \end{pmatrix}
  \text{ and  }
 \langle  \Gamma_a , p, p \rangle = 
   \begin{pmatrix} r ({\mathcal M}_{\flat} \ell)^{\perp} \\ \ell^{\perp} \cdot {\mathcal M}_{\flat} \ell \end{pmatrix} 
    +  r p \times \mu ,
\end{equation}
where
\begin{equation} \label{VectMass}
{\mu} := \begin{pmatrix} m_{1,3} \\ m_{2,3}  \\ 0 \end{pmatrix}  \text{ and }
\mathcal{M}_{\flat}
:=\begin{pmatrix} m_{i,j} \end{pmatrix}_{i,j \in \{1,2\}} .
\end{equation}
Note that
\begin{equation} \label{AnnulationGamma}
\forall {p} \in \mathbb{R}^{3}, \ \ \ \langle \Gamma_{g}, {p}, p \rangle \cdot p =0    \text{  and } \langle \Gamma_{a}, {p}, p \rangle \cdot p =0 ,
\end{equation}
and that $\Gamma_a$ is depending only on ${\mathcal S}_{0}$.
Because of \eqref{AnnulationGamma}  we will refer to the quadratic mappings $\Gamma_{g}$ and $\Gamma_{a}$  as {\it gyroscopic terms}.

One will deduce   Theorem  \ref{pasdenom} from the following result  by going back in the original frame.

\begin{thm} \label{reup}
The equations  \eqref{Euler11}-\eqref{Solide11}  are equivalent to the following ODE for $p :=(\ell , r )^T$:
\begin{equation}
\label{eq}
\big[  {\mathcal M}_{g} +  {\mathcal M}_{a} \big]
  p' 
 + \langle  \Gamma_g ,  p ,  p \rangle
 +\langle  \Gamma_a ,  p ,  p \rangle
=  \gamma p \times B ,
\end{equation}
the fluid velocity $v$ being given as the unique solution to the following div/curl type system:
\begin{gather}
\label{r1}
 \operatorname{div} v = 0     \text{ and } 
 \operatorname{curl} v  =  0  \text{ in }   \mathcal{F}_{0}  , \, 
 \\ \label{r2}
 v \cdot n = \left(\ell+r x^\perp\right)\cdot n     \text{ on }  \partial \mathcal{S}_0  , \,    \quad   \int_{ \partial \mathcal{S}_0} v  \cdot  \tau \, ds=  \gamma ,
 \\ \label{r3}
\lim_{x\to \infty} v = 0 .
 \end{gather}  
\end{thm}
Let us recall that ${\mathcal M}_{g} $ and $  {\mathcal M}_{a}$ are respectively  the genuine and  added inertia,  see \eqref{AddedMass}  and  \eqref{DefMG},
and that   $ B$ is a fixed vector defined in \eqref{DefB}.

Observe that one may also obtain from this formulation the 
conservative feature of the system since  it suffices to multiply \eqref{eq} by $p$,  to use the symmetry of the matrices 
 ${\mathcal M}_{g} $ and $  {\mathcal M}_{a}$ 
  and the properties 
  \eqref{AnnulationGamma}  
 to deduce  that the total energy 
$ \frac12 p \cdot \big({\mathcal M}_{g} + {\mathcal M}_{a} \big) p$  is conserved along time.

 The rest of the section is devoted to the proof of Theorem  \ref{reup}. 
 Indeed after a slight reformulation of the solid equations and the decomposition of the velocity into several pieces corresponding to the various sources in the right hand sides of the system \eqref{r1}-\eqref{r2}-\eqref{r3}, we will compare two methods, one based on complex argument and the other one on real analysis only.

\subsubsection{Reformulation of the solid equations}
The first step of the  proof of Theorem  \ref{reup} uses the Euler equations and the Kirchhoff potentials in order to get rid of the pressure and to make appear the added inertia.
\begin{lem} \label{Refo}
Equations \eqref{Solide11} can be rewritten in the form
\begin{equation}
\label{anoter}
({\mathcal M}_{g} + {\mathcal M}_{a} ) p ' 
 + \langle  \Gamma_g ,  p ,  p \rangle
= - ( \frac{1}{2} \int_{\partial  \mathcal{S}_{0}  } |{v}|^2  K_i \, ds
- \int_{\partial  \mathcal{S}_{0}  } (\ell + r x^\perp)\cdot {v}  K_i \, ds)_{i} ,
\end{equation}
where $i$ runs over the integers $1,2,3$.
\end{lem}
\begin{proof}
Using the following equality for two vector fields $a$ and $b$ in $\mathbb{R}^2$:
\begin{equation} \label{vect}
\nabla(a\cdot b)=a\cdot \nabla b + b \cdot \nabla a - (a^\perp \operatorname{curl} b + b^\perp \operatorname{curl} a), 
\end{equation}
the equation \eqref{Euler11} can be written as
\begin{equation}
\label{looser}
\frac{\partial v}{\partial t}
+  \frac{1}{2} \nabla (v^2 )
- \nabla ( (\ell + r x^\perp)\cdot v )
+ \nabla \tilde{\pi} =0 .
\end{equation}
We use this equation do deduce the force/torque acting on the body:
\begin{equation*}
\left( \int_{ \partial \mathcal{S}_0} \tilde{\pi} n \, ds, \
\int_{ \partial \mathcal{S}_0} \tilde{\pi} x^{\perp} \cdot n \,ds \right) = 
\left( \int_{ \mathcal{F}_0} \nabla \tilde{\pi} \cdot \nabla \Phi_{i} \, dx \right)_{i=1,2,3} .
\end{equation*}
One can check that the above integration by parts is licit by using  the decay properties of $v$ and $\nabla \Phi_{i}$. 
Using an integration by parts, the boundary condition \eqref{Euler13} and another integration by parts, one observes that 
the contribution of $\partial_{t}  v $ is
\begin{equation}
\label{vieux}
\int_{  \mathcal{F}_{0}  }  \partial_{t}  v  \cdot  \nabla   \Phi_i (x) \, dx = {\mathcal M}_{a} \begin{pmatrix} \ell \\ r \end{pmatrix}', 
\end{equation}
and one obtains the result.
\end{proof}
%
%
%

%
 %
%
%
\subsubsection{Decomposition of the velocity field}
\label{sec-decompo}
Finally, for $\ell$ in $\mathbb{R}^2$, $r$ and $\gamma $ in  $\mathbb{R}$ given, there exists a unique vector field $v$ verifying \eqref{r1}-\eqref{r2}-\eqref{r3} 
and it is given by the  law:
\begin{equation}
\label{vdecomp}
v =  \gamma   H + \ell_1 \nabla \Phi_1 + \ell_2 \nabla \Phi_2
+ r \nabla \Phi_3 ,
\end{equation}
We will denote by  
\begin{equation} \label{allo1}
\tilde{v} : = v- \gamma H .
\end{equation}
the part without circulation, 
that we will decomposed sometimes into
\begin{equation}
 \label{diese99}
\tilde{v} = v_{\#} + r \nabla \Phi_{3}  , 
\end{equation}
with
\begin{equation}
 \label{diese}
 v_{\#}  :=  \ell_1 \nabla \Phi_1 + \ell_2 \nabla \Phi_2 , 
\end{equation}
in order to analyse separately the effects of the body translation and of the body rotation.

Observe that in the particular case  the fluid velocity is assumed to be globally a gradient (the so-called potential case corresponding to $\gamma =0$)
it may be expanded with respect to the Kirchhoff  potentials only.

Another crucial observation is that the first term in the right hand side of  \eqref{anoter} is quadratic in $v$ and that $v$ is decomposed into a potential part and a circulatory part, cf. \eqref{vdecomp}. Roughly speaking the a-connection (the last term in the right hand side of \eqref{eq}) will result from the quadratic self-interaction of the potential part and the Kutta-Joukowski term (the left hand side of \eqref{eq}) from the crossed interaction between the potential part and the circulatory part.  There will be a cancellation of the  quadratic self-interaction of the  circulatory part, cf. \eqref{T7} and Lemma \ref{lamb}.
Indeed this cancellation echoes the renormalization hinted in Remark \ref{remarep}. The kinetic energy of the fluid should be 
$\frac12  \int_{ \mathcal{F}_{0}  } v^2 \ dx $ but  the term $ \frac12  ( {\mathcal M}_{a, \theta}  \, q' )\cdot q' $ included in \eqref{NR} is equal to 
$\frac12  \int_{ \mathcal{F}_{0}  } \tilde{v}^2 \ dx $. Since $\tilde{v}$ and $H$ are orthogonal in $L^2 (\mathcal{F}_{0})$  this renormalization therefore  formally  consists in discarding
 the contribution to the fluid kinetic energy due to the fluid velocity associated with the circulation around the body. Observe that this contribution is infinite because of the slow decay at infinity of this part of the fluid velocity but depends only on the body's shape, not on its position or velocity.

\subsection{A complex-analytic  proof of Theorem \ref{reup}}

We will follow here a strategy based on complex analysis  after Blasius, Kutta, Joukowski, Chaplygin and Sedov.
A key lemma is the following  Blasius formula about tangent vector fields where we use the identifications \eqref{bonnet} and \eqref{chapeau}. 
\begin{lem}
\label{blasius}
Let  $\mathcal{C}$ be a smooth Jordan curve, $f:=(f_1 , f_2)$ and $g:=(g_1 ,g_2 )$ two smooth tangent vector fields on $\mathcal{C}$. 
Then 
\begin{gather} 
\label{bla1}
\int_{ \mathcal{C}} (f  \cdot g) n \, ds =  i \left( \int_{ \mathcal{C}} \widehat f \widehat g \, dz \right)^* , \\
\label{bla2}
\int_{ \mathcal{C}} (f  \cdot g) (x^{\perp} \cdot n)  \,ds =  {\rm Re} \left( \int_{ \mathcal{C}} z  \widehat f \widehat g \, dz \right).
\end{gather}
\end{lem}
Above$(\cdot )^*$ denotes the complex conjugation.
\begin{proof}[Proof of Lemma \ref{blasius}]
Thanks to the polarization identity, it is sufficient to consider the case where $f=g$. Let us consider $\gamma= (\gamma_{1},\gamma_{2}): [0,1] \rightarrow \mathbb{R}^{2}$ a smooth 
arc length parameterization of the Jordan curve ${\mathcal C}$. On one side, one has
\begin{equation} \label{cotereel}
\int_{{\mathcal C}} (f \cdot f) n \, ds = \int_{0}^{1} \big( f_{1}(\gamma(t))^{2} + f_{2}(\gamma(t))^{2} \big) \begin{pmatrix} -{\gamma}_{2}'(t) \\ {\gamma}_{1} ' (t)\end{pmatrix} \, dt.
\end{equation}
On the other side, one has
\begin{gather*}
\int_{{\mathcal C}} (f_{1}(z) - i f_{2}(z))^{2} \, dz
 \\  = \int_{0}^1 \Big( f_{1}(\gamma(t)) - i f_{2}(\gamma(t)) \Big)
\Big[ \big( f_{1}(\gamma(t)) - i f_{2}(\gamma(t)) \big) ({\gamma}_{1}'(t) + i {\gamma}_{2}'(t)) \Big]  \, dt .
\end{gather*}
But since $f$ is tangent to ${\mathcal C}$, one sees that the expression inside the brackets  above  is real, and hence is equal to its complex conjugate. It follows that
\begin{equation*}
\int_{{\mathcal C}} (f_{1}(z) - i f_{2}(z))^{2} \, dz
= \int_{0}^1  \big| f_{1}(\gamma(t)) - i f_{2}(\gamma(t)) \big|^{2} ({\gamma}_{1}'(t) - i {\gamma}_{2}'(t))  \, dt ,
\end{equation*}
and \eqref{bla1} follows. \par
The proof of \eqref{bla2} is analogous: using again 
\begin{equation*}
(f_{1} (\gamma(t)) - i f_{2} (\gamma(t)))( {\gamma}_{1}'(t) + i {\gamma}_{2}'(t))
=(f_{1} (\gamma(t)) + i f_{2} (\gamma(t)))({\gamma}_{1}'(t) - i {\gamma}_{2}'(t)),
\end{equation*}
we deduce
\begin{gather*}
\int_{{\mathcal C}} (f_{1}(z) - i f_{2}(z))^{2} z \, dz
 \\ = \int_{0}^1  |f_{1}(\gamma(t)) - i f_{2}(\gamma(t))|^{2} (\gamma_{1}(t) + i \gamma_{2}(t)) ({\gamma}_{1}'(t) - i {\gamma}_{2}'(t))  \, dt ,
\end{gather*}
so that
\begin{eqnarray*}
&& {\rm Re} \left(\int_{{\mathcal C}} (f_{1}(z) - i f_{2}(z))^{2} z \, dz\right) \\
\quad &=& \int_{0}^1  \Big( f_{1}(\gamma(t))^{2} + f_{2}(\gamma(t))^{2} \Big) \big( \gamma_{1}(t) {\gamma}_{1}'(t) +  \gamma_{2}(t) {\gamma}_{2}'(t) \big)  \, dt \\
\quad &=& \int_{ \mathcal{C}} (f  \cdot f) (x^{\perp} \cdot n)  \,ds .
\end{eqnarray*}
\end{proof}
The idea of the complex-analytic approach of the computation of  
the terms in the right hand side of the equation in 
Lemma \ref{Refo}
 is to decompose  them in order to make appear some vector fields  tangent  to the boundary $\partial  \mathcal{S}_{0} $, to use Blasius' lemma and then Cauchy's residue theorem.
\ \par
\par
\ \par
In this process of computing the residue we will use the Laurent series of $\widehat{\nabla \Phi_{i}}$.
Because of the decay property at infinity in \eqref{def Phi} the Laurent series of $\widehat{\nabla \Phi_{i}}$ has to start at least with a term in $ \mathcal{O}(1/z)$, and this term is 
  $$ \frac{ 1 }{ 2 i \pi}  \int_{ \partial  \mathcal{S}_{0} } \widehat{\nabla \Phi_{i}} \, dz  .$$ 

  Thanks to Lemma \ref{lem-flux-circu}
   we get that $\int_{ \partial  \mathcal{S}_{0} } \widehat{\nabla \Phi_{i}} \, dz  = 0$
   since the circulation of a gradient around $\partial  \mathcal{S}_{0}$ vanishes and the flux as well according to \eqref{Stokes0}.

Moreover Lemma \ref{lem-flux-circu}
 also allows to compute the second term in the Laurent series:
\begin{cor}\label{cor}
Let  $\mathcal{C}$ be a smooth Jordan curve, $f:=(f_1 , f_2)$ a smooth vector fields on $\mathcal{C}$:
\begin{eqnarray*}
\int_{ \mathcal{C}} z\widehat f  \, dz &=&  \int_{ \mathcal{C}} \begin{pmatrix} x\cdot f \\ x^\perp \cdot f \end{pmatrix} \cdot \tau \, ds  -i \int_{ \mathcal{C}} \begin{pmatrix} x\cdot f \\ x^\perp \cdot f \end{pmatrix} \cdot n \, ds\\
&=& \int_{ \mathcal{C}} (x_{1} +i x_{2}) (f\cdot \tau)  \, ds -i \int_{ \mathcal{C}} (x_{1} +i x_{2}) (f\cdot n)  \, ds.
\end{eqnarray*}
\end{cor}
\begin{proof}
To apply the previous lemma, we have to identify a function $g$ such that $z(f_1 - if_2)=g_{1}-ig_{2}$. Hence, to get the first equality, it is sufficient to check that
\[
(x_{1}+ix_{2})(f_{1}-if_{2})= (x_{1} f_{1} + x_{2} f_{2})-i(-x_{2}f_{1}+x_{1}f_{2})= (x\cdot f) -i(x^\perp \cdot f).
\]
To obtain the second equality, we simply use $(n_{1},n_{2})=(-\tau_{2},\tau_{1})$:
\begin{eqnarray*}
\begin{pmatrix} x\cdot f \\ x^\perp \cdot f \end{pmatrix} \cdot \tau  -i \begin{pmatrix} x\cdot f \\ x^\perp \cdot f \end{pmatrix} \cdot n &=&  (x_{1}f_{1}+x_{2}f_{2})\tau_{1} +  (-x_{2}f_{1}+x_{1}f_{2})\tau_{2}
\\ &&   -i(x_{1}f_{1}+x_{2}f_{2})n_{1} -i  (-x_{2}f_{1}+x_{1}f_{2})n_{2}\\
&=&  x_{1}  (f_{1}\tau_{1}+f_{2}\tau_{2})+x_{2}(f_{2}\tau_{1}-f_{1}\tau_{2})
\\ &&   -ix_{1}(f_{1}n_{1}+f_{2}n_{2})-ix_{2}(f_{2}n_{1}-f_{1}n_{2})\\
&=&  x_{1}  (f_{1}\tau_{1}+f_{2}\tau_{2})+x_{2}(f_{2}n_{2}+f_{1}n_{1}) 
\\ &&  -ix_{1}(f_{1}n_{1}+f_{2}n_{2})-ix_{2}(-f_{2}\tau_{2}-f_{1}\tau_{1})\\
&=& (x_{1}+ix_{2})  (f_{1}\tau_{1}+f_{2}\tau_{2})
\\ &&   -i(x_{1}+ix_{2})(f_{1}n_{1}+f_{2}n_{2})
\end{eqnarray*}
which ends the proof.
\end{proof}
Replacing $x_{2}$ by $-x_{2}$ in the previous proof, we note that we obtain
\begin{equation}
\label{Rembarz}
\int_{ \mathcal{C}} \bar{z} \widehat f  \, dz
=\int_{ \mathcal{C}} (x_{1} -i x_{2}) (f\cdot \tau)  \, ds -i \int_{ \mathcal{C}} (x_{1} -i x_{2}) (f\cdot n)  \, ds.
\end{equation}
We apply the previous results to the function $\nabla \Phi_{i}$:
\begin{lem}\label{lemzphi}
One has:
\begin{eqnarray*}
\int_{ \partial  \mathcal{S}_{0}} z\widehat{\nabla \Phi_{i}}  \, dz &=&-(m_{i,2}+|\mathcal{S}_{0} |\delta_{i,2})+i(m_{i,1}+|\mathcal{S}_{0} |\delta_{i,1}),\quad \text{for }i=1,2;\\
\int_{ \partial  \mathcal{S}_{0}} z\widehat{ \nabla\Phi_{3}}   \, dz &=&-(m_{3,2}+|\mathcal{S}_{0} |x_{G,1})+i(m_{3,1}-|\mathcal{S}_{0} |x_{G,2});\\
\end{eqnarray*}
where $m_{i,j}$ is defined in \eqref{DefMasse}.
\end{lem}
\begin{proof}
We use the previous corollary with $f=\nabla \Phi_{i}$:
\begin{eqnarray*}
\int_{ \partial  \mathcal{S}_{0}} z\widehat{ \nabla\Phi_{i}}   \, dz
&=& \int_{ \partial  \mathcal{S}_{0}} (x_{1} +i x_{2}) \partial_{\tau}\Phi_{i}  \, ds 
-i \int_{ \partial  \mathcal{S}_{0}} (x_{1} +i x_{2}) \partial_{n}\Phi_{i}  \, ds.
\end{eqnarray*}
We can integrate by part in the first integral:
\begin{eqnarray*}
\int_{ \partial  \mathcal{S}_{0}} (x_{1} +i x_{2}) \partial_{\tau}\Phi_{i}  \, ds
&=&  -\int_{ \partial  \mathcal{S}_{0}} \partial_{\tau}(x_{1} +i x_{2}) \Phi_{i}  \, ds\\
&=&  -\int_{ \partial  \mathcal{S}_{0}}(\tau_{1} +i \tau_{2}) \Phi_{i}  \, ds\\
&=& -\int_{ \partial  \mathcal{S}_{0}}(n_{2} -i n_{1}) \Phi_{i}  \, ds\\
&=& -\int_{ \mathcal{F}_{0}}\nabla\Phi_{2} \cdot  \nabla  \Phi_{i}   \, ds
+i \int_{ \mathcal{F}_{0}}\nabla\Phi_{1} \cdot  \nabla  \Phi_{i}  \, ds \\
&=& -m_{i,2} +i m_{i,1} .
\end{eqnarray*}
The second integral can be computed thanks to the boundary condition and \eqref{Stokes1}:
\begin{equation*}\begin{split}
\int_{ \partial  \mathcal{S}_{0}} x_{j}  \partial_{n}\Phi_{i}  \, ds  
= \int_{ \partial  \mathcal{S}_{0}} x_{j} n_{i}  \, ds  
= -\delta_{i,j} |\mathcal{S}_{0}| \quad &\text{for }i,j=1,2,\\
\int_{ \partial  \mathcal{S}_{0}} x_{1}  \partial_{n}\Phi_{3}  \, ds  
= \int_{ \partial  \mathcal{S}_{0}} x_{1} (x^\perp \cdot n)  \, ds  
=  |\mathcal{S}_{0}| x_{G,2} 
& \text{ and } \\
\int_{ \partial  \mathcal{S}_{0}} x_{2}  \partial_{n}\Phi_{3}  \, ds  
= \int_{ \partial  \mathcal{S}_{0}} x_{2} (x^\perp \cdot n)  \, ds
=  - |\mathcal{S}_{0}| x_{G,1} .
\end{split}\end{equation*}
\end{proof}
Using  \eqref{Rembarz}, we can reproduce exactly the previous proof to establish that:
\begin{eqnarray*}
\int_{ \partial  \mathcal{S}_{0}} \bar{z}\widehat{\nabla \Phi_{i}}  \, dz &=&(-m_{i,2}+|\mathcal{S}_{0} |\delta_{i,2})+i(-m_{i,1}+|\mathcal{S}_{0} |\delta_{i,1}),\quad \text{for }i=1,2;\\
\int_{ \partial  \mathcal{S}_{0}} \bar{z}\widehat{\nabla \Phi_{3}}   \, dz &=&(-m_{3,2}+|\mathcal{S}_{0} |x_{G,1})-i(m_{3,1}+|\mathcal{S}_{0} |x_{G,2}) .
\end{eqnarray*}

We are now almost ready to start the proof by itself. 
The last preparation consists in using  the decomposition \eqref{allo1} to deduce from 
Lemma \ref{Refo} that Equations \eqref{Solide11} can be rewritten in the form
\begin{equation} \label{clacla}
({\mathcal M}_{g} + {\mathcal M}_{a} ) p ' 
 + \langle  \Gamma_g ,  p ,  p \rangle 
=  - (A_{i}+ B_{i} + C_i)_{i=1,2,3} ,
\end{equation}
where for $i=1,2,3$, 
\begin{eqnarray}
\nonumber
A_{i} &:=& \frac{1}{2} \int_{\partial  \mathcal{S}_{0}  } |\tilde{v}|^2  K_i \, ds
- \int_{\partial  \mathcal{S}_{0}  } (\ell + r x^\perp)\cdot \tilde{v}  K_i \, ds, \\ 
\label{Cib}
B_{i} &:=& \gamma  \int_{\partial  \mathcal{S}_{0}  } (\tilde{v} - (\ell + r x^\perp))\cdot  H  K_i \, ds, \\
\label{Cic}
C_{i} &:=& \frac{\gamma^2}{2}   \int_{\partial  \mathcal{S}_{0} } |H|^2  K_i \, ds.
\end{eqnarray}

We start with the computation of the terms $(A_{i})_{i=1,2,3}$.

\begin{lem} \label{PropCia}
\begin{equation} \label{Ciad}
\begin{split}
 \begin{pmatrix}A_{1} \\A_{2} \end{pmatrix} 
& = r^2  \begin{pmatrix}-m_{3,2}\\m_{3,1}\end{pmatrix} 
 + r\Big(\mathcal{M}_{\flat} \ell    \Big)^\perp
\end{split}
\end{equation}
and
\begin{equation} \label{C3ad}
\begin{split}
A_{3} 
 = \ell^\perp \mathcal{M}_{\flat} \ell 
- r\ell \cdot \begin{pmatrix} -m_{3,2} \\m_{3,1} \end{pmatrix}  .
 \end{split}
\end{equation}
\end{lem}

\begin{proof}
We start with the following observation:
\begin{eqnarray*}
A_{i} &=&\frac{1}{2} \int_{\partial  \mathcal{S}_{0}  } |\tilde{v}-(\ell + r x^\perp)|^2  K_i \, ds - \frac{1}{2} \int_{\partial  \mathcal{S}_{0}  } |\ell + r x^\perp|^2  K_i \, ds ,
\end{eqnarray*}
which makes appear at least one term with the wished tangence property. 
Since Blasius' lemma is different for the torque, we
 replace 
 $\tilde{v} $ by the decomposition \eqref{diese99} to get
\begin{eqnarray*}
A_{i} 
&=&  \frac{1}{2} \int_{\partial  \mathcal{S}_{0}  } |v_{\#} -\ell |^2  K_i \, ds + \frac{1}{2} \int_{\partial  \mathcal{S}_{0}  } | r(\nabla \Phi_{3} -x^\perp) |^2  K_i \, ds
\\ &&  + \int_{\partial  \mathcal{S}_{0}  }  r(v_{\#}-\ell) \cdot (\nabla \Phi_{3} -x^\perp)  K_i \, ds
 - \frac{1}{2} \int_{\partial  \mathcal{S}_{0}  } |\ell + r x^\perp|^2  K_i \, ds\\
 &=:&A_{i,a}+A_{i,b}+A_{i,c}+A_{i,d}.
\end{eqnarray*}
The first three terms have a form appropriated for the strategy mentioned above. 
One may worry above the last one but it benefits from a special structure resembling the Archimedes' force (despite the absence of gravity in our setting). 
Let us see first how this term can be simply computed thanks to the Stokes formula so that we will then be serene to implement the complex-analytic approach to the three other terms.

\textbf{An Archimedes-type term}. In order to compute the  term $A_{i,d}$ we first expand 
\begin{eqnarray*}
\int_{\partial  \mathcal{S}_{0}  } |\ell + r x^\perp|^2  K_i \, ds 
&=& |\ell|^2  \int_{\partial  \mathcal{S}_{0}  } K_i \, ds 
- 2 \ell_1 r \int_{\partial  \mathcal{S}_{0}  } x_2  K_i \, ds 
\\ &&  +2 \ell_2 r \int_{\partial  \mathcal{S}_{0}  } x_1  K_i \, ds 
+ r^2  \int_{\partial  \mathcal{S}_{0}  } |x|^2 K_i \, ds .
\end{eqnarray*}
Thanks to the Stokes formula:
\begin{subequations}
\label{Stokes1}
\begin{alignat}{3}
\int_{ \partial  \mathcal{S}_{0}} x_{j} n_{i}  \, ds
&=& -\int_{ \mathcal{S}_{0}} \operatorname{div}(x_{j} e_{i})  \, dx  =-\delta_{i,j} |\mathcal{S}_{0} |,\quad \text{for }i,j=1,2 ;\\
\int_{ \partial  \mathcal{S}_{0}} x_{1} (x^\perp \cdot n)  \, ds 
&=& -\int_{ \mathcal{S}_{0}} \operatorname{div}(x_{1} x^\perp)  \, dx 
= -\int_{ \mathcal{S}_{0}} (-x_2)  \, dx =  |\mathcal{S}_{0} |x_{G,2};\\
\int_{ \partial  \mathcal{S}_{0}} x_{2} (x^\perp \cdot n)  \, ds
&=& -\int_{ \mathcal{S}_{0}} \operatorname{div}(x_{2} x^\perp)  \, dx 
= -\int_{ \mathcal{S}_{0}} x_1  \, dx = - |\mathcal{S}_{0} |x_{G,1} ;
\end{alignat}
\end{subequations}
and
\begin{subequations}
\label{Stokes2}
\begin{alignat}{1}
\nonumber
\int_{\partial  \mathcal{S}_{0}  }|x|^2 n_{i}  \, ds 
= -\int_{  \mathcal{S}_{0}  } \operatorname{div}(|x|^2e_{i}) \, dx
= -\int_{   \mathcal{S}_{0}} 2x_{i}  \, dx
\\ \qquad =-2x_{G,i} |\mathcal{S}_{0} |,  \quad \text{for }i=1,2;\\
\int_{\partial  \mathcal{S}_{0}  }|x|^2x^\perp\cdot n  \, ds
= -\int_{  \mathcal{S}_{0}  } \operatorname{div}(|x|^2x^\perp) \, dx=0 ,
\end{alignat}
\end{subequations}
where $|\mathcal{S}_{0} |$ is the Lebesgue measure of $\mathcal{S}_{0}$ and $x_{G} = (x_{G,1} , x_{G,2} )$ is the position of the geometrical center of $\mathcal{S}_{0}$ (which can be different of the mass center $0$ if the solid is not homogenous):
\begin{equation} \label{DefXg}
x_{G} := \frac{1}{|{\mathcal S}_{0}|} \int_{{\mathcal S}_{0}} x \, dx .
\end{equation}
\par

Then, using also \eqref{Stokes0}, we check easily that
\begin{equation}
  \label{T1}
\Big( A_{i,d}\Big)_{i=1,2} 
=-r \ell^\perp |\mathcal{S}_{0} | + r^2x_{G} |\mathcal{S}_{0} |
  \text{ and } 
A_{3,d} 
=-r (\ell\cdot x_{G}) |\mathcal{S}_{0} |.
\end{equation}

\textbf{Computation of the three other terms}. Recalling the definition 
 \eqref{diese} and using the notation 
\eqref{chapeau} 
 we first remark that
\begin{equation}\label{v-l}
\widehat{v_{\#} -\ell}(z)= -\ell_{1}  + i \ell_{2} +\ell_{1}  \widehat{\nabla \Phi_{1}}+ \ell_{2} \widehat{\nabla \Phi_{2}}  .
\end{equation}
\ \par
\noindent
\textbf{Computation of $A_{i,a} $}.
We compute separately the case where $i=1,2$ and the case where $i=3$.
 \ \par \
   $\bullet$  As $v_{\#}-\ell$ is tangent to the boundary, we can apply the Blasius formula (see Lemma \ref{blasius}), \eqref{v-l}, Cauchy's residue theorem, to obtain
\begin{equation}
  \label{T2}
\Big(A_{i,a}\Big)_{i=1,2} = 0 .
\end{equation}
$\bullet$ We proceed in the same way for $i=3$:
\begin{equation*}\begin{split}
A_{3,a}=\frac12\int_{\partial  \mathcal{S}_{0}  } |v_{\#} -\ell |^2  K_3 \, ds=\frac12{\rm Re} \left( \int_{ \partial  \mathcal{S}_{0}} z (\widehat{v_{\#} -\ell})^2 \, dz \right) 
\\ = {\rm Re} \Bigg( \Big[(-\ell_{1})-i(-\ell_{2})\Big]
\int_{ \partial  \mathcal{S}_{0}} z\Big( \ell_{1}  \widehat{\nabla \Phi_{1}}+ \ell_{2}  \widehat{\nabla \Phi_{2}}  \Big) \, dz 
\Bigg) .
\end{split}\end{equation*}
so that, thanks to Lemma \ref{lemzphi},
\begin{equation*}\begin{split}
A_{3,a}
=&(-\ell_{1})\Big[
-\ell_{1} m_{1,2}-\ell_{2} (m_{2,2}+|\mathcal{S}_{0} |) \Big]
+(-\ell_{2})\Big[
\ell_{1} (m_{1,1}+|\mathcal{S}_{0}|) +\ell_{2} m_{2,1} \Big]
\end{split}\end{equation*}
which finally can be simplified as follows:
\begin{equation}
  \label{T3}
A_{3,a}
= \ell^\perp \mathcal{M}_{\flat} \ell .
\end{equation}
\par
\ \par
\textbf{Computation of $A_{i,b} $}.
Once again we distinguish the case where $i=1,2$ and the case where $i=3$.
 \ \par \
$\bullet$  As $\nabla \Phi_{3} -x^\perp$  is tangent to the boundary, we can write for $i=1,2$ by Lemma \ref{blasius} and by the Cauchy's residue theorem:
\begin{align*}
\Big(A_{i,b}\Big)_{i=1,2} 
= \frac{r^2}{2} \int_{\partial  \mathcal{S}_{0}  } | \nabla \Phi_{3} -x^\perp |^2  n \, ds
&=  \frac{i r^2}2 \left( \int_{ \partial  \mathcal{S}_{0}} (\widehat{\nabla \Phi_{3} -x^\perp})^2 \, dz \right)^*
\\ &= \frac{i r^2}2 \left( \int_{ \partial  \mathcal{S}_{0}} 2i\bar{z} \widehat{\nabla \Phi_{3} } \, dz 
- \int_{ \partial  \mathcal{S}_{0}} \bar{z}^2  \, dz\right)^*,
\end{align*}
where we have noted that $-\widehat{x^\perp}=i\bar{z}$. 
Let us observe that 
\begin{eqnarray*}
\nonumber
\int_{ \partial  \mathcal{S}_{0}} \bar{z}^2  \, dz
&=& \int (\gamma_{1}-i\gamma_{2})^2(\gamma_{1}'+i\gamma_{2}') \\
\nonumber
&=& \int(\gamma_{1}^2\gamma_{1}'-\gamma_{2}^2\gamma_{1}'+2\gamma_{1}\gamma_{2}\gamma_{2}')
+ i \int(\gamma_{1}^2\gamma_{2}'-\gamma_{2}^2\gamma_{2}'-2\gamma_{1}\gamma_{2}\gamma_{1}') \\
\nonumber
&=& - \int_{ \partial  \mathcal{S}_{0}} \begin{pmatrix} 2x_{1}x_{2}\\x_{2}^2\end{pmatrix}\cdot n \,ds
- i \int_{ \partial  \mathcal{S}_{0}} \begin{pmatrix} x_{1}^2\\2x_{1}x_{2}\end{pmatrix}\cdot n \,ds \\
\nonumber
&=& \int_{   \mathcal{S}_{0}} \operatorname{div}\begin{pmatrix} 2x_{1}x_{2} \\ x_{2}^2\end{pmatrix} \,dx 
+ i \int_{   \mathcal{S}_{0}} \operatorname{div}\begin{pmatrix} x_{1}^2\\2x_{1}x_{2}\end{pmatrix} \,dx \\
& = & 4 |\mathcal{S}_{0} |x_{G,2}+4i|\mathcal{S}_{0} |x_{G,1}.
\end{eqnarray*}

Therefore thanks to Lemma \ref{lemzphi}  we state that:
\begin{eqnarray*}
\Big(A_{i,b}\Big)_{i=1,2}
&=&i r^2\Big(   i(-m_{3,2}+|\mathcal{S}_{0} |x_{G,1})+(m_{3,1}+|\mathcal{S}_{0} |x_{G,2}) 
 \\ &&-2( |\mathcal{S}_{0} |x_{G,2}+i|\mathcal{S}_{0} |x_{G,1} )\Big)^*\\
&=&r^2\Big( (-m_{3,2}-|\mathcal{S}_{0} |x_{G,1})+i(m_{3,1}-|\mathcal{S}_{0} |x_{G,2})\Big)\\
&=& r^2 \Big( \begin{pmatrix}-m_{3,2}\\m_{3,1}\end{pmatrix}- |\mathcal{S}_{0} |x_{G}\Big),
\end{eqnarray*}
and thus
\begin{equation}
  \label{T4}
\Big(A_{i,b}\Big)_{i=1,2} = 0 .
\end{equation}

$\bullet$  For $i=3$, we have that:
\begin{equation*}\begin{split}
A_{3,b}&=\int_{\partial  \mathcal{S}_{0}  } |\nabla \Phi_{3} -x^\perp |^2  (  x^\perp \cdot n) \, ds
 \\ &= \int_{\partial  \mathcal{S}_{0}  } |\nabla \Phi_{3} |^2  x^\perp \cdot n \, ds - 2\int_{\partial  \mathcal{S}_{0}  } (\nabla \Phi_{3} \cdot x^\perp  )(  x^\perp \cdot n) \, ds +\int_{\partial  \mathcal{S}_{0}  } |x |^2  x^\perp \cdot n \, ds\\
&= \int_{  \mathcal{F}_{0}  } \operatorname{div}(|\nabla \Phi_{3} |^2  x^\perp) \, dx - 2\int_{  \mathcal{F}_{0}  }\nabla (\nabla \Phi_{3} \cdot x^\perp )  \cdot\nabla  \Phi_{3}\, dx -\int_{  \mathcal{S}_{0}  } \operatorname{div}(|x |^2  x^\perp) \, dx,
\end{split}\end{equation*}
where there is no boundary term at infinity because $\nabla \Phi_{3}=\mathcal{O}(1/|x|^2)$. Next we use the general relation \eqref{vect}
to obtain that
\begin{equation*}\begin{split}
\nabla(\nabla \Phi_{3} \cdot x^\perp )  \cdot\nabla  \Phi_{3} &= \Big[(\nabla \Phi_{3} \cdot \nabla)x^\perp+(x^\perp\cdot \nabla)\nabla \Phi_{3}\Big] \cdot\nabla  \Phi_{3}
  \\ &= -(\nabla  \Phi_{3})^\perp\cdot \nabla  \Phi_{3}+\frac12 (x^\perp \cdot \nabla) |\nabla  \Phi_{3} |^2= \frac12 \operatorname{div} (x^\perp |\nabla  \Phi_{3} |^2).
\end{split}\end{equation*}
Hence,
\begin{equation}
 \label{T5}
A_{3,b}=0.
\end{equation}
\textbf{Computation of $A_{i,c} $}.
 \ \par \
$\bullet$  We use again the Blasius formula together with \eqref{v-l} and the Cauchy's residue theorem:
\begin{equation*}\begin{split}
\Big( A_{i,c}\Big)_{i=1,2} 
&=  \int_{\partial  \mathcal{S}_{0}  }  r(v_{\#}-\ell) \cdot (\nabla \Phi_{3} -x^\perp)  n \, ds  \\
&= i r\left( \int_{ \partial  \mathcal{S}_{0}} (\widehat{v_{\#} -\ell})(  \widehat{\nabla \Phi_{3}} 
+ i \bar{z}) \, dz \right)^* \\
&=i r\Bigg(  i\Big(-\ell_{1} +i\ell_{2} \Big)\int_{ \partial  \mathcal{S}_{0}}  \bar{z} \, dz 
+i\ell_{1}  \int_{ \partial  \mathcal{S}_{0}} \widehat{\nabla \Phi_{1} } \bar{z} \, dz
+i\ell_{2}  \int_{ \partial  \mathcal{S}_{0}} \widehat{\nabla \Phi_{2} } \bar{z} \, dz
 \Bigg)^*.
\end{split}
\end{equation*}

Let us observe that 
\begin{eqnarray}
\nonumber
\int_{ \partial  \mathcal{S}_{0}} \bar{z}  \, dz
&=& \int (\gamma_{1}-i\gamma_{2})(\gamma_{1}'+i\gamma_{2}')=\int(\gamma_{1}\gamma_{1}'+\gamma_{2}\gamma_{2}' )
+ i\int(\gamma_{1}\gamma_{2}'-\gamma_{2}\gamma_{1}')  \\
\nonumber
&=& -i \int_{ \partial  \mathcal{S}_{0}} \begin{pmatrix} x_{1}\\x_{2}\end{pmatrix}\cdot n \,ds \\
\label{compu2}
&=& i \int_{ \mathcal{S}_{0}} \operatorname{div}\begin{pmatrix} x_{1}\\x_{2}\end{pmatrix} \,dx
= 2 i |\mathcal{S}_{0} |; 
\end{eqnarray}
\begin{eqnarray}
\nonumber
\int_{ \partial  \mathcal{S}_{0}} |z|^2  \, dz &=& \int (\gamma_{1}^2 + \gamma_{2}^2)(\gamma_{1}'+i\gamma_{2}') \\
\nonumber
&=& \int_{ \partial  \mathcal{S}_{0}} \begin{pmatrix} 0\\x_{2}^2\end{pmatrix}\cdot n \,ds
- i \int_{ \partial  \mathcal{S}_{0}} \begin{pmatrix} x_{1}^2\\0\end{pmatrix}\cdot n \,ds \\
\nonumber
&=&  -\int_{   \mathcal{S}_{0}} \operatorname{div} \begin{pmatrix} 0\\x_{2}^2\end{pmatrix} \,dx
+ i \int_{   \mathcal{S}_{0}} \operatorname{div} \begin{pmatrix} x_{1}^2\\0\end{pmatrix} \,dx \\
\label{compu3}
&=&  -2 |\mathcal{S}_{0} |x_{G,2}+2 i|\mathcal{S}_{0} |x_{G,1}.
\end{eqnarray}

Therefore, it suffices to write the value obtained in Lemma \ref{lemzphi}  to get:
\begin{equation*}\begin{split}
\Big( A_{i,c}\Big)_{i=1,2} 
=& r \Big[
 (-\ell_{2}) 2 |\mathcal{S}_{0} |
 -\ell_{1} m_{1,2} + \ell_{2}  (-m_{2,2}+|\mathcal{S}_{0} |)  \Big]
\\
&+ ir \Big[
\ell_{1} 2 |\mathcal{S}_{0} | 
 -\ell_{1}  (-m_{1,1}+|\mathcal{S}_{0} |)  +\ell_{2} m_{2,1} \Big],
\end{split}\end{equation*}
which can be simplified as
\begin{equation}
  \label{T5.5}
\Big( A_{i,c}\Big)_{i=1,2} 
=  r\Bigg[(\mathcal{M}_{\flat} + |\mathcal{S}_{0} | {\rm I}_{2})\ell  \Big)^\perp \Bigg]
.
\end{equation}
\par
$\bullet$ For $i=3$,  Lemma \ref{blasius}, \eqref{v-l} and Cauchy's residue theorem imply that
\begin{equation*}\begin{split}
A_{3,c} =& \int_{\partial  \mathcal{S}_{0}  }  r(v_{\#}-\ell) \cdot (\nabla \Phi_{3} -x^\perp)  K_3 \, ds
\\
=&r{\rm Re} \left( \int_{ \partial  \mathcal{S}_{0}} z (\widehat{v_{\#} -\ell})(  \widehat{\nabla \Phi_{3}} + i \bar{z})  \, dz \right) \\
=&r{\rm Re}\Bigg[ \Big(-\ell_{1} +i\ell_{2} \Big) \int_{ \partial  \mathcal{S}_{0}} (z \widehat{\nabla \Phi_{3}} + i|z|^2)\, dz   \\
&+\ell_{1}  i \int_{ \partial  \mathcal{S}_{0}} \widehat{\nabla \Phi_{1}} |z|^2\, dz
+\ell_{2}   i \int_{ \partial  \mathcal{S}_{0}} \widehat{\nabla \Phi_{2}} |z|^2\, dz 
  \Bigg].
\end{split}
\end{equation*}
Now applying Lemma \ref{lem-flux-circu} to $(|z|^2\partial_1 \Phi_{i},|z|^2\partial_2 \Phi_{i})$ we have
\[
\int_{ \partial  \mathcal{S}_{0}} |z|^2(\partial_1 \Phi_{i} - i \partial_2 \Phi_{i})  \, dz = 
\int_{ \partial  \mathcal{S}_{0}} |x|^2\partial_{\tau} \Phi_{i}  \, ds
-i\int_{ \partial  \mathcal{S}_{0}} |x|^2\partial_{n} \Phi_{i}  \, ds
\]
where we easily verify that
\begin{eqnarray*}
\int_{ \partial  \mathcal{S}_{0}} |x|^2\partial_{\tau} \Phi_{i}  \, ds 
= -\int_{ \partial  \mathcal{S}_{0}}  \Phi_{i} 2x\cdot\tau  \, ds
= -\int_{ \partial  \mathcal{S}_{0}}  \Phi_{i} 2(x^\perp \cdot n)  \, ds = -2m_{i,3}.
\end{eqnarray*}
The value of $\int_{ \partial  \mathcal{S}_{0}} |x|^2\partial_{n} \Phi_{i}  \, ds$ has already been computed in  \eqref{Stokes2}.
We therefore obtain: 
\begin{gather*}
\label{lem|z|2phi}
\int_{ \partial  \mathcal{S}_{0}} |z|^2\widehat{\nabla \Phi_{i}}   \, dz =-2m_{i,3}+2i|\mathcal{S}_{0} |x_{G,i}\,  \text{for }i=1,2 \text{ and }
\int_{ \partial  \mathcal{S}_{0}} |z|^2\widehat{\nabla \Phi_{3}}   \, dz =-2m_{3,3} .
\end{gather*}

Hence, we deduce  from this, \eqref{compu3} and Lemma \ref{lemzphi} that
\begin{equation*}\begin{split}
A_{3,c}=&r\Bigg[ -(-\ell_{1}) (m_{3,2}+|\mathcal{S}_{0} |x_{G,1} + 2 |\mathcal{S}_{0} |x_{G,1})
\\
&\quad 
(-\ell_{2}) (m_{3,1}-|\mathcal{S}_{0} |x_{G,2} -2 |\mathcal{S}_{0} |x_{G,2})\\
&\quad 
-\ell_{1} 2|\mathcal{S}_{0} |x_{G,1} -\ell_{2} 2|\mathcal{S}_{0} |x_{G,2} 
  \Bigg],
  \end{split}\end{equation*}
so that 
\begin{equation}
  \label{T6}
A_{3,c}
=  -r \ell  \cdot \Big(\begin{pmatrix} -m_{3,2} \\m_{3,1} \end{pmatrix} -|\mathcal{S}_{0} |x_{G}\Big) .
\end{equation}
\textbf{Conclusion}. Gathering  \eqref{T1},  \eqref{T2},  \eqref{T3},  \eqref{T4},  \eqref{T5},  \eqref{T5.5} and  \eqref{T6}  we obtain \eqref{Ciad} and  \eqref{C3ad}.
This ends the proof of  Lemma \ref{PropCia}.

\end{proof}

 Let us continue with the term $B_{i}$. Let us prove the following.
\begin{lem} \label{PropCib}
One has
\begin{equation} \label{C12b}
 \begin{pmatrix}B_{1} \\B_{2} \end{pmatrix}
= -  \gamma \ell^\perp 
+ \gamma  r \xi
\text{ and }
B_{3} = - \gamma  \, \xi \cdot    \ell ,
\end{equation}
where $\xi$ was defined in \eqref{DefXi}.
\end{lem}
\begin{proof}
Putting  the decomposition \eqref{diese99} in the definition of $B_{i}$, we write:
\begin{equation*}
B_{i} =   \gamma  \int_{\partial  \mathcal{S}_{0}  } (v_{\#}  - \ell )\cdot  H  K_i \, ds
+ \gamma  \int_{\partial  \mathcal{S}_{0}  }r (\nabla \Phi_{3} - x^\perp)\cdot  H  K_i \, ds.
\end{equation*}
\ \par
Concerning the second integral, as $v_{\#}-\ell$ and $H$ are tangent to the boundary, we apply the Blasius formula (see Lemma \ref{blasius}), then we compute by \eqref{v-l} and Cauchy's residue theorem and \eqref{HSeriesLaurent}:
\begin{equation*}\begin{split}
\left(\gamma \int_{\partial  \mathcal{S}_{0}  }(v_{\#}  - \ell )\cdot  H   K_i \, ds\right)_{i=1,2} =& \gamma\int_{\partial  \mathcal{S}_{0}  }(v_{\#}  - \ell )\cdot  H   n \, ds \\
=& i\gamma \left( \int_{ \partial  \mathcal{S}_{0}} \widehat{(v_{\#} -\ell)}\widehat{H} \, dz \right)^*\\
=& i\gamma \Bigg( \Big(-\ell_{1} +i\ell_{2} \Big) \int_{ \partial  \mathcal{S}_{0}}\widehat{H} \, dz   \Bigg)^*\\
=& i\gamma \Bigg( \Big(-\ell_{1} +i\ell_{2} \Big) \Bigg)^*\\
=& -\gamma \ell^\perp .
\end{split}
\end{equation*}

For $i=3$, we compute by Lemma \ref{blasius} and the Cauchy's residue theorem that
\begin{equation*}\begin{split}
\gamma\int_{\partial  \mathcal{S}_{0}  }(v_{\#}  - \ell )\cdot  H   K_3 \, ds=&\gamma{\rm Re} \left( \int_{ \partial  \mathcal{S}_{0}} z (\widehat{v_{\#} -\ell})\widehat{H}  \, dz \right) \\
=& \gamma{\rm Re} \left(  \Big(-\ell_{1} +i\ell_{2} \Big) \int_{ \partial  \mathcal{S}_{0}}z\widehat{H} \, dz \right)\\
=& \gamma{\rm Re} \left(  \Big(-\ell_{1} +i\ell_{2} \Big) (\xi_{1}+i\xi_{2})   \right)\\
=&-  \gamma  \ell \cdot \xi .
\end{split}\end{equation*}
\ \par
For the last term, we use that $\nabla \Phi_{3} -x^\perp$ and $H$ are tangent to the boundary, and we write by Lemma \ref{blasius} and by the Cauchy's residue theorem:
\begin{eqnarray*}
\left( \gamma r  \int_{\partial  \mathcal{S}_{0}  } (\nabla \Phi_{3} - x^\perp)\cdot  H   K_i \, ds\right)_{i=1,2}
&=&  i \gamma r  \left( \int_{ \partial  \mathcal{S}_{0}} (\widehat{\nabla \Phi_{3} -x^\perp})\widehat{H} \, dz \right)^*  \\
&=&  i \gamma r  \left(  i\int_{ \partial  \mathcal{S}_{0}} \bar{z}\widehat{H}  \, dz\right)^* \\
&=& \gamma r  \left(  \int_{ \partial  \mathcal{S}_{0}} \bar{z}\widehat{H}  \, dz\right)^*
 = \gamma  r \xi,
\end{eqnarray*}
where we have used that $-\widehat{x^\perp}=i\bar{z}$ and 
\begin{equation}\label{zH}
 \left(  \int_{ \partial  \mathcal{S}_{0}} \bar{z}\widehat{H}  \, dz \right)^* = \int_{ \partial  \mathcal{S}_{0}} z \widehat{H}  \, dz, 
\end{equation}
the latter being easily shown by using a parametrization.
For $i=3$, we have that:
\begin{eqnarray*}
\gamma r  \int_{\partial  \mathcal{S}_{0}  } (\nabla \Phi_{3} - x^\perp)\cdot  H  K_3 \, ds&=& \gamma  r {\rm Re} \left( \int_{ \partial  \mathcal{S}_{0}} z (\widehat{\nabla \Phi_{3} -x^\perp})\widehat{H}  \, dz \right) \\
&=&  \gamma  r {\rm Re} \left( i \int_{ \partial  \mathcal{S}_{0}} z \bar{z}\widehat{H}  \, dz \right)=0,
\end{eqnarray*}
which can also be shown by using a parametrization.
%
%
%
Gathering the equalities above yields \eqref{C12b}.
This ends the proof of Lemma \ref{PropCib}. 
\end{proof} 
We now turn to $C_{i}$ (see \eqref{Cic}). From Lemma \ref{blasius}, \eqref{HSeriesLaurent} and Cauchy's Residue Theorem, we deduce that 
\begin{equation}
  \label{T7}
C_{1}=C_{2}=C_{3}=0 .
\end{equation}
Indeed, we can verify that $\int_{\partial  \mathcal{S}_{0} } z (\widehat{H})^2  \, dz=-i/(2\pi) $ is purely imaginary.
\par

Plugging  \eqref{Ciad}, \eqref{C3ad},   \eqref{C12b} and    \eqref{T7}
in \eqref{clacla}  and using that for any $p_a := (\ell_a, \omega_a)$ and $p_b := ({\ell}_b, \omega_b) $ in  $  \mathbb{R}^{2}  \times  \mathbb{R} $, 
$p_a \times  {p}_b = (  \, \omega_a  \;  {\ell}^{\perp}_b -    {\omega}_b \;   \ell^{\perp}_a ,  \ell^{\perp}_a \cdot {\ell}_b) $, yields \eqref{eq}. 
This ends the complex-analytic proof of Theorem  \ref{reup}.

 \subsection{A real-analytic  proof of Theorem \ref{reup}}
 \label{nabedian}

We now consider another approach which dates back to Lamb.
We  therefore go back to Lemma \ref{Refo} and provide an alternative real-analytic  end of the   proof of Theorem \ref{reup}. 
Of crucial importance is the following identity which we will  use instead of Blasius' lemma though to different terms. 

Let 
\begin{equation*} 
\zeta_{1}  (x) := e_{1} ,
 \quad
\zeta_{2}  (x) := e_{2}    \quad
\text{ and } \quad
\zeta_{3}  (x) := x^\perp .
\end{equation*}
denote the elementary rigid velocities.
The following lemma seems to originate from Lamb's work. %
\begin{lem} \label{lamb}
For any pair of vector fields $(u,v)$ in $C^\infty (\overline{\mathbb{R}^2 \setminus \mathcal S_0} ; \mathbb{R}^2 )$ satisfying 
\begin{itemize}
\item $\operatorname{div} u = \operatorname{div} v = \operatorname{curl} u  = \operatorname{curl} v = 0$, 
\item $u(x) = O(1/|x|)$ and $v(x) = O(1/|x|)$ as $|x| \rightarrow + \infty$,
\end{itemize}
one has, for any $i=1,2,3$, 
\begin{equation} \label{blasius-formu}
\int_{\partial \mathcal S_0} (u\cdot v) K_{i} {\rm d}s 
= \int_{\partial  \mathcal S_0} \zeta_{i}   \cdot \Big(  (u \cdot n) v +  (v \cdot n)  u  \Big){\rm d}s .
\end{equation}
\end{lem}
\begin{proof}[Proof of Lemma~\ref{lamb}]
Let us  start with the case where $i=1$ or $2$. 
Then 
\begin{equation} \label{Bla1}
\int_{\partial  \mathcal S_0} (u\cdot v) K_{i} {\rm d}s  
= \int_{\partial  \mathcal S_0} \big( (u\cdot v) \zeta_{i} \big) \cdot n{\rm d}s 
= \int_{\mathbb{R}^2 \setminus\mathcal S_0} \operatorname{div} \big( (u\cdot v) \zeta_{i} \big)    \, {\rm d}x ,
\end{equation}
by using that $u(x) = O(1/|x|)$ and $v(x) = O(1/|x|)$ when $|x| \rightarrow + \infty$.
Therefore
\begin{equation} \label{Bla2}
\int_{\partial  \mathcal S_0} (u\cdot v) K_{i} {\rm d}s  
= \int_{\mathbb{R}^2 \setminus\mathcal S_0} \zeta_{i} \cdot \nabla (u\cdot v)     \, {\rm d}x 
= \int_{\mathbb{R}^2 \setminus\mathcal S_0} \zeta_{i}  \cdot (  u \cdot \nabla  v + v\cdot \nabla u )  \, {\rm d}x ,
\end{equation}
using that $ \operatorname{curl} u  = \operatorname{curl} v = 0$.
Now, integrating by parts, using that $\operatorname{div} u = \operatorname{div} v = 0$ and once again that $u(x) = O(1/|x|)$ and $v(x) = O(1/|x|)$ as $|x| \rightarrow + \infty$, we obtain \eqref{blasius-formu} when $i=1$ or $2$. \par
We now tackle the case where $i=3$.
We follow the same lines as above, with two precisions.
First we observe that there is no contribution at infinity in \eqref{Bla1} and \eqref{Bla2} when $i=3$ as well.
Indeed $ \zeta_{3}$ and the normal to a centered circle are orthogonal. 
Moreover there is no additional distributed term coming from the integration by parts in \eqref{Bla2} when $i=3$  since
\begin{equation*}
\int_{\mathbb{R}^2 \setminus\mathcal S_0}
v \cdot (  u \cdot \nabla_x  \zeta_{i}  )+  u \cdot (v\cdot \nabla_x  \zeta_{i}  )     \, {\rm d}x
= \int_{\mathbb{R}^2 \setminus\mathcal S_0} ( v \cdot   u^\perp  +  u \cdot v^\perp )   \, {\rm d}x = 0.
\end{equation*}
\end{proof}
As a consequence, using Lamb's lemma and the boundary conditions \eqref{Euler13} we obtain for any $i=1,2,3$, 
\begin{eqnarray*}
 \frac{1}{2} \int_{\partial  \mathcal{S}_{0}  } |{v}|^2  K_i \, ds 
 &=&
\int_{\partial  \mathcal{S}_{0} }  (v \cdot n) (v \cdot \zeta_{i} )  \, ds
\\ &=& \int_{\partial  \mathcal{S}_{0}  } \big( (\ell + r x^\perp)\cdot  n \big)  \big( v \cdot \zeta_{i}  \big)  \, ds
\\ &=& \int_{\partial  \mathcal{S}_{0}  } \big( (\ell + r x^\perp)\cdot  n \big)  \big( v \cdot n \big)  K_{i}  \, ds
\\ && \quad + \int_{\partial  \mathcal{S}_{0}  } \big( (\ell + r x^\perp)\cdot  n \big)  \big( v \cdot \tau  \big) \big(  \zeta_{i} \cdot  \tau \big)  \, ds 
\end{eqnarray*}
so that the right hand side of \eqref{anoter} can be recast as follows:
\begin{eqnarray*}
&& \frac{1}{2} \int_{\partial  \mathcal{S}_{0}  } |{v}|^2  K_i \, ds  - \int_{\partial  \mathcal{S}_{0}  } (\ell + r x^\perp)\cdot {v}  K_i \, ds 
 \\  &&\quad = - \int_{\partial  \mathcal{S}_{0}  } \big( (\ell + r x^\perp)\cdot \tau  \big)  \big( v \cdot \tau \big)  K_{i}  \, ds
+ \int_{\partial  \mathcal{S}_{0}  } \big( (\ell + r x^\perp)\cdot  n \big)  \big( v \cdot \tau  \big) \big(  \zeta_{i} \cdot  \tau \big)  \, ds 
\\ && \quad = 
 \sum_{k}  \, p_{k}  \,  \int_{\partial  \mathcal{S}_{0}  } \,   \big( v\cdot \tau  \big) 
  [ \big(  \zeta_{i} \cdot  \tau \big) K_{k}  -  \big(  \zeta_{k} \cdot  \tau \big)   K_{i}   ] \, ds ,
 \end{eqnarray*}
for any $i=1,2,3$, where the sum runs for $k$ over $1,2,3$.

 \subsubsection{Computation of the brackets}

We now compute the brackets $ [ \big(  \zeta_{i} \cdot  \tau \big) K_{k}  -  \big(  \zeta_{k} \cdot  \tau \big)   K_{i}   ]$, for $i,k =1,2,3$,  in order to make explicit the previous integrals.
\ \par \ 
$\bullet$ For $i=1$, we therefore obtain that
\begin{eqnarray*}
 && \frac{1}{2} \int_{\partial  \mathcal{S}_{0}  } |{v}|^2  K_i \, ds  - \int_{\partial  \mathcal{S}_{0}  } (\ell + r x^\perp)\cdot {v}  K_i \, ds 
  = 
   \\ && \quad   p_{2}  \,  \int_{\partial  \mathcal{S}_{0}  } \,   \big( v \cdot \tau  \big) 
  [ \big(  \zeta_{1} \cdot  \tau \big) K_{2}  -  \big(  \zeta_{2} \cdot  \tau \big)   K_{1}   ] \, ds 
 \\ && \quad +  p_{3}  \,  \int_{\partial  \mathcal{S}_{0}  } \,   \big( v \cdot \tau  \big) 
  [ \big(  \zeta_{1} \cdot  \tau \big) K_{3}  -  \big(  \zeta_{3} \cdot  \tau \big)   K_{1}   ] \, ds .
\end{eqnarray*}
Using that 
\begin{eqnarray*}
   \big(  \zeta_{1} \cdot  \tau \big) K_{2}  -  \big(  \zeta_{2} \cdot  \tau \big)   K_{1}   
&=&   \big(  \zeta_{1} \cdot  \tau \big) \big(  \zeta_{2} \cdot  n \big)    -  \big(  \zeta_{2} \cdot  \tau \big)       \big(  \zeta_{1} \cdot  n \big) 
\\ &=& 
 \big(  \zeta_{2} \cdot n\big)^2  + \big(  \zeta_{2} \cdot  \tau \big)^2    
 \\ &=& 1
\end{eqnarray*}
and
\begin{eqnarray*}
 \big(  \zeta_{1} \cdot  \tau \big) \big(  \zeta_{3} \cdot  n \big)
   -  \big(  \zeta_{3} \cdot  \tau \big)   \big(  \zeta_{1} \cdot  n \big)
 &=&
   \big(  \zeta_{2} \cdot n \big) \big(  \zeta_{3} \cdot  n \big)
   +  \big(  \zeta_{3} \cdot  \tau \big)   \big(  \zeta_{2} \cdot  \tau \big)
    \\ &=& \zeta_{2} \cdot  \zeta_{3} ,
\end{eqnarray*}
we infer from  the decomposition \eqref{vdecomp} and  \eqref{eqB/L} that 
\begin{eqnarray*}
 && \frac{1}{2} \int_{\partial  \mathcal{S}_{0}  } |{v}|^2  K_i \, ds  - \int_{\partial  \mathcal{S}_{0}  } (\ell + r x^\perp)\cdot {v}  K_i \, ds 
   =  \gamma \ell_{2} +  \gamma r  \zeta_{2} \cdot  \xi^\perp 
 \\ && \qquad \qquad    + r   \sum_{j=1}^3  \, p_{j}  \, \zeta_{2} \cdot  \int_{\partial  \mathcal{S}_{0}  }  \big(   \nabla \Phi_{j}  \cdot  \tau \big)  x^\perp\, ds .
\end{eqnarray*}
\ \par \ 
$\bullet$ For $i=2$, 
\begin{eqnarray*}
 &&  \frac{1}{2} \int_{\partial  \mathcal{S}_{0}  } |{v}|^2  K_i \, ds  - \int_{\partial  \mathcal{S}_{0}  } (\ell + r x^\perp)\cdot {v}  K_i \, ds 
  = 
   \\ && \quad   p_{1}  \,  \int_{\partial  \mathcal{S}_{0}  } \,   \big( v \cdot \tau  \big) 
  [ \big(  \zeta_{2} \cdot  \tau \big) K_{1}  -  \big(  \zeta_{1} \cdot  \tau \big)   K_{2}   ] \, ds 
 \\ && \quad +  p_{3}  \,  \int_{\partial  \mathcal{S}_{0}  } \,   \big( v \cdot \tau  \big) 
  [ \big(  \zeta_{2} \cdot  \tau \big) K_{3}  -  \big(  \zeta_{3} \cdot  \tau \big)   K_{2}   ] \, ds 
\end{eqnarray*}
But
$$ 
\big(  \zeta_{2} \cdot  \tau \big) K_{1}  -  \big(  \zeta_{1} \cdot  \tau \big)   K_{2}  
= 
- \big(  \zeta_{2} \cdot  \tau \big)^2 -  \big(  \zeta_{1} \cdot  \tau \big)^2 
=-1 ,
$$
and
$$
\big(  \zeta_{2} \cdot  \tau \big) K_{3}  -  \big(  \zeta_{3} \cdot  \tau \big)   K_{2} 
= -\big(  \zeta_{1} \cdot n \big) \big(  \zeta_{3} \cdot  n \big)    -  \big(  \zeta_{3} \cdot  \tau \big)   \big(  \zeta_{1} \cdot \tau \big)
= - \zeta_{1} \cdot  \zeta_{3} , 
$$
so that 
\begin{eqnarray*}
 && \frac{1}{2} \int_{\partial  \mathcal{S}_{0}  } |{v}|^2  K_i \, ds  - \int_{\partial  \mathcal{S}_{0}  } (\ell + r x^\perp)\cdot {v}  K_i \, ds 
  = 
-  \gamma \ell_{1} -   \gamma r  \zeta_{1} \cdot  \xi^\perp 
   \\ && \qquad \qquad    - r   \sum_{j=1}^3  \, p_{j}  \, \zeta_{1} \cdot  \int_{\partial  \mathcal{S}_{0}  }  \big(   \nabla \Phi_{j}  \cdot  \tau \big)  x^\perp\, ds .
\end{eqnarray*}
Moreover, by an integration by parts, 
$$
\int_{\partial  \mathcal{S}_{0}  }  \big(   \nabla \Phi_{j}  \cdot  \tau \big)  x^\perp\, ds 
= - 
\int_{\partial  \mathcal{S}_{0}  } \Phi_{j}  \cdot n \, ds  
=
- 
   \big( \int_{\partial  \mathcal{S}_{0}  } \Phi_{j} \partial_n  \Phi_{i}  \, ds  )_{i} 
   =
-    \big( m_{i,j} \big)_{i} ,
$$
where $i$ runs over $1,2$, 
so that 
\begin{eqnarray*}
-  \big( \frac{1}{2} \int_{\partial  \mathcal{S}_{0}  } |{v}|^2  K_i \, ds  - \int_{\partial  \mathcal{S}_{0}  } (\ell + r x^\perp)\cdot {v}  K_i \, ds   \big)_{i=1,2} 
    &=& - r  \big(   \mathcal{M}^{\flat} \ell + r  \begin{pmatrix} m_{1,3} \\ m_{2,3}   \end{pmatrix}  \big)^{\perp}  
  \\ && \quad +  \gamma  \big(  \ell^{\perp} - r  \xi \big) .
\end{eqnarray*}
$\bullet$ Proceeding in the same way for $i=3$ and using the definitions  \eqref{DefGammag} and  \eqref{VectMass} we finally arrive at \eqref{eq}.
This ends the real-analytic proof of Theorem  \ref{reup}.


 %

\subsection{Zero radius limit}
\label{sec-zr-unbounded}

We now assume that, for every $\varepsilon \in (0,1]$, the domain occupies 
\eqref{DomInit} and for every $q=(h,\theta)\in\mathbb R^3$, 
\begin{equation} \label{def-solide-scaled}
\mathcal S^\varepsilon(q) :=R(\theta)\mathcal S^\varepsilon_0 +h \text{ and }\mathcal F^\varepsilon(q)=\mathbb R^2 \setminus\bar{\mathcal S}^\varepsilon(q).
\end{equation}
 We will treat at once both the  massive and massless cases. 
 The following statement implies Theorem \ref{thm-intro-sec1-VO}.

\begin{thm} 
\label{pasdenom-VO}
We consider a rescaled initial domain $\mathcal{S}_0$ occupied by the body, some  initial solid translation and rotation velocities $( \ell_0 , r_0 )$  and a  circulation $ \gamma$  in $ \mathbb R$ in the case of a  massive particle and in  $ \mathbb R^*$ in the  case of a massless particle, all 
  independent of $\varepsilon$.
Let, for each $\varepsilon > 0$, the solution $h^\varepsilon$  associated with an initial solid domain ${\mathcal S}_0^\varepsilon$ defined by \eqref{DomInit}
 with the inertia scaling  described in Definition \ref{massiveP}, 
 and initial data $q(0)= 0$ and $q'  (0) = (\ell_0 ,  r_0 ) $, 
given by Proposition \ref{CP}. 
Then for all $T>0$, as $\varepsilon \to 0$ one has in the case of massive particle (respectively massless particle) 
 $h^\varepsilon \relbar\joinrel\rightharpoonup h$ 
 in $W^{2,\infty}([0,T];\mathbb R^2)$ weak-$\star$  (resp. $W^{1,\infty}([0,T];\mathbb R^2)$ weak-$\star$) 
and 
 $\varepsilon\theta^\varepsilon \relbar\joinrel\rightharpoonup 0$ in $W^{2,\infty}([0,T];\mathbb R)$ weak-$\star$  (resp. $W^{1,\infty}([0,T];\mathbb R)$ weak-$\star$) .
 Moreover the limit time-dependent vector $h$ satisfies  the  equations $m h''  = \gamma (h^\prime)^\perp $  (resp.  $h^\prime = 0 $). 
\end{thm}

\begin{proof}
In order to compare the influence of  the circulation and of the solid velocity in the zero-radius limit we first consider the harmonic field   $H^{\varepsilon}$ and  the Kirchhoff potentials  $(\nabla \Phi_{i}^\varepsilon)_{i=1,2, 3}$ associated with the rigid body $\mathcal S^\varepsilon(0)$  as the harmonic field   $H$ and  the Kirchhoff potentials  $(\nabla \Phi_{i})_{i=1,2, 3}$ were associated with the rigid body $\mathcal S_0$ in Section \ref{sec-ex}.
They satisfy the following scaling law
$H^{\varepsilon}(x) = \varepsilon^{-1} H \left(x/\varepsilon\right)$,
whereas obey:
$\Phi_{i}^\varepsilon(x)= \varepsilon \Phi_{i}^1(x/\varepsilon)$ for $ i=1,2$, 
 $\Phi_{3}^\varepsilon(x)= \varepsilon^{2} \Phi_{3}^1(x/\varepsilon)$.
Therefore  the harmonic field  $H^{\varepsilon}$ is more singular than  the Kirchhoff potentials $\nabla \Phi_{i}^\varepsilon(x)$ in the vanishingly small limit. 
On the other hand we deduce that the added inertia  is given by the following matrix
\begin{equation}
  \label{wemerde}
\mathcal{M}^{\varepsilon}_{a,\theta}
= \varepsilon ^2 \, I_\varepsilon \mathcal{M}_{a,\theta}I_\varepsilon ,
\end{equation}
where $I_{\varepsilon}$ is the diagonal matrix 
$I_{\varepsilon} :=     \text{diag } (1,1,\varepsilon)$.
This has to be compared with the genuine inertia matrix which, according to Definition \ref{massiveP}, scales as follows:
\begin{equation}
  \label{genouine}
 \mathcal{M}^{\varepsilon}_g :=    \text{diag } (m^\varepsilon  , m^\varepsilon  , \mathcal{J}^{\varepsilon} ) =
\varepsilon^\alpha  I_\varepsilon \mathcal{M}_g I_\varepsilon ,
\end{equation}
where the matrix $\mathcal{M}_g$ is given in terms of $m>0$  and  $\mathcal J>0$  by \eqref{DefMG}.
Recall  that $ \alpha > 0$, $m>0$  and  $\mathcal J>0$ are defined in Definition \ref{massiveP}
and fixed  independent of $\varepsilon$.

Two remarks are in order.
\begin{itemize}
\item First we observe from the comparison of  \eqref{wemerde} and  \eqref{genouine} 
 that the physical case $\alpha = 2$ appears as critical. 
\item Secondly because of the matrix $ I_\varepsilon$ in the right hand sides of the two inertia matrices $ \mathcal{M}^{\varepsilon}_g$ and $\mathcal{M}^{\varepsilon}_{a,\theta}$, it is natural to introduce the vector
$ p^\varepsilon = ( (h^\varepsilon)'  ,  \varepsilon (\theta^\varepsilon)'  )^t $.
Hence the natural counterpart to $(h^\varepsilon)'  $ for what concerns the angular velocity is rather $\varepsilon (\theta^\varepsilon)'  $ than $(\theta^\varepsilon)' $. 
This can also be seen on the boundary condition \eqref{souslab}: when $x$ belongs to $\partial\mathcal S^{\varepsilon} (t)$, 
the term $ (\theta^\varepsilon)'  (x-h^\varepsilon)^\perp$ is of order $\varepsilon (\theta^\varepsilon)'  $ and is added to $ (h^\varepsilon)' $. 
\end{itemize}
 Examining  how the other terms scale with $ \varepsilon$
one obtains:
\begin{equation}  \label{ODE_ext-eps}
 ( \varepsilon^\alpha  \,  \mathcal{M}_g   + \varepsilon^2 \mathcal{M}_{a,\,  \theta^\varepsilon} )  \, (p^\varepsilon)'
  + \varepsilon \langle  \Gamma_{a,\theta^\varepsilon} ,p^\varepsilon,p^\varepsilon\rangle
 =  \gamma p^\varepsilon \times B_{ \theta^\varepsilon}  .
\end{equation}
The energy associated with this scaling is twice 
$( \varepsilon^\alpha \mathcal{M}_g  +\varepsilon^2  \mathcal{M}_{a,\,  \theta^\varepsilon} )  p^\varepsilon \cdot p^\varepsilon $
and its  conservation provides that $p^\varepsilon$ is bounded uniformly with respect to  $\varepsilon$ on the time interval $ [0,+\infty)$, whatever is $\alpha$. 
Now our goal is to pass to the limit in   \eqref{ODE_ext-eps}. Let $T> 0$. 
 Computing the right-hand-side of  \eqref{ODE_ext-eps} gives 
\begin{equation}
 \label{ODE_ext-eps2}
 ( \varepsilon^\alpha  \,  \mathcal{M}_g   + \varepsilon^2 \mathcal{M}_{a,\,  \theta^\varepsilon} )  \, (p^\varepsilon)'
  + \varepsilon\langle  \Gamma_{a,\theta^\varepsilon} ,p^\varepsilon,p^\varepsilon\rangle
 =  \gamma
\begin{pmatrix}
((h^\varepsilon)')^\perp -  \varepsilon (\theta^\varepsilon)'  {R}(\theta^\varepsilon) \xi\\
{R}(\theta^\varepsilon) \xi\cdot (h^\varepsilon)'
\end{pmatrix} .
\end{equation}

We start with the massive case for which $\alpha=0$. Using the equation we deduce some uniform $W^{2,\infty}$ bounds on $h_{\varepsilon}$ and $\varepsilon \theta^\varepsilon$ and this entails the existence of a  subsequence
of $(h^{\varepsilon}, \varepsilon \theta^{\varepsilon})$ converging to $(h,\Theta) $ in $W^{2,\infty}$ weak-$\star$.
Moreover the left hand side of  \eqref{ODE_ext-eps2} (with $\alpha=0$) converges to $ \mathcal{M}_g  ( h''  ,\Theta'' )^t$            
 in $ L^{\infty}   $ weak-$\star$, and using that 
 \begin{equation}
  \label{simplemaiscostaud}
 \varepsilon (\theta^\varepsilon)'  {R}(\theta^\varepsilon) \xi = \varepsilon \big( {R}(\theta^\varepsilon- \pi /2) \xi \big)'
\end{equation}
  converges in  $W^{-1,\infty} $  weak-$\star$ up to a subsequence and that 
the weak-$\star$ convergence in $W^{2,\infty}$ entails the strong $W^{1,\infty}$ one,  we get  from the two first lines of  \eqref{ODE_ext-eps2}
that $m h''= \gamma (h' )^{\perp}$ and $ \big( h(0) , h'(0)  \big) =  \big( 0, \ell_{0}  \big) $.
In order to  prove that $\Theta=0$ one may use a stationary phase argument, cf. Lemma $10$ in \cite{GLS} for more on this.

In the massless case, that is when  $\alpha > 0$, a few modifications in the arguments are in order. 
 First, thanks to the energy estimate, $\varepsilon  \mathcal{M}_{a,\, \theta^\varepsilon}$ is bounded in $W^{1,\infty}$ and 
since   $\mathcal{M}_g $ is constant and $(p^\varepsilon)'$ is bounded uniformly with respect to  $\varepsilon$ in $W^{-1,\infty}$, we can conclude that the left hand side of \eqref{ODE_ext-eps2}
converges to $0$ in $W^{-1,\infty}$ due to the extra powers of $\varepsilon$.
Next, concerning the right hand side, the term $\varepsilon \langle \Gamma_{a,\theta^\varepsilon}, {(p^\varepsilon)}, {(p^\varepsilon)} \rangle$ converges to $0$ in $L^{\infty}$ since the terms inside the brackets are bounded. 
As before the last term in the two first lines of the equation \eqref{ODE_ext-eps2}, converges weakly to $0$ in $W^{-1,\infty}$ 
Hence we infer that $(h^\varepsilon)'$ converges weakly-$\star$ to $0$ in $W^{-1,\infty}$. Due to the a priori estimate, this convergences occurs in $L^{\infty}$ weak-$\star$.
Again this is sufficient to deduce  the strong convergence of $h^{\varepsilon}$ towards some $h$   in $L^{\infty}$, and that
$h'=0$  and $ h (0) = 0 $. This concludes the proof of Theorem \ref{pasdenom-VO}.

\end{proof}

%
\section{Case of a bounded domain}
\label{chap2}

We consider now the case where the system fluid-solid occupies a 
 bounded open regular  connected and simply connected  domain $\Omega$ of $\mathbb R^2$. 
We assume that the body initially occupies  the closed domain $ \mathcal{S}_0 \subset  \Omega$, so that the  domain of the fluid is $\mathcal F_0 =\Omega\setminus{\mathcal S}_0$ 
at the initial time, and 
(without loss of generality) that the center of mass of the solid coincides at the initial time with the origin and that  $0 \in \Omega.$ 
The  domain of the fluid is denoted by  $\mathcal F(t) = \Omega\setminus{\mathcal S}(t)$ at  time $t>0$.
The fluid-solid system is governed by the following set of coupled equations:
\begin{subequations} \label{SYS_full_system}
\begin{gather} 
\frac{\partial u}{\partial t}+(u\cdot\nabla)u +\nabla \pi=0\quad\text{in }\mathcal F(t) , \label{EEE1} \\
\operatorname{div} u=0 \quad \text{in }\mathcal F(t), \label{E2} \\
m h'' =\int_{\partial\mathcal S(t)}\pi n\, {\rm d}s  \quad \text{ and } \quad 
\mathcal J \theta''=\int_{\partial\mathcal S(t)}(x-h(t))^\perp\cdot \pi n\, {\rm d}s,  \label{EqRot} \\
u\cdot n=\big( \ell  + r (\cdot-h)^\perp\big) \cdot n   \quad  \text{on }\partial\mathcal S(t) ,\label{souslab}\\
u\cdot n=0  \quad  \text{on }\partial\Omega, \label{souslabis} \\
u_{t=0}=u_0  \quad  \text{in }\mathcal F_0   \quad \text{ and } \quad 
(h,h' ,\theta,\theta' )  (0)= (0 , \ell_0 , 0 ,    r_0 ).
\end{gather}
\end{subequations}
Above we have denoted by $( \ell_0 , r_0 )$ in $\mathbb{R}^2 \times \mathbb{R}  $ the   initial solid translation and rotation velocities, 
and  by $u_0$ the  compatible initial fluid velocity  associated with $( \ell_0 , r_0 )$ and with the initial circulation  $ \gamma$ in $ \mathbb{R}$
according to Definition \ref{CompDataBd}. 
In particular we still consider the case without any initial vorticity  that is we assume that  the initial velocity $ u_0$ satisfies $\operatorname{curl} u_0 = 0$ in $\mathcal F_0$ (cf. Definition \ref{CompDataBd}),
so that it will remain irrotational for every time, that is 
\begin{equation} \label{irr}
\operatorname{curl} u (t,\cdot) = 0 \text{ in } \mathcal F(t) .
\end{equation}
On the other hand the circulation around the body is constant in time equal to $ \gamma$ 
according to Kelvin's theorem.
Since the domains $\mathcal S(t)$ and $\mathcal F(t)$ depend on $q:=(h , \theta)\in\mathbb R^3.$ only, we shall rather denote them $\mathcal S(q)$ and $\mathcal F(q)$ in the sequel.
Since  we will not consider any collision here, we introduce:
$\mathcal Q :=\{q\in\mathbb R^3\,:\,d(\mathcal S(q),\partial\Omega)>0\}$, where $d(A,B)$ denotes for two sets $A$ and $B$ in the plane that is 
$d(A,B) := \min \, \{  | x-y  |_{\mathbb{R}^{2}} , \ x \in A, \ y \in B \}.$
Above  the notation stands $ | \cdot |_{\mathbb{R}^{d}}$ for the Euclidean norm in $\mathbb{R}^{d}$.
Since $\mathcal S_0$ is a closed subset in the open set $ \Omega$  the initial position $q(0)=0$ of the solid belongs to $\mathcal Q$.

As in the unbounded case of the previous section, our analysis here will rely on a reformulation of the system above as an second-order differential ODE for $q$ together with an auxiliary div-curl type system for the fluid velocity.  
Indeed, again, the solid drives  the dynamics of the coupled system  as a consequence of the added inertia phenomenon. 
However in the case where the system occupies a bounded domain 
the matrix $\mathcal{M}_a$ encoding the added inertia depends not only on $\theta$ but  on the three components of  $q$. 
We therefore extend the definition \ref{Christ} to this new setting. 

\begin{defn}[a-connection] \label{Christq}
Given a $C^{\infty}$ mapping  
$ q\in \mathcal Q  \mapsto {\mathcal M}_{a}(q) \in S^{+}_3 (\mathbb{R})$, 
we say that the  $C^{\infty}$ mapping
$ q\in \mathcal Q  \mapsto  \Gamma_{a} (q)  \in \mathcal{BL} (\mathbb{R}^3 \times \mathbb{R}^3 ; \mathbb{R}^3 ) $  is the a-connection associated with this mapping if for any 
$ p \in\mathbb R^3$,

\begin{equation}
\langle\Gamma_{a} (q)  ,p,p\rangle :=\left(\sum_{1\leqslant i,j\leqslant 3} (\Gamma_{a} (q) )^k_{i,j} p_i p_j  \right)_{1\leqslant k\leqslant 3}\in\mathbb R^3 ,
\end{equation}
with  for every $i,j,k \in\{1,2,3\}$, 
\begin{equation} 
(\Gamma_{a} (q) )^k_{i,j} (q)  :=  \frac12
\Big(  ({\mathcal M}_{a} (q))_{k,j}^{i} + ({\mathcal M}_{a} (q))_{k,i}^{j} - ({\mathcal M}_{a} (q))_{i,j}^{k}  \Big) (q) ,
\end{equation}
where $({\mathcal M}_{a} (q))_{i,j}^{k} $ denotes the partial derivative with respect to $ q_{k}$ of the entry of indexes $(i,j)$ of the matrix ${\mathcal M}_{a} (q)$, that is
\begin{equation} 
\label{PBX}
({\mathcal M}_{a} (q))_{i,j}^{k} := \frac{\partial ({\mathcal M}_{a} (q))_{i,j}}{\partial q_{k}} .
\end{equation}
\end{defn}

Remark \ref{rem-order} is still in order for the definition above. 

\subsection{Reduction to an ODE in the case where $\gamma =0$. Munnier's theorem.}

Let us start with the case where the circulation  $\gamma$ is zero. 
Then the initial fluid velocity  and therefore the velocity at any time is potential (that is a gradient globally on  $\mathcal F(q)$).
The following result  was proven  surprisingly only recently, by Munnier in \cite{Munnier}.
This result consists in a  reformulation of the system~\eqref{SYS_full_system} in terms of an ordinary differential equation for the motion of the rigid body which corresponds to 
geodesics associated with the Riemann metric induced on $\mathcal Q$ by the matrix $ {\mathcal M}_{g}   +\mathcal{M}_{a}(q)$, where we recall that ${\mathcal M}_{g} $ is the genuine inertia defined in Definition \ref{DefMG}.
This reformulation therefore establishes an equivalence of the Newtonian and the Lagrangian points of view in the potential case.
\begin{thm}
 \label{THEO-mumu}
Let be given the open  regular  connected and simply connected bounded cavity $ \Omega$,  the  initial closed domain $\mathcal{S}_0  \subset  \Omega$ occupied by the body, the  initial solid translation and rotation velocities $( \ell_0 , r_0 )$ in $\mathbb{R}^2 \times \mathbb{R}  $.
Assume that the circulation  $ \gamma$ is $0$. 
Let $u_0$ be the associated compatible initial fluid velocity according to Definition \ref{CompDataBd}.
Then there exists  a $C^{\infty}$ mapping $ q\in\mathcal Q  \mapsto {\mathcal M}_{a}(q) \in S^{+}_3 (\mathbb{R})$,  depending only on $ \mathcal S_0$ and $ \Omega$,  
such that  up to the first collision, System~\eqref{SYS_full_system} is equivalent to  the second order differential equation: 
$$ ({\mathcal M}_{g}   + \mathcal{M}_{a}(q) )  q''  + \langle \Gamma_{a}  (q),q',q'\rangle =  0,$$
with Cauchy data $q(0)=0 \in \mathcal Q,  \ q'(0)= (\ell_0 ,r_0)\in   \mathbb R^2 \times \mathbb R$, where $\Gamma_{a}  $  is the a-connection associated with ${\mathcal M}_{a}$ by 
Definition \ref{Christq}.
For any $q \in \mathcal Q$ 
the fluid velocity $u (q,\cdot)$ is the unique solution of  the div-curl type system in the doubly-connected domain $\mathcal F(q)$, constituted of 
\eqref{E2},
\eqref{irr},
\eqref{souslab},
\eqref{souslabis},
together with the prescription of zero circulation. 
\end{thm}

Indeed we are going to provide a quite explicit  expression of ${\mathcal M}_{a}(q)$.
Consider the functions $\zeta_{j}$, for $j=1,2,3$, defined for $q=(h,\theta) \in \mathcal Q$ and $x \in  \mathcal F(q) $, by the formula
$\zeta_{j}  (q,x) := e_{j}, \text{ for } j=1,2
\text{ and } \zeta_{3}  (q,x) := (x-h)^\perp $.
Above $e_1$ and $e_2$ are the unit vectors of the canonical basis.
We introduce $ \boldsymbol\Phi :=(\Phi_1,\Phi_2,\Phi_3)^{t}$ where 
 the Kirchhoff's potentials $\Phi_j(q,\cdot)$, for $j=1,2,3$,  are the unique (up to an additive constant) solutions in $\mathcal F(q)$ of the following Neumann problem:
\begin{equation} \label{Kir}
\Delta \Phi_j = 0   \text{ in } \mathcal F(q),\quad
\frac{\partial \Phi_j}{\partial n} (q,\cdot)= n \cdot  \zeta_{j} (q,\cdot)  \text{ on }\partial\mathcal S(q), \quad
\frac{\partial\Phi_j}{\partial n}(q,\cdot)  =0  \text{ on } \partial \Omega.
\end{equation}
We can now define the added inertia
\begin{gather}
\label{def-MAq}
\mathcal{M}_a (q)  := \int_{\partial \mathcal S(q)}\boldsymbol\Phi(q,\cdot)\otimes\frac{\partial\boldsymbol\Phi}{\partial n}(q,\cdot){\rm d}s
= \Big( \int_{ {\mathcal F}(q) } \nabla \Phi_{i}  \cdot  \nabla\Phi_{j} {\rm d}x 
  \Big)_{1  \leqslant i,j  \leqslant  3} .
\end{gather}

The added inertia matrix  $\mathcal{M}_{a} (q)$ is  symmetric  positive-semidefinite. 
%

%
%
\subsection{Proof of Munnier's result: Theorem~\ref{THEO-mumu}}
\label{proofMumu}

The first step of the proof of Theorem~\ref{THEO-mumu} consists in a trade of the fluid pressure against the fluid velocity and its first order derivatives in the body's dynamics.
Indeed we start with the observation that the equations (\ref{SYS_full_system}g-h) can be summarized in the variational form:
\begin{equation}
\label{bern_1}
m h''\cdot\ell^\ast+\mathcal J\theta'' r^\ast=
\int_{\partial\mathcal S(q)}\pi(r^\ast(x-h)^\perp+\ell^\ast)\cdot n{\rm d}s,\quad\forall\,p^\ast=(\ell^\ast ,r^\ast)\in\mathbb R^3.
\end{equation}
Let us associate with $(q,p^\ast) \in \mathcal Q \times \mathbb{R}^3$, with $p^\ast =(p^\ast_1 , p^\ast_2 , p^\ast_3 )$,  the potential vector field 
\begin{equation}
  \label{pot*}
u^\ast :=\nabla (\boldsymbol\Phi(q,\cdot)\cdot p^\ast) = \nabla ( \sum_{j=1}^3 \Phi_j (q,\cdot) p^\ast_j ),
\end{equation}
which is defined on $\mathcal F(q)$. 
The pressure $\pi$ can be recovered by means of Bernoulli's formula which is obtained by combining \eqref{EEE1} and  \eqref{irr}, and which reads:
\begin{equation}
\label{EQ_bernoulli}
\nabla \pi =-\left(\frac{\partial u}{\partial t}+\frac{1}{2}\nabla (u^2)\right)\quad\text{in }\mathcal F(q).
\end{equation}
According to Bernoulli's formula  \eqref{EQ_bernoulli} and upon an integration by parts, identity \eqref{bern_1} can be turned into:
\begin{equation}
\label{bern_2}
mh''\cdot\ell^\ast+\mathcal J\theta''r^\ast=-\int_{\mathcal F(q)}\left(\frac{\partial u}{\partial t}+
\frac{1}{2}\nabla (u^2) \right)\cdot u^\ast{\rm d}x,\quad\forall\,p^\ast=(\ell^\ast ,r^\ast)\in\mathbb R^3.
\end{equation}

So far we have only used that the fluid velocity $u$ is irrotational. 
Let us now use that it is potential and therefore reads as $u=u_1$ with $u_1$ as follows:
\begin{equation} \label{DecompU1}
u_1  (q,\cdot) :=\nabla (\boldsymbol\Phi(q,\cdot)\cdot q') = \nabla ( \sum_{j=1}^3 \Phi_j (q,\cdot)q'_j ) ,
\end{equation}
where $q\in\mathcal Q$. 
For any $q\in\mathcal Q$ this vector field $u_1  (q,\cdot)$  is the only solution to  the div-curl type system in the doubly-connected domain $\mathcal F(q)$, constituted of 
\eqref{E2},
\eqref{irr},
\eqref{souslab},
\eqref{souslabis},
together with the prescription of zero circulation. 
 Observe that besides the dependence with respect to  $\mathcal S_0$, to $\Omega$ and to the space variable, $u_1$ 
depends on $q$ and linearly on $q'$. 
 
 Then   Theorem~\ref{THEO-mumu} will follow from the  following lemma.
\begin{lem} \label{LEM_3}
For any smooth curve $q(t)$ in $\mathcal Q$ and every $p^\ast =(\ell^\ast ,r^\ast) \in\mathbb R^3$, the following identity holds:
\begin{gather} \label{lagrange_1}
m h''\cdot\ell^\ast+\mathcal J\theta'' r^\ast
+ \int_{\mathcal F(q)} \left(\frac{\partial u_1}{\partial t}+\frac{1}{2}\nabla|u_1|^2\right)\cdot u^\ast{\rm d}x
\\ \nonumber =  ({\mathcal M}_{g}   + \mathcal{M}_{a}(q) ) q'' \cdot p^\ast 
+
\langle \Gamma_{a} (q),q',q'\rangle \cdot p^\ast ,
\end{gather}
where $u^\ast$ is given by \eqref{pot*}, $u_{1}$ is given by  \eqref{DecompU1},
$\mathcal{M}_a(q)$ and $\Gamma_{a} (q)$ are defined in \eqref{def-MAq} and  Definition \ref{Christq}.
\end{lem}
\begin{proof}[Proof of Lemma~\ref{LEM_3}]
We start with observing that, under the assumptions of Lemma  \ref{LEM_3}, 
\begin{equation} \label{trotriv}
m h''\cdot\ell^\ast+\mathcal J\theta'' r^\ast
 =  \mathcal{M}_g\,  q''\cdot p^\ast .
\end{equation} 
Now in order to deal with the last term of the left hand side of \eqref{lagrange_1} we use a Lagrangian strategy.
For any  $q$ in $\mathcal Q$ and every $p =(p_{1} , p_{2} , p_{3}) $ in $\mathbb R ^3$, let us denote
\begin{equation}
\label{cococo}
\mathcal{E}_1(q,p) :=\frac{1}{2}\int_{\mathcal F(q)}| \nabla (\boldsymbol\Phi(q,\cdot)\cdot p) |^2{\rm d}x .
\end{equation} 
Thus $\mathcal{E}_1(q,p)$ denotes the kinetic energy  of the potential part of the flow associated with a body at position $q$ with velocity $p$.
It follows from classical shape derivative theory that $\mathcal{E}_1 \in C^{\infty} \big(\mathcal Q \times  \mathbb R ^3 ; \lbrack 0,+\infty ) \big)$.

\begin{lem}
\label{LEM_1}
For any smooth curve $t \mapsto q(t)$ in $\mathcal Q$ , for  every $p^\ast\in\mathbb R ^3$, we have:
\begin{equation}
\label{lagrange}
\int_{\mathcal F(q)}
\left(\frac{\partial u_1}{\partial t}+\frac{1}{2}\nabla|u_1|^2\right)\cdot u^\ast{\rm d}x =
\mathcal E  \mathcal L 
\end{equation}
where   $u_{1}$ is given by  \eqref{DecompU}, $u^\ast$ is given by \eqref{pot*} and $\mathcal E  \mathcal L $ denotes the time-dependent smooth real-valued function:
\begin{equation}
  \label{def-EL}
\mathcal E  \mathcal L  :=  \Big( \frac{d}{dt}  \big(  \frac{\partial \mathcal{E}_1 }{\partial p} (q(t),q'(t)) \big)   - \frac{\partial \mathcal{E}_1 }{\partial q}   (q(t),q'(t))   \Big)\cdot p^\ast .
\end{equation}
\end{lem}

The name of the function $\mathcal E  \mathcal L $ refers to  Euler and Lagrange. For sake of simplicity below we will simply denote 
$$\mathcal E  \mathcal L  = \left(\frac{d}{dt}\frac{\partial \mathcal{E}_1 }{\partial p}-\frac{\partial \mathcal{E}_1 }{\partial q}\right)\cdot p^\ast  .$$
Let us also introduce a slight abuse of notations which simplifies the presentation of the proof of Lemma \ref{LEM_1}.
For a smooth function $I(q,p)$, where $(q,p)$ is running into $\mathcal Q \times \mathbb R ^3$, and  a smooth curve $q(t)$ in $\mathcal Q$ let us denote
\begin{equation*}
\left( \frac{\partial}{\partial q} \frac{d}{dt} I(q,p) \right) (t) := (\frac{\partial}{\partial q} J) \big(q(t),q'(t),q''(t))\big), 
\end{equation*}
where, for $(q,p,r)$ in  $\mathcal Q \times \mathbb R ^3 \times \mathbb R ^3$,
\begin{equation} \label{chain}
J(q,p,r) = p  \frac{\partial I}{\partial q} (q,p) + r \frac{\partial I}{\partial p} (q,p) .
\end{equation}
Observe in particular that
\begin{equation*}
\frac{d}{dt} \big(   I (q(t),q'(t)) \big) =  J \big(q(t),q'(t),q''(t))\big), 
\end{equation*}
and
\begin{equation} \label{ty}
\frac{d}{dt} \left( \frac{\partial I}{\partial q}   (q(t),q'(t)) \right) 
=  \left( \frac{\partial}{\partial q} \frac{d}{dt} I(q,p) \right) (t) .
\end{equation}
Below, in such circumstances,  it will be comfortable to write 
\begin{equation*}
\frac{\partial}{\partial q}  \left[ J \big(q(t),q'(t),q''(t))\big)  \right] \text{ instead of }
\left(\frac{\partial J}{\partial q} \right) \big(q(t),q'(t),q''(t))\big) , 
\end{equation*}
and it will be understood that $J$ is  extended  from $\big(q(t),q'(t),q''(t))\big) $ to general $(q,p,r)$ by \eqref{chain}.

\begin{proof}[Proof of Lemma~\ref{LEM_1}]
We start with computing the right hand side of  \eqref{lagrange}.
On the one hand the linearity of $u_1$ with respect to $p$ and then an integration by parts  leads to:
$$
\frac{\partial \mathcal{E}_1 }{\partial p}\cdot p^\ast =\int_{\mathcal F(q)} u_1\cdot u^\ast{\rm d}x  =\int_{\partial\mathcal S(q)}(\boldsymbol\Phi\cdot p )(u^\ast\cdot n)\,{\rm d}s .
$$
Then, invoking the Reynold transport theorem, we get:
\begin{equation}
\label{theV}
\frac{\partial \mathcal{E}_1 }{\partial p} \cdot p^\ast 
=\frac{\partial}{\partial q}\left(\int_{\mathcal F(q)}(\boldsymbol\Phi\cdot p )\,{\rm d}x\right)\cdot p^\ast-
\int_{\mathcal F(q)}(\frac{\partial \boldsymbol\Phi }{\partial q}\cdot p)\cdot p^\ast{\rm d}x.
\end{equation}
On the other hand, again using Reynold's formula, we have:
 \begin{equation}
\label{7}
\frac{\partial \mathcal{E}_1}{\partial q}  \cdot p^\ast=\int_{\mathcal F(q)}\left(\frac{\partial u_1}{\partial q}\cdot p^\ast\right)\cdot u_1\,{\rm d}x+\frac{1}{2}\int_{\partial\mathcal S(q)}|u_1|^2(u^\ast\cdot n)\,{\rm d}s .
\end{equation}

Differentiating  \eqref{theV} with respect to $t$, we obtain:
\begin{equation} \label{PPP}
\frac{d}{dt}\frac{\partial \mathcal{E}_1 }{\partial p} \cdot p^\ast =\frac{d}{dt}\frac{\partial}{\partial q}\left(\int_{\mathcal F(q)}(\boldsymbol\Phi\cdot p )\,{\rm d}x\right)\cdot p^\ast-\frac{d}{dt}\left(\int_{\mathcal F(q)}(\frac{\partial \boldsymbol\Phi }{\partial q}\cdot p)\cdot p^\ast{\rm d}x\right) .
\end{equation}
With the abuse of notations  mentioned above we commute the derivatives involved in the first term of the right hand side, so that 
 the identity \eqref{PPP} can be rewritten as follows:
\begin{equation} \label{DD0}
\frac{d}{dt}\frac{\partial \mathcal{E}_1 }{\partial p} \cdot p^\ast 
= \frac{\partial}{\partial q} \frac{d}{dt}\left(\int_{\mathcal F(q)}(\boldsymbol\Phi\cdot p )\,{\rm d}x\right)\cdot p^\ast - \frac{d}{dt}\left(\int_{\mathcal F(q)}(\frac{\partial \boldsymbol\Phi }{\partial q}\cdot p)\cdot p^\ast{\rm d}x\right).
\end{equation}
Moreover, using again Reynold's formula, we have:
\begin{align} \label{DD1}
\frac{d}{dt}\left(\int_{\mathcal F(q)}(\boldsymbol\Phi\cdot p )\,{\rm d}x\right) &= \int_{\mathcal F(q)}\partial_t(\boldsymbol\Phi\cdot p )\,{\rm d}x + \int_{\partial\mathcal S(q)}(\boldsymbol\Phi\cdot p )(u_1\cdot n)\,{\rm d}s \\
&\quad =\int_{\mathcal F(q)}\partial_t(\boldsymbol\Phi\cdot p )\,{\rm d}x + 2 \mathcal{E}_1 (q,p),
\end{align}
by integration by parts.

We infer from \eqref{DD0} and \eqref{DD1}, again with the abuse of notations  mentioned above, that:
\begin{align}
\label{6}
 \mathcal E  \mathcal L 
=&
\frac{\partial \mathcal{E}_1}{\partial q} 
+ \frac{\partial}{\partial q} 
\left[ \int_{\mathcal F(q)}
\partial_t(\boldsymbol\Phi\cdot p )\,{\rm d}x\right]\cdot p^\ast
 -\frac{d}{dt}\left(\int_{\mathcal F(q)}(\frac{\partial \boldsymbol\Phi }{\partial q}\cdot p)\cdot p^\ast{\rm d}x\right).
\end{align}
Thanks to Reynold's formula, we get for the second term of the right hand side 
\begin{equation} \label{8}
\frac{\partial}{\partial q} \left[ \int_{\mathcal F(q)}
\partial_t(\boldsymbol\Phi\cdot p )\,{\rm d}x\right] \cdot p^\ast
= \int_{\mathcal F(q)}\frac{\partial}{\partial q}(\partial_t(\boldsymbol\Phi\cdot p ))\cdot p^\ast\,{\rm d}x
+ \int_{\partial\mathcal S(q)}\partial_t(\boldsymbol\Phi\cdot p )(u^\ast\cdot n)\,{\rm d}s,
\end{equation}
and for the last one:
\begin{align} \nonumber
\frac{d}{dt} \left( \int_{\mathcal F(q)} \left(\frac{\partial \boldsymbol\Phi }{\partial q}\cdot p\right) \cdot p^\ast{\rm d}x\right)
&= \int_{\mathcal F(q)}\partial_t\left( \left(\frac{\partial \boldsymbol\Phi }{\partial q}\cdot p\right)\cdot p^\ast\right)\,{\rm d}x
\\ & \quad +\int_{\partial\mathcal S(q)}\left(\left(\frac{\partial \boldsymbol\Phi }{\partial q}\cdot p\right)\cdot p^\ast\right)(u_1 \cdot n)\,{\rm d}s .
\end{align}
Using again \eqref{ty} for the first term and integrating by parts the second one, we obtain: 
\begin{align} \label{9} 
\frac{d}{dt}\left(\int_{\mathcal F(q)} \left(\frac{\partial \boldsymbol\Phi }{\partial q}\cdot p\right)\cdot p^\ast{\rm d}x\right)
&=  \int_{\mathcal F(q)} \frac{\partial }{\partial q} \big( \partial_t (\boldsymbol\Phi\cdot p ) \big)\cdot p^\ast \,{\rm d}x
\\ \nonumber & \quad +\int_{\mathcal F(q)} \left(\frac{\partial u_1}{\partial q}\cdot p^\ast\right) \cdot u_1 \,{\rm d}x .
\end{align}
Plugging the expressions \eqref{7}, \eqref{8} and \eqref{9} into \eqref{6} and simplifying, we end up with:
\begin{equation*}
 \mathcal E  \mathcal L 
= \int_{\partial\mathcal S(q)}\left[\partial_t(\boldsymbol\Phi\cdot p )+\frac{1}{2}|u_1|^2\right](u^\ast\cdot n)\,{\rm d}s.
\end{equation*}
Upon an integration by parts, we recover \eqref{lagrange} and the proof is then completed.
\end{proof}
%
%
%

Now, we observe that   $\mathcal{E}_1(q,p)$, as defined by \eqref{cococo}, can be rewritten as:
\begin{equation} \label{E1}
\mathcal{E}_1 (q,p) =\frac{1}{2} \mathcal{M}_a(q) p \cdot p,
\end{equation}
where $ \mathcal{M}_a (q)$ is  defined by \eqref{def-MAq}.
Indeed this allows us to prove the following result.
%
%
\begin{lem} \label{niou}
For any smooth curve $q(t)$ in $\mathcal Q$, for  every $p^\ast\in\mathbb R ^3$, we have:
\begin{equation} \label{full_RHS}
 \mathcal E  \mathcal L  = \mathcal{M}_a(q)q'' \cdot p^\ast + \langle \Gamma_{a}  (q),q',q'\rangle\cdot p^\ast .
\end{equation}
\end{lem}
%
%
%

%
\begin{proof}[Proof of Lemma~\ref{niou}]
Using \eqref{E1} in the definition   \eqref{def-EL} of $ \mathcal E  \mathcal L $ we have 
\begin{align*}
 \mathcal E  \mathcal L 
&= \mathcal{M}_a (q)q''  \cdot p^\ast
+  \Big( \big( D\mathcal{M}_a (q)\cdot q' \big) q' \Big) \cdot p^\ast
- \frac12 \Big( \big( D\mathcal{M}_a  (q)\cdot p^\ast \big) q'\Big)  \cdot q' .
\end{align*}
Then
\begin{align*}
 \mathcal E  \mathcal L 
= \mathcal{M}_a \, q''  \cdot p^\ast
+  \sum  (\mathcal{M}_a)_{i,j}^{k}  \,  q'_{k} q'_{j}  p^\ast_{i} 
- \frac12  \sum  (\mathcal{M}_a)_{i,j}^{k}  \,  q'_{i} q'_{j}  p^\ast_{k} ,
\end{align*}
where the sums are over $1 \leqslant i,j,k \leqslant 3$. Let us recall  the notation $(\mathcal{M}_a)_{i,j}^{k} (q)$   in \eqref{PBX}.
A symmetrization with respect to $j$ and $k$ 
of the second term and an exchange of $i$ and $k$ in the last sum of the right hand side 
 above leads to the result.
\end{proof}

Then Lemma~\ref{LEM_3} straightforwardly results from the combination of \eqref{trotriv}, Lemma \ref{LEM_1} and  Lemma  \ref{niou}. 
\end{proof}
%
%

\subsection{Reduction to an ODE in the general case.  Statement of Theorem \ref{THEO-intro}}

Now let us deal with the general case of a nonzero circulation $\gamma$. Next result,  obtained in \cite{GLS}, extends Theorem \ref{THEO-mumu} 
and  establishes a reformulation of the system in terms of an ordinary differential equation in the general case of a circulation $\gamma \in \mathbb R$.

\begin{thm} \label{THEO-intro}
Let be given the open  regular  connected and simply connected bounded cavity $ \Omega$,  the  initial closed domain $\mathcal{S}_0  \subset  \Omega$ occupied by the body, the  initial solid translation and rotation velocities $( \ell_0 , r_0 )$ in $\mathbb{R}^2 \times \mathbb{R}  $, the circulation $ \gamma$ in $ \mathbb{R}$,  and  $u_0$ the associated compatible initial fluid velocity according to Definition \ref{CompDataBd}.
There exists $F$ in $C^{\infty}(\mathcal Q \times \mathbb{R}^3 ; \mathbb{R}^3)$  depending only on
  $\mathcal S_0, \gamma$  and $ \Omega$,  and vanishing when  $\gamma =0$,
such that, up to the first collision, System~\eqref{SYS_full_system} is equivalent to  the second order ODE: 
\begin{equation} \label{ODE_intro}
 ({\mathcal M}_{g}   +\mathcal{M}_{a}(q) )  q''  + \langle \Gamma_{a}  (q),q',q'\rangle = {F} (q,q') ,
\end{equation}
with Cauchy data $q(0)=0 \in \mathcal Q,  \ q'(0)=   (\ell_0 ,r_0)\in   \mathbb R^2 \times \mathbb R  $, where  $ {\mathcal M}_{a}(q)$  and its associated a-connection  $\Gamma_{a} (q) $ 
are given by Theorem~\ref{THEO-mumu}.
For a solid position $q \in \mathcal Q$ 
the fluid velocity $u (q,\cdot)$ is uniquely determined as the solution of  a div-curl type system in the doubly-connected domain $\mathcal F(q)$, constituted of 
\eqref{E2},
\eqref{irr},
\eqref{souslab},
\eqref{souslabis},
together with the prescription of the circulation $ \gamma$. 
\end{thm}
The  local-in-time existence and uniqueness of  smooth solutions stated  in 
Theorem  \ref{thm-intro-sec2-CP} therefore simply follows from Theorem  \ref{THEO-intro} and  the Cauchy-Lipschitz theorem. 
That the  life-time of such a smooth solution can only be limited by a collision will follow from an energy argument below, cf. Section \ref{cla}.

Indeed we are going to provide a rather explicit definition of  the force term $F(q,q')$.
Let us first introduce a normalized stream function for the circulation term: for every $q \in \mathcal Q$, there exists a unique $C(q)$   in\footnote{A simple computation shows that the function $C(q)$ is actually the opposite of the inverse of the condenser capacity  of $\mathcal S(q)$ in $\Omega$.}  $\mathbb R$ such that the unique solution 
$\psi(q,\cdot)$ of the Dirichlet problem:
\begin{subequations} 
\label{def_stream}
\begin{alignat}{3}
\Delta \psi(q,\cdot)  =0  \,   \text{ in } \mathcal F(q) \, \quad 
\psi(q,\cdot)  =C(q)  \,   \text{ on } \partial \mathcal S(q)\, \quad 
\psi(q,\cdot)  =0 \,   \text{ on } \partial\Omega,
\end{alignat}
satisfies
\begin{equation} \label{circ-norma}
\int_{\partial\mathcal S(q)} \frac{\partial\psi}{\partial n} (q,\cdot) {\rm d}s=-1.
\end{equation}
\end{subequations}

Observe that for any $q \in \mathcal Q$, $ C(q) < 0 $ and that 
$C \in C^\infty   ( \mathcal Q ; (-\infty ,0))$   and depends on $ \mathcal S_0  \text{ and } \Omega$.
Eventually, we define:
\begin{align}
\label{B-def}
B(q) & :=  \int_{\partial\mathcal S(q)} \left( \frac{\partial\psi}{\partial n}  \left( \frac{\partial\boldsymbol\Phi}{\partial n}  \times 
\frac{\partial\boldsymbol\Phi}{\partial \tau} \right)  \right)(q,\cdot) \, {\rm d}s,\\
\label{E-def}
E(q) &:= - \frac{1}{2} \int_{\partial\mathcal S(q)}  
\left( \left| \frac{\partial\psi}{\partial n} \right|^2 \frac{\partial\boldsymbol\Phi}{\partial n} \right) (q,\cdot)\,  {\rm d}s ,
\end{align}
and, for  $(q,p)$  in $\mathcal Q \times \mathbb{R}^3 $,
 the force term
\begin{equation} \label{def-upsilon}
F (q,p) :=  \gamma^2 E(q) +  \gamma \, p \times B(q) .
\end{equation}

The notations $E$ and $B$ are chosen on purpose to highlight the analogy with the Lorentz force acting on a charged particle moving under the influence of a couple of electromagnetic fields $E$ and $B$.

%
%
\subsection{Proof of Theorem~\ref{THEO-intro}}
As mentioned above  \eqref{bern_2} only relies on the fact  that the fluid velocity is irrotational  and  is therefore still granted. 
However the fluid velocity  $u (q,\cdot)$  now involves an extra term due to the nonzero circulation.  
 Indeed, for any $q\in\mathcal Q$, one obtains, 
 using \eqref{def_stream} and  \eqref{Kir}, that
 the solution  $u (q,\cdot)$
 to  the div-curl type system in the doubly-connected domain $\mathcal F(q)$, constituted of 
\eqref{E2},
\eqref{irr},
\eqref{souslab},
\eqref{souslabis},
together with the prescription of  circulation $\gamma$ 
   takes the form:
\begin{equation}
\label{EQ_irrotational_flow}
u (q,\cdot) = u_1 (q,\cdot) + u_2 (q,\cdot) , 
\end{equation}
where  $u_1 (q,\cdot) $ is given by \eqref{DecompU1} as in the potential case and the new contribution $u_2 (q,\cdot) $ is defined by
\begin{equation} \label{DecompU}
u_2  (q,\cdot) :=\gamma \nabla^\perp\psi (q,\cdot) .
\end{equation}
 So besides the dependence with respect to  $\mathcal S_0$, to $\Omega$ and to the space variable, $u_2$ depends on $q$ and linearly on $\gamma$. 
Therefore plugging the decomposition  \eqref{EQ_irrotational_flow} into \eqref{bern_2} leads to
\begin{gather}
 \nonumber
m\ell'\cdot\ell^\ast+\mathcal Jr'r^\ast 
+ \int_{\mathcal F(q)} \Big( \frac{\partial u_1}{\partial t}+\frac{1}{2}\nabla|u_1|^2 \Big) \cdot u^\ast{\rm d}x
= 
- \int_{\mathcal F(q)} \big( \frac{1}{2}\nabla|u_2|^2 \big) \cdot u^\ast{\rm d}x \\
\label{bern_20bis}  \qquad-
\int_{\mathcal F(q)}\big(\frac{\partial u_2}{\partial t} + \frac{1}{2}\nabla (u_{1} \cdot u_{2} ) \big) \cdot u^\ast{\rm d}x ,
\end{gather}
for all $p^\ast :=(\ell^\ast ,r^\ast) \in \mathbb R^3$, with $u^\ast$ given by \eqref{pot*}.

By a simple  integration by parts, on obtains that the first term in the right hand side above satisfies:
\begin{equation} \label{rab}
- \int_{\mathcal F(q)}\left(\frac{1}{2}\nabla|u_2|^2\right)\cdot u^\ast{\rm d}x=\gamma^2 E(q)\cdot p^\ast,
\end{equation}
where  $E(q)$  defined in \eqref{E-def}. 

Then the reformulation of Equations (\ref{SYS_full_system}g-h) mentioned in  Theorem~\ref{THEO-intro} will follow from \eqref{bern_20bis},  \eqref{rab}, 
Lemma \ref{LEM_3} and from the following identity: 
\begin{equation} \label{last_one}
- \int_{\mathcal F(q)}\left(\frac{\partial u_2}{\partial t}+
\nabla (u_{1} \cdot u_{2} )\right)\cdot u^\ast{\rm d}x=\gamma\big(q' \times B(q)\big)\cdot p^\ast,
\end{equation}
 where $B(q)$ is defined in \eqref{B-def}. We refer to \cite{GMS} for the proof of  \eqref{last_one}.
%
%
%
\subsection{The role of the energy}
\label{cla}
An important feature of the system \eqref{ODE_intro} is that it is conservative. 
Let us denote for any $(q,p) $ in $\mathcal Q \times \mathbb R^3$, 
\begin{equation} \label{conserv}
\mathcal{E} (q,p) := \frac{1}{2}  ({\mathcal M}_{g}   + \mathcal{M}_{a}(q) ) p \cdot p - \frac{1}{2} \gamma^{2}  C (q) ,
\end{equation}
 with $C(q)$ given by \eqref{def_stream}.
 Indeed one can prove that  for any $q \in \mathcal Q$,  
 \begin{equation}
  \label{atten}
 E(q)   = \frac{1}{2} DC (q) , 
\end{equation}
where  the notation $ DC (q)$ stands for the derivative of $C(q)$ with respect to $q$, cf. Lemma $2.4$ in \cite{GMS} for a proof, so that the second term in the right-hand-side of 
 \eqref{conserv} can be seen as a potential energy related to the first  term in the right-hand-side of 
 \eqref{def-upsilon}. Observe that $\mathcal{E} (q,p)$ is the sum of two positive terms and  that in addition to its dependence on $q$ and $p$, the energy $ \mathcal{E}$ depends on $\mathcal S_0 , m ,\mathcal{J}, \gamma$ and $\Omega$. 
Next result proves that $\mathcal E(q,q')$ is indeed the natural  total kinetic energy of the ``fluid+solid'' system.
\begin{prop}
\label{equienergy}
For any $q=(h,\theta) \in C^\infty([0,T] ;\mathcal Q  )$ satisfying  \eqref{ODE_intro}, 
as far as there is no collision, 
\begin{equation*}
\mathcal E (q,q') = \frac{1}{2} \int_{\mathcal   F(q)} u(q,\cdot)^{2}   \, {\rm d}x +  \frac{1}{2} m (h')^{2} + \frac{1}{2} \mathcal  J (\theta')^{2} .
\end{equation*}
\end{prop}
\begin{proof}
First we have by integrations by parts that
\begin{gather*}
\frac{1}{2} \int_{\mathcal   F(q)} u_2^{2}   \, {\rm d}x
= - \frac{1}{2}  \gamma^2 C(q) 
  \text{ and }
\int_{\mathcal  F(q)} u_1 \cdot u_2   \, {\rm d}x  = 0.
\end{gather*}
Then we use  \eqref{E1} and the decomposition \eqref{EQ_irrotational_flow} to conclude.
\end{proof}

The following result is therefore very natural.
\begin{prop}[] \label{energy}
For any $q \in C^\infty([0,T] ;\mathcal Q  )$ satisfying  \eqref{ODE_intro}, 
as far as there is no collision, $\mathcal{E} (q,q') $ is constant in time. 
\end{prop}

\begin{proof}
Let us give a proof of Proposition \ref{energy} which  uses the ODE formulation \eqref{ODE_intro}.
We start with the observation that the energy
$\mathcal{E} (q,q') $ as defined in  \eqref{conserv} has for time derivative
\begin{equation} \label{conserv-proof1}
\big( \mathcal{E} (q,q') \big)' = ({\mathcal M}_{g}   + \mathcal{M}_{a}(q) )q'' \cdot q' 
+ \frac{1}{2} (D\mathcal{M}_a(q) \cdot q') q' \cdot q'  - \frac{1}{2} \gamma^{2}   DC (q) \cdot q'  .
\end{equation}

Now, thanks to \eqref{ODE_intro} and \eqref{def-upsilon},  we have 
\begin{equation} \label{conserv-proof2}
 ({\mathcal M}_{g}   + \mathcal{M}_{a}(q) ) q'' \cdot q' = - \langle \Gamma_{a}  (q),q',q'\rangle  \cdot q' + F(q,q')  \cdot q' ,
\end{equation}
and
\begin{equation}
\label{dima}
F (q,q')   \cdot q'=  \gamma^2 E(q)   \cdot q' .
\end{equation}
We introduce the matrix 
\begin{equation} \label{pad1}
S_a (q,q') := \left(\sum_{1\leqslant i\leqslant 3} (\Gamma_{a})^k_{i,j}(q) q'_i  \right)_{1\leqslant k,j\leqslant 3} ,
\end{equation}
so that
\begin{equation}
\label{pad2}
\langle \Gamma_{a}  (q), q', q' \rangle = S_a (q,q') q' .
\end{equation}

Combining \eqref{conserv-proof1}, \eqref{conserv-proof2}, \eqref{dima},  \eqref{pad1} and  \eqref{pad2} we obtain
\begin{equation*}
\big( \mathcal{E} (q,q') \big)'  = \gamma^2  \big(  E(q)  - \frac{1}{2} DC (q) \big)   \cdot q'
+ \big( \frac{1}{2}  D\mathcal{M}_a(q) \cdot q' - S_a (q,q') \big) q' \cdot q'   .
\end{equation*}
The first term of the right hand side vanishes thanks to \eqref{atten} and the proof of Proposition~\ref{energy} then follows  from the following result.
\begin{lem} \label{antis}
For any $(q,p) \in \mathcal Q \times \mathbb{R}^{3}$, the matrix $\frac{1}{2}  D\mathcal{M}_a(q) \cdot p - S_a (q,p)  $  is skew-symmetric.%
\end{lem}
\begin{proof}[Proof of Lemma \ref{antis}]
We start with the observation that $D\mathcal{M}_a(q) \cdot p$ is the $3 \times 3$ matrix containing the entries
\begin{equation*}
\sum_{1\leqslant k \leqslant 3} (\mathcal{M}_a)_{i,j}^{k}  (q)  \, p_{k} , \text{ for } 1\leqslant i,j\leqslant 3 .
\end{equation*}
On the other hand, the  $3 \times 3$ matrix $S_a (q,p)$ contains the entries
\begin{equation*}
\frac{1}{2} \sum_{1\leqslant k \leqslant 3} \Big(  (\mathcal{M}_a)_{i,j}^{k}  +    (\mathcal{M}_a)_{i,k}^{j}   -    (\mathcal{M}_a)_{k,j}^{i} \Big) (q) \, p_{k} ,
\end{equation*}
for $1\leqslant i,j\leqslant 3$.
Therefore, the $3 \times 3$ matrix $D\mathcal{M}_a(q) \cdot p - S_a (q,p)$  contains  the entries
\begin{equation*}
c_{ij} (q,p)  = -  \frac12 \sum_{1\leqslant k \leqslant 3}    \Big(   (\mathcal{M}_a)_{i,k}^{j}   -    (\mathcal{M}_a)_{k,j}^{i}  \Big) (q) \,  p_{k},
\end{equation*}
for $1\leqslant i,j\leqslant 3$.
Using that the matrix $\mathcal{M}_a (q)$ is symmetric, we get that $c_{ij}(q,p)  = - c_{ji} (q,p) $ for $1\leqslant i,j\leqslant 3$, which ends the proof.
\end{proof}
This ends the proof of Proposition  \ref{energy}.\footnote{It is also possible to achieve an alternative  proof  of Proposition  \ref{energy} thanks to the original PDE formulation of  the ``fluid+solid'' system, relying on the equivalence between the ODE and PDE  formulations obtained in Theorem \ref{THEO-intro}
and on the reformulation of the energy obtained in Proposition  \ref{equienergy}.}
\end{proof}

If we assume that the body stays at distance at least $\delta > 0$ from the boundary we may infer from Proposition \ref{energy}  a bound of the body velocity depending only on the data and on $\delta$.
Indeed  we have the following immediate corollary of Proposition \ref{energy} and of the regularity properties of the functions $C(q)$ and $\mathcal{M}_a(q)$.
We denote $\mathcal Q_{\delta} := \{q\in\mathbb R^3\,:\,d(\mathcal S(q),\partial\Omega)> \delta\} $.

\begin{cor}[]
\label{bd-loin}
Let
 $\mathcal S_0 \subset \Omega$,  $p_0 \in  \mathbb R^3$ and  $(\gamma ,m ,\mathcal J) \in \mathbb R \times (0,+\infty) \times (0,+\infty)$;
 $\delta > 0$;  $q \in C^\infty([0,T] ;\mathcal Q_{\delta})$ satisfying  \eqref{ODE_intro} with the Cauchy data
$(q , q')(0)= (0, p_0 )$.
Then there exists $K>0$ depending only on 
 $\mathcal S_0 ,  \Omega,  p_0 , \gamma ,m ,\mathcal J , \delta $
  such that 
$ | q' |_{\mathbb{R}^{3}}  \leqslant K$ on $[0,T]$.
\end{cor}
This entails in particular that  the life-time of a smooth solution to \eqref{ODE_intro}
can only be limited by a collision and therefore completes the proof of  Theorem \ref{thm-intro-sec2-CP}.
%
%
%
%
%
%
%

%
%
\ \par
\noindent

\subsection{Zero radius limit}
\label{sec3-VO}

Let us now turn our attention to the limit of the dynamics  when the size of the solid goes to 0 that is considering an initial domain for the body of the form 
\eqref{DomInit} with the inertia scaling  described in Definition \ref{massiveP}.  This aims to extend the analysis performed in Section  \ref{sec-zr-unbounded} to the case where the  ``fluid+solid'' system occupies a bounded domain rather than the whole plane.
 
Below, we will use the following notation: for $\mathcal S_0 \subset \Omega$;
$p_0  = ( \ell_0 , r_0  ) \in  \mathbb R^3$,   $(m ,\mathcal J) \in  (0,+\infty) \times (0,+\infty)$, 
 $\gamma $ in $ \mathbb R$ (resp. in $ \mathbb R^*$)  in the case of a massive (respectively massless) particle, 
  for every $\varepsilon \in (0,1]$ small enough to ensure that the set ${\mathcal S}_0^\varepsilon$ defined by \eqref{DomInit} satisfies
   ${\mathcal S}_0^\varepsilon \subset \Omega $, 
 we denote $(q^\varepsilon,T^\varepsilon)$   the maximal solution to  \eqref{ODE_intro} 
 associated with the coefficients $\mathcal{M}^\varepsilon$, $\Gamma_{a}^\varepsilon$ and ${F}^\varepsilon$ 
which are themselves  associated with $\mathcal S_0^\varepsilon, m^\varepsilon ,\mathcal{J}^\varepsilon $ and $\gamma $
(as $\mathcal{M}$, $\Gamma_{a}$ and $F$ were associated with $\mathcal S_0, m ,\mathcal{J} $ and $\gamma $)
where $m^\varepsilon ,\mathcal{J}^\varepsilon$ are given  in Definition \ref{massiveP}, 
and with the initial data $(q^\varepsilon , (q^\varepsilon)')(0)= (0, p_0 ).$

\begin{thm} \label{Bounded-VO}
Let $\mathcal S_0 \subset \Omega$;
$p_0  = ( \ell_0 , r_0  ) \in  \mathbb R^3$,   $(m ,\mathcal J) \in  (0,+\infty) \times (0,+\infty)$, 
 $\gamma $ in $ \mathbb R$ (resp. in $ \mathbb R^*$)  in the case of a massive (respectively massless) particle.
Let   $(h,T)$ be the maximal solution to \eqref{ODE-mass} (resp.  $h $ be the global solution to \eqref{argh}).
Then, as $\varepsilon\to 0$, 
  $\liminf T^\varepsilon\geqslant T$ (resp. $T^\varepsilon \longrightarrow +\infty$) and 
 $h^\varepsilon \relbar\joinrel\rightharpoonup h$ in $W^{2,\infty}([0,T'];\mathbb R^2)$ (resp.  in $W^{1,\infty}([0,T'];\mathbb R^2)$) weak-$\star$ for all $T'\in (0,T)$  (resp. for all $T' >0$).
 Furthermore in the case a massive particle, one also has that 
 $\varepsilon\theta^\varepsilon \relbar\joinrel\rightharpoonup 0$ in $W^{2,\infty}([0,T'];\mathbb R)$ weak-$\star$ for all $T'\in (0,T)$.
\end{thm}
In the statement above it is understood that $q^\varepsilon$ was decomposed into $q^\varepsilon = (h^\varepsilon , \theta^\varepsilon)$.
It follows from Theorem \ref{THEO-intro} that Theorem  \ref{Bounded-VO} implies Theorem  \ref{thm-intro-sec2-VO}.

Let us provide a scheme of proof of Theorem  \ref{Bounded-VO}.

\begin{proof}[Scheme of proof of Theorem  \ref{Bounded-VO}]
Using that for the inertia regimes considered in Definition \ref{massiveP} 
 the genuine inertia matrix scales as in  \eqref{genouine}, 
   the equation for  $q^\varepsilon$ reads: 
\begin{equation}
 \label{ODE_intro_epsB}
\Big(  \varepsilon^\alpha  I_\varepsilon \mathcal{M}_g I_\varepsilon +  \mathcal{M}^\varepsilon_{a}  (q^\varepsilon)\Big)   (q^\varepsilon)'' 
+  \langle \Gamma_{a}^\varepsilon  (q^\varepsilon ),(q^\varepsilon)',(q^\varepsilon)'\rangle
 =  \gamma^2 E^\varepsilon (q^\varepsilon) + \gamma (q^\varepsilon)' \times B^\varepsilon (q^\varepsilon)    ,
\end{equation}
where  the added inertia matrix $\mathcal{M}_{a}^\varepsilon$,  the  a-connection $\Gamma_{a}^\varepsilon$  and  the electric and magnetic type terms $E^\varepsilon$ and $B^\varepsilon$ are 
associated with the body of size $\varepsilon$ as mentioned above. 
Here one crucial step in passing to the limit in \eqref{ODE_intro_epsB} is to find some uniform bounds in $\varepsilon$. 
The energy is of course a natural candidate in order to get such estimates. In particular one may turn toward 
an appropriate modification of Corollary \ref{bd-loin} in the zero radius limit.
A difficulty is  that the potential part of the energy (corresponding to the second term in \eqref{conserv})  diverges logarithmically as $\varepsilon \rightarrow 0^{+}$.
However such a contribution can be discarded from the energy conservation 
 since it does not depend on the solid position and velocity. 
 Indeed  an appropriate renormalization of the energy provides an uniform estimate of  $\varepsilon^{\min(1,\frac{ \alpha }{2 })} \, | (h^\varepsilon)' , \varepsilon (\theta^\varepsilon)' )    | _{\mathbb{R}^{3}} $ at least till the solid stays away from the external boundary. 
Unfortunately  in the massless case the coefficient $ \alpha$ satisfies $ \alpha >0$ and the previous estimate is not sufficient.\footnote{Indeed the case where $\alpha \geq 2$ is the most delicate and we will focus on it.}

One then turns toward the search for an asymptotic normal form of  \eqref{ODE_intro_epsB} with the hope that more structure shows up in the zero radius limit and reveals another candidate in order to obtain some uniform bounds in $\varepsilon$. 
In order to do so we first establish some expansions in the limit $\varepsilon \rightarrow 0$ 
of  $\mathcal{M}_{a}^\varepsilon$,  $\Gamma_{a}^\varepsilon$, $E^\varepsilon$ and $B^\varepsilon$.
These expansions are obtained by a multi-scale analysis of the Kirchhoff potentials and of the stream functions and repeated use of Lamb's lemma, that is Lemma \ref{lamb}.
More precisely these expansions involve two scales corresponding
respectively to variations over length $O (1)$ and $O (\varepsilon )$ respectively on $\partial \Omega$ and $\partial {\mathcal S}^\varepsilon (q)$.
The profiles appearing in these expansions are obtained by successive corrections, considering alternatively at their respective scales the body boundary from which the external boundary seems remote and the external boundary from which the body seems tiny, so that good approximations are given respectively by the case without external boundary and without the body. 
We refer to  \cite{GMS}  for more details on this intricate process  and sum up the results below. 
The leading term of the expansions of  $\mathcal{M}_{a}^\varepsilon $, $\Gamma_{a}^\varepsilon$ and $B^\varepsilon$  in the zero-radius limit are given, up to an appropriate scaling, by the terms obtained in the case where the rigid body is of size $\varepsilon = 1$ and is immersed in a fluid filling the whole plane, that is in the case tackled in Theorem \ref{pasdenom}. On the other hand the  leading term of the expansion of $E^\varepsilon$ in the zero-radius limit is given, up to an appropriate scaling, by the field 
\begin{equation} 
\label{tical}
\mathsf{E}_{{\it 0}} (q) :=
-
\begin{pmatrix}
  u_\Omega (h)^{\perp} \\
 u_\Omega (h) \cdot  R(\theta) \xi
\end{pmatrix} ,    \text{ where  }  q=(h,\theta ).
\end{equation}
We recall that $u_\Omega$ and $\xi$ were defined respectively above \eqref{ODE-mass} and  in  \eqref{DefXi}. 
 Given $\delta > 0$ and $\varepsilon_{0} $ in $(0,1)$, we define the bundle of shrinking body positions at distance $\delta$ from the boundary for a radius of order $\varepsilon$ with 
 $0<\varepsilon<\varepsilon_0$: 
$$\mathfrak Q_{\delta,\varepsilon_{0}} :=  \{(\varepsilon, q) \in (0,\varepsilon_{0}) \times \mathbb R^3 \ /\  d(\mathcal S^{\varepsilon}(q),\partial\Omega)> \delta \}  .$$
\begin{prop} \label{dev-added}
Let  $\delta > 0$. There exists $\varepsilon_{0} $ in $(0,1)$,
$\mathsf{E}_{{\it 1}} (q) $ and $\mathsf{B}_{{\it 1}} (q) $ in $ L^\infty ( \mathcal Q_{\delta} ; \mathbb{R}^{3})$, 
  $\mathcal{M}_{r}$ in $ L^\infty (\mathfrak Q_{\delta,\varepsilon_{0}} ; \mathbb{R}^{3 \times 3})$,
 $\Gamma_{r} $ in $ L^\infty (\mathfrak Q_{\delta,\varepsilon_{0}} ;  \mathcal{BL} (\mathbb{R}^3 \times \mathbb{R}^3 ; \mathbb{R}^3 ))$,
 and   $E_{r} $ and $B_{r}$ in  $L^\infty (\mathfrak Q_{\delta,\varepsilon_{0}}; \mathbb{R}^3)$ 
 such that, for all $(\varepsilon ,q ) $ in $\mathfrak Q_{\delta,\varepsilon_{0}}$,  with $q=(h,\theta )$, 
\begin{gather} \label{dev-mat-ii-proof}
\mathcal{M}_{a}^\varepsilon (q) = 
\varepsilon^2 I_\varepsilon  
\Big( \mathcal{M}_{a, \theta} 
+ \varepsilon^2 \mathcal{M}_{r} (\varepsilon ,q )
\Big) I_\varepsilon  ,
\\  \label{exp-GammaS}
\langle \Gamma_{a}^\varepsilon (q),\cdot,\cdot\rangle 
=  \varepsilon I_\varepsilon \big(\langle \Gamma_{a,\theta} , I_{\varepsilon} \cdot , I_{\varepsilon} \cdot \rangle
 + \varepsilon^{2} \langle \Gamma_{r} (\varepsilon,q), I_{\varepsilon} \cdot , I_{\varepsilon} \cdot \rangle  \big), 
 \\ \label{E-dev} E^\varepsilon(q) =  I_\varepsilon  \Big(   \mathsf{E}_{{\it 0}} (q) + \varepsilon \mathsf{E}_{{\it 1}} (q) + \varepsilon^2  E_{r}  (\varepsilon,q)  \Big) ,
\\ \label{B-dev} B^\varepsilon  (q) =  \varepsilon I_\varepsilon^{-1}   \Big( B_{\theta} + \varepsilon \mathsf{B}_{{\it 1}} (q) + \varepsilon^2 B_{r} (\varepsilon,q)  \Big).
\end{gather}
\end{prop}
We recall that ${\mathcal M}_{a, \theta} $ and 
$B_{ \theta}$ are given by Theorem \ref{pasdenom} as associated with the rigid body of size $\varepsilon = 1$ and as if the body was immersed in a fluid filling the whole plane,
 $ \Gamma_{a,\theta}$ denotes the  a-connection associated with ${\mathcal M}_{a, \theta}$,
  $I_{\varepsilon}$ is the diagonal matrix 
$I_{\varepsilon} :=     \text{diag } (1,1,\varepsilon)$ and let us avoid any confusion by highlighting that the $\cdot$ in  \eqref{exp-GammaS} stands for the application to any   ${p} $ in $\mathbb{R}^3$ (which determines completely the bilinear symmetric mapping). Let us also recall that quite explicit expressions of  ${\mathcal M}_{a, \theta} $,
$B_{ \theta}$  and   $ \Gamma_{a,\theta}$ are given in  Section \ref{sec-ex}.

Therefore, using \eqref{E-dev}, \eqref{B-dev}, \eqref{tical} and \eqref{DefB}, one obtains that the leading part of the expansion of  the right hand side 
 of \eqref{ODE_intro_epsB} is 
\begin{equation} \label{EB1}
I_\varepsilon  \Big(\gamma^2 \mathsf{E}_{{\it 0}}  (q^\varepsilon) + \gamma \big( I_{\varepsilon} (q^{\varepsilon})' \big) \times B_{\theta^\varepsilon}   \Big)
=  I_\varepsilon  \Big(\gamma {\hat{p}}^\varepsilon \times B_{\theta^\varepsilon}   \Big),
\end{equation}
for any ${\hat{p}}^\varepsilon$  of the form 
\begin{equation} \label{smodu}
{\hat{p}}^\varepsilon = (  (h^{\varepsilon})'  - \gamma  u_\Omega (h^\varepsilon ) ,\varepsilon (\theta^\varepsilon)' )^{t} + \eta  B_{\theta^\varepsilon} ,
\end{equation}
with $ \eta $ in $\mathbb R$.
\paragraph{An instructive digression.} The identities \eqref{EB1} and \eqref{smodu}
 remind a well-known modulation strategy  used by Berkowitz and Gardner, cf. \cite{BG}, in order the tackle the zero-mass limit of the following dynamics of a light particle in a smooth electro-magnetic field:
\begin{equation}
  \label{GR0}
\varepsilon^2 q''   =  E(q) +   q' \times B(q)    \text{ with the condition  } E(q)  \cdot B(q)  = 0 .
\end{equation}
Here we have dropped the index $\varepsilon$ of $q$ for sake of clarity and we will assume that the fields $ E(q)$ and $B(q) $ (which actually stand here for electric and magnetic fields) smoothly depend on its argument  $q$ but not on $\varepsilon$ otherwise. 
The setting of \cite{BG} is slightly more general but the toy-system above will be sufficient for the exposition of the gain obtained by modulation  in the analysis. 
The starting point is that a naive application of the Cauchy-Lipshitz theorem does only provide existence of a solution over a time which may vanish as $\varepsilon$ converges to $0$. 
The difficulty resides within the lack of sign or structure of the $E(q)$ term which prevents from obtaining straightforwardly some uniform estimates by energy. 
To overcome this difficulty Berkowitz and Gardner introduced the modulated variable:
\begin{equation}
  \label{GR1}
 \tilde{p} = q' - u(q)    \text{ where  }   u(q)   \text{ satisfies } E(q) + u(q)  \times B(q) = 0.
\end{equation}
Observe that the existence for any $q$ of such a vector $u(q)$ is guaranteed by the condition $E(q)  \cdot B(q)  = 0$ and that the set of such vectors is an one-dimensional  affine space. 
Indeed in \cite{BG} Berkowitz and Gardner makes use of the following explicit  field 
\begin{equation}
  \label{explicit}
 u(q)  := \vert B(q)  \vert^{-2}  \,E(q)  \times B(q) ,
\end{equation}
 which  satisfies the condition in   \eqref{GR1} and which 
turns to be the actual physical drift velocity for this system.

Using the chain-rule, one obtains $ \tilde{p}' = q'' - q' \cdot \nabla u(q)$,
and then, by using  \eqref{GR0} and \eqref{GR1},
\begin{align*}
\varepsilon^2  \tilde{p}'    =   E(q) +   q' \times B(q) -   \varepsilon^2 q' \cdot \nabla u(q) =    \tilde{p}  \times B(q) -   \varepsilon^2 ( \tilde{p}  + u(q) ) \cdot \nabla u(q) .
\end{align*}
Therefore, one obtains the following gyroscopic normal form: 
\begin{equation}
  \label{GR2}
  \tilde{p}'  = \frac{ 1 }{\varepsilon^2  }    \, \tilde{p}  \times B(q) -   (  \tilde{p}  + u(q) ) \cdot \nabla u(q)
  \end{equation}
Now that the $ E(q)$ has been absorbed by the choice of the modulated variable, the only factor with a  singular (i.e. negative) power of $\varepsilon$ is in front of the  $B(q)$ term and this term disappears when taking the inner product of  \eqref{GR2} with $p$ in an energy-type estimate. 
Some Gronwall estimates on  \eqref{GR2} and   \eqref{GR1} then provide uniforms bound of $q$ and $q'$. 
In particular the second Gronwall estimate allows to estimate $q$ and $q'$ from $p$ thanks to  \eqref{GR1} and therefore 
relies on the fact that the modulation $ u(q)$ involves one less time derivative than $q'$.

\
\par
\
Let us now go back to our search for an asymptotic normal form of  \eqref{ODE_intro_epsB}  and let see how to extend the analysis performed above. 
We first observe that the drift velocity (naively) computed as in \eqref{explicit} with $ \mathsf{E}_{{\it 0}} (q^\varepsilon)$ and $  B_{\theta^\varepsilon} $  instead of  $E(q)$ and $B(q)$
corresponds to a nonzero $\eta$ in  \eqref{smodu}.
Still in the case of  \eqref{EB1} one observes as we already did in the  proof of Theorem   \ref{pasdenom-VO}
 that the natural counterpart to $(h^\varepsilon)'  $ for what concerns the angular velocity is rather $\varepsilon (\theta^\varepsilon)'  $ than $(\theta^\varepsilon)' $. 
Moreover we will benefit from the fact that the contribution due to $\varepsilon (\theta^\varepsilon)'  $  in the first two coordinates of the result of the cross product in the right hand side of  \eqref{EB1} provides the term \eqref{simplemaiscostaud} whose special structure somehow allows to gain one factor $\varepsilon$.
It turns out that  the leading part of the relevant drift velocity in order to pass to the limit in \eqref{ODE_intro_epsB}  
is given by  \eqref{smodu} with  $\eta=0$, that is by 
\begin{equation}
  \label{mod0}
{\hat{p}}^\varepsilon = (  (h^{\varepsilon})'  - \gamma  u_\Omega (h^\varepsilon ) ,\varepsilon (\theta^\varepsilon)' )^{t} .
\end{equation}

Still the leading terms of the inertia matrix $  \varepsilon^\alpha  I_\varepsilon \mathcal{M}_g I_\varepsilon +  \mathcal{M}^\varepsilon_{a}  (q^\varepsilon) $ in front of $(q^\varepsilon)'' $ in \eqref{ODE_intro_epsB} is\footnote{Observe that one recovers the same inertia for the leading terms (for which the hierarchy depends on whether $\alpha\geq 2 $ or $\alpha\leq 2$) than in Section  \ref{sec-zr-unbounded} for the case where the  ``fluid+solid'' system occupies the full plane, cf.  \eqref{wemerde}.} 
$ I_\varepsilon  (  \varepsilon^\alpha \mathcal{M}_g + \varepsilon^2  \mathcal{M}_{a, \theta}   )   I_\varepsilon $, and therefore,  in order to cover the case where $\alpha  \geq 2$, 
 one has to investigate further the structure of the other terms of the equation  \eqref{ODE_intro_epsB},  and to hope that a recombination as nice  as in \eqref{EB1} occurs at the next order.
This is actually  why we had expanded up to order $\varepsilon^2 $ in Proposition  \ref{dev-added}.
One observes in particular from  \eqref{exp-GammaS} that 
at order $\varepsilon$ the a-connection $\Gamma_{a}^\varepsilon$ comes into play. 
Indeed combining the previous expansions of $\Gamma_{a}^\varepsilon$, $E^\varepsilon$ and $B^\varepsilon$
one obtains 
\begin{gather} \label{mesaoulent}
\gamma^2 E^\varepsilon (q^\varepsilon) + \gamma (q^\varepsilon)' \times B^\varepsilon (q^\varepsilon)  
 - \langle\Gamma_{a}^\varepsilon  (q^\varepsilon), (q^{\varepsilon})', (q^{\varepsilon})'\rangle 
  \\  \nonumber = I_\varepsilon  \Big[ \gamma {\hat{p}}^\varepsilon \times B_{\theta^\varepsilon}  
 \\  \nonumber + \varepsilon \Big(  \gamma^2 \mathsf{E}_{{\it 1}}  (q^\varepsilon) +\gamma  I_{\varepsilon} (q^{\varepsilon})' \times \mathsf{B}_{{\it 1}}  (q^\varepsilon)  
- \langle \Gamma_{a,\theta^\varepsilon} ,  I_{\varepsilon} (q^{\varepsilon})' ,  I_{\varepsilon} (q^{\varepsilon})' \rangle \Big) 
 + O(\varepsilon^2 ) \Big], 
\end{gather}
Above and thereafter  the notation $O(\varepsilon^2) $ holds for a term of the form $ \varepsilon^2 F(\varepsilon, q^{\varepsilon} , {\hat{p}}^\varepsilon )$ where 
 $F$ is  a vector field  which is weakly nonlinear in the sense that 
there exists  $\delta>0$,  $\varepsilon_0\in(0,1)$ and 
 $K>0$  
such that for any $(\varepsilon, q, p) $ in $\mathfrak Q_{\delta,\varepsilon_{0}} \times \mathbb R^3 $, 
$| F  (\varepsilon , q, p) |_{\mathbb R^{3}} \leq K ( 1 + | p|_{\mathbb R^{3}}  + \varepsilon | p|_{\mathbb R^{3}}^{2} ) $.
 Indeed the way \eqref{mesaoulent} has to be understood  is even more intricate because among the terms hidden in the $ O(\varepsilon^2 )$ there is a term for which one obtains such an order only when performing a Gronwall estimate for an energy-type method. More precisely one term abusively included in   the notation $O(\varepsilon^2) $ in \eqref{mesaoulent}  
  is of the form  $ O(\varepsilon ) F(q^\varepsilon)$, where  $F$  is  a vector field  in $C^{\infty}(\mathbb R \times \Omega ; \mathbb R^3)$  weakly gyroscopic in the sense that for any  $\delta>0$ and $\varepsilon_0\in (0,1)$  there exists $K>0$ depending on  $\mathcal S_0$, $\Omega$, $\gamma $ and $\delta$ such that for any smooth curve $q(t) = ( h(t),\theta (t))$ in $ \{x \in  \Omega \ / \  d (x, \partial \Omega ) > \delta \} \times \mathbb R$,  we have, for any $t\geq 0$ and any $\varepsilon \in (0,\varepsilon_{0})$,
$| \int_0^t \, 
\tilde{p} \cdot F (q) |
 \leq \varepsilon K  ( 1 + t +  \int_0^t \, |  \tilde{p}|^{2}_{ \mathbb R^3} ) $, with $\tilde{p}= (h' - \gamma u_\Omega (h), \varepsilon \theta')^t$.

A striking and crucial phenomenon  is that some subprincipal contributions (that is, of order  $\varepsilon$)  of the right hand side of  \eqref{mesaoulent} can be gathered  into an a-connection term involving the bilinear mapping $\Gamma_{a}^\varepsilon$  obtained in the case where the rigid body is of size $\varepsilon = 1$ and is immersed in a fluid filling the whole plane, but 
applied to the modulated variable as follows\footnote{As for \eqref{EB1}, this relation is algebraic, in the sense that it does not rely on the fact that $q^\varepsilon$ satisfies \eqref{ODE_intro}.}:
\begin{multline} \label{grav}
\gamma^2 \mathsf{E}_{{\it 1}}  (q^\varepsilon) + \gamma  I_{\varepsilon} (q^{\varepsilon})'  \times {\mathsf{B}}_{{\it 1}}  (q^\varepsilon)
- \langle \Gamma_{a,\theta^\varepsilon},  I_{\varepsilon} (q^{\varepsilon})' ,  I_{\varepsilon} (q^{\varepsilon})' \rangle  
 \\ =    - \langle \Gamma_{a,\theta^\varepsilon}, \hat{p}^\varepsilon, \hat{p}^{\varepsilon} \rangle +  \gamma  (u_{c} (q^{\varepsilon}),0)^t \times   B_{\theta^\varepsilon} 
+ O (\varepsilon), 
\end{multline}
where ${\hat{p}}^\varepsilon $ is given by  \eqref{mod0},
$ u_c $ is a smooth vector field on  $\mathcal Q$ with values in $\mathbb{R}^2$ which depends on $\Omega$ and $ \mathcal S_0$.
 Indeed a quite explicit expression can be given by 
$ u_c :=\nabla^\perp_{h} \big(  D_h \psi_{\Omega} (h) \cdot R(\theta) \xi  \big) $, where  $D_h$ denotes the derivative with respect to $h$.
We refer here again to  \cite{GMS}  for a proof of  \eqref{grav}; it relies on  explicit computations of the profiles $\mathsf{E}_{{\it 1}} (q) $ and $\mathsf{B}_{{\it 1}} (q) $ thanks to geometric quantities and some tedious algebraic computations. 

Next the second term in the right hand side of  \eqref{grav}  can be absorbed by the principal term in the right hand side of  \eqref{mesaoulent} 
up to a modification of size $\varepsilon$ of the arguments that is, thanks to 
 the following second order modulation: 
\begin{equation}  \label{dmodu}
{\tilde{p}}_{\varepsilon} := \big( h_\varepsilon'  - \gamma [ u_\Omega (h_\varepsilon ) + \varepsilon   u_c (q_\varepsilon) ], \varepsilon \vartheta_\varepsilon'  \big)^t .
\end{equation}
Observe also that, as long as the solid does not touch the boundary, the  drift term in the velocity of the center of mass
is bounded. Indeed one may easily proves that there exists $\delta > 0$,  $\varepsilon_{0} $ in $(0,1)$ and $K>0$ such that for any $(\varepsilon ,q) $ in $ \mathfrak Q_{\delta,\varepsilon_{0}} $
with $q=(h,\theta)$, $| u_\Omega (h) + \varepsilon   u_c (q) |_{\mathbb R^{3}} \leq K$.

Thus we deduce from   \eqref{mesaoulent}  and \eqref{grav} that 
\begin{multline} \label{AUX-tot}
\gamma^2 E^\varepsilon (q^\varepsilon) + \gamma (q^\varepsilon)' \times B^\varepsilon (q^\varepsilon)  
- \langle \Gamma_{a}^\varepsilon  (q^\varepsilon)  , (q^\varepsilon)', (q^\varepsilon)' \rangle \\
= I_\varepsilon  \Big[  \gamma \tilde{p}^\varepsilon     \times B_{\theta^\varepsilon} 
- \varepsilon \langle \Gamma_{a,\theta^\varepsilon}, \tilde{p}^\varepsilon, \tilde{p}^\varepsilon \rangle
+ O(\varepsilon^2 )  \Big],
\end{multline}
Using now \eqref{dev-mat-ii-proof}  and a few further tedious manipulations  
 the equation \eqref{ODE_intro_epsB}  
can now be recast into the following geodesic-gyroscopic normal form: 
\begin{multline}
\label{fnorm2}
\Big( \varepsilon^{\alpha}\,  M_g + \varepsilon^2 \, M_{a, \theta_\varepsilon} \Big) \tilde{p}_\varepsilon '
+ \varepsilon \langle \Gamma_{a,\theta_\varepsilon} ,{\tilde{p}}_\varepsilon,{\tilde{p}}_\varepsilon\rangle 
 = \gamma \tilde{p}^\varepsilon     \times B_{\theta^\varepsilon} 
 + O(\varepsilon^{\min(2,\alpha)} )   .
\end{multline}
Observe how \eqref{fnorm2}  is close to \eqref{ODE_ext-eps}: the only two differences are the modulation of $p^\varepsilon$ into $\tilde{p}^\varepsilon$ and the remainder $O(\varepsilon^{\min(2,\alpha)} )$, which actually suffers from the same abuse of notation than the term $O(\varepsilon^2)$
described below \eqref{mesaoulent}.
At least till the solid stays away from the external boundary one may take advantage of 
this normal form to obtain an estimate of the modulated energy 
$$\frac12 \Big( \varepsilon^{\alpha}\,  M_g + \varepsilon^2 \, M_{a, \theta_\varepsilon} \Big) \tilde{p}_\varepsilon \cdot  \tilde{p}_\varepsilon  ,$$
thanks to a Gronwall estimate.
This provides uniform bounds of  $ | ( (h^\varepsilon)' , \varepsilon (\theta^\varepsilon)' )|_{\mathbb{R}^{3}} $.  This estimate in turn allows to pass to the limit proceeding as in the proof of Theorem \ref{pasdenom-VO}.
The issue of a possible collision is then tackled in a bootstrapping argument thanks to the behavior of the limit systems.

More precisely we first  prove that  the lifetime $T^{\varepsilon}$ of the solution $q^\varepsilon$, which can be only limited by a possible encounter between the solid and the boundary $\partial \Omega$, satisfies the following: 
there exist $\varepsilon_{0}>0$, $\underline{T}>0$ and $\underline{\delta} >0$, such that for any $\varepsilon $ in $(0,\varepsilon_{0})$, we have
$T^{\varepsilon} \geq \underline{T}$ and moreovoer on $ [0,\underline{T} ]$, one has 
$(\varepsilon,q^\varepsilon )  \in  \mathfrak Q_{\underline{\delta}  ,\varepsilon_{0}} $.

Then, using again the uniform estimates obtained thanks to the asymptotic geodesic-gyroscopic normal form \eqref{fnorm2} one 
  establishes the desired convergence on any time interval during which we have a minimal distance between ${\mathcal S}_{\varepsilon}(q)$ and $\partial \Omega$, uniform for small $\varepsilon$.
This  consists in passing to the weak limit, with the help of all a priori bounds, in the two first components of each term of \eqref{fnorm2}.
 It finally only remains to extend the time interval on which the above convergences are valid to any time interval.
\end{proof}

 
\section{Case of an unbounded  flow with vorticity}
\label{chap3}

In this section we investigate the case of a rigid body immersed in an unbounded  flow with vorticity. 

\subsection{Statement of a theorem \`a la Yudovich in the body frame}
For the Cauchy problem  it is more convenient to consider the body frame which does not depend on time, as we did in Section \ref{chap1}.
We will therefore start back from the equations  \eqref{Euler11}-\eqref{Solide11}. 
In the sequel we will use an abuse of notation and still denote by $\omega$ the vorticity  in the body frame given by 
$\omega(t,x) :=  \operatorname{curl} v(t,x).$
Taking the curl of the equation \eqref{Euler11} we get
\begin{equation}
\label{vorty}
\partial_t  \omega + \left[(v-\ell-r x^\perp)\cdot\nabla\right]  \omega =0 \text{ for }
x \in \mathcal{F}_{0} .
\end{equation}
Due to the equation of vorticity \eqref{vorty} 
the following quantities are conserved as time proceeds, at least for smooth solutions: for any $t>0$, for any $p $ in $\left[ 1,+\infty \right]$,
\begin{eqnarray} \label{ConsOmega}
\| \omega  (t, \cdot) \|_{L^{p}(\mathcal{F}_{0})} 
= \| \omega_{0} \|_{ L^{p}( \mathcal{F}_{0} )}  .
\end{eqnarray}
In the case of a fluid alone, the conservation laws \eqref{ConsOmega} allowed Yudovich, and  DiPerna and Majda to construct some global-in-time solutions of the $2$d Euler equations in the case of a velocity with finite local energy and  $L^{p}$ initial vorticity, with $p>1$. In the case $p= +\infty$
Yudovich  also obtained a uniqueness result using in particular that the corresponding fluid velocity is in the space  $\mathcal{LL}({\mathcal F}_0)$ of log-Lipschitz $\mathbb R^2$-valued vector fields on ${\mathcal F}_0$, that is the set of functions $f \in L^{\infty}({\mathcal F}_0)$ such that
\begin{equation} \label{DefLL}
\| f \|_{\mathcal{LL}({\mathcal F}_0)} := \| f\|_{L^{\infty}({\mathcal F}_0)} + \sup_{x\not = y} \frac{|f(x)-f(y)|}{|(x-y)(1+ \ln^{-}|x-y|)|} < +\infty.
\end{equation}

These results can be adapted to the case where there is a rigid body. 
In theses notes we will focus on  a result  of global in time existence and uniqueness similar to the celebrated result by 
Yudovich about a fluid alone. 

Let us first give a global weak formulation of the problem by considering (for the solution as well as for test functions) a velocity field on the whole plane, with the constraint to be rigid on $ \mathcal{S}_{0} $. We introduce the following space
$$\mathcal{H}: = \left\{\Psi\in L^2_{loc} (\mathbb{R}^2)  \, \Big/ \,   \operatorname{div} \Psi =0 \, \text{ in } \, \mathbb{R}^2 \, \text{ and } \
D \Psi =0 \, \text{ in } \, \mathcal{S}_0 \right\},$$
  where $ D\Psi := \nabla \Psi + (\nabla \Psi)^{T} $.
It is classical that the space $\mathcal{H}  $ can be recast thanks to the property:
 \begin{eqnarray}
  \label{EEE} 
\exists (\ell_\Psi , r_\Psi ) \in \mathbb{R}^2 \times \mathbb{R},  \  \forall x \in  \mathcal{S}_0 , \
\Psi  (x) = \ell_\Psi +  r_\Psi x^\perp .
\end{eqnarray}
More precisely, $\mathcal{H} = \left\{\Psi\in L^2_{loc}  (\mathbb{R}^2)  \ \Big/ \   \operatorname{div} \Psi =0 \ \text{ in } \ \mathbb{R}^2 
  \text{ and satisfies }   \eqref{EEE}  \right\} $,  and the ordered pair $(\ell_\Psi , r_\Psi )$ above is unique. 
Let us also introduce
$$\tilde{\mathcal{H}}:= \left\{ \Psi \in \mathcal{H} \ \Big/ \ \Psi |_{ \overline{\mathcal{F}_{0}}} \in C^1_c ( \overline{\mathcal{F}_{0}}) \right\} ,$$
where $\Psi |_{ \overline{\mathcal{F}_{0}}}$ denotes the restriction of $\Psi$ to the closure of the fluid domain. We also introduce for $T>0$, 
$\tilde{\mathcal{H} }_{T} := C^{1}([0,T];\tilde{{\mathcal H}}).$
When $(\overline{u},\overline{v}) \in \mathcal{H} \times \tilde{\mathcal{H} }$, we denote by 
\begin{equation*}
<\overline{u},\overline{v} >
:= m \, \ell_u \, \cdot \ell_v + \mathcal{J} r_u \,  r_v +  \int_{\mathcal{F}_{0} } u \cdot v \, dx ,
\end{equation*}
where we use the notations $u$ and $v$ for the restrictions of $\overline{u}$ and $\overline{v} $ to $\overline{\mathcal{F}_{0}}$.
Our definition of a weak solution is the following.
\begin{defn}[Weak Solution] \label{DefWS}
Let us be given  $ \overline{v}_0 \in \mathcal{H}$ and $T>0$.
We say that  $ \overline{v} \in C ([0,T]; \mathcal{H}-w)$ is a weak solution to \eqref{Euler11}--\eqref{Solide1ci} in $[0,T]$ if for any test function $\Psi \in \tilde{\mathcal{H} }_{T}$,
\begin{multline} \label{FormulationFaible}
<  \Psi(T,\cdot ) , \overline{v}(T,\cdot) >
- <  \Psi (0,\cdot ) , \overline{v}_0 >_\rho
= \int_0^T < \frac{\partial \Psi}{\partial t},\overline{v} > \, dt \\ 
+ \int_0^T \!\! \int_{\mathcal{F}_0} v\cdot\left(\left(v-\ell_v  - r_v x^\perp  \right)\cdot\nabla\right)\Psi \, dx\, dt
- \int_0^T \!\! \int_{\mathcal{F}_0} r_v\,  v^\perp   \cdot \Psi \, dx \, dt \\ 
-  \int_0^T    m r_v \, \ell_v^\perp   \cdot  \ell_{\Psi}   \, dt .
\end{multline}
We say that $ \overline{v} \in C ([0,+\infty); \mathcal{H}-w)$ is a weak solution to \eqref{Euler11}--\eqref{Solide1ci} in $[0,+\infty)$ if it satisfies \eqref{FormulationFaible} for all $T>0$.
\end{defn}
Definition \ref{DefWS} is legitimate since a classical solution to \eqref{Euler11}--\eqref{Solide1ci} in $[0,T]$ is also a weak solution. This follows easily from an integration by parts in space which provides the identity on $[0,T]$:
\begin{equation} \label{SemiW}
<\partial_{t} \overline{v} , \Psi  > 
= \int_{\mathcal{F}_0} v \cdot\left(\left(v- \ell_v - r_v x^\perp \right)\cdot\nabla\right)\Psi \, dx
- \int_{\mathcal{F}_0} r_v \, v^\perp   \cdot \Psi  \, dx -  m r_v \,  \ell_v^\perp   \cdot  \ell_{\Psi} ,
\end{equation}
and then from an integration by parts in time.

In the sequel we will often drop the index of $\ell_v$ and $r_v$ and we will therefore rather write $\ell$ and $r$.
We will equivalently say that $(\ell,r,v)$ is a weak solution to \eqref{Euler11}--\eqref{Solide1ci}.

One has the following result of existence of weak solutions for the above system, the initial position of the solid being given.

\begin{thm} \label{ThmYudo0}
For any   $ \overline{v}_0 \in \mathcal{H}$ such that the  restriction  of 
$\operatorname{curl} \overline{v}_0 $ to $\overline{{\mathcal F}_0}$ is in $ L_c^{\infty}(\overline{{\mathcal F}_0})$, 
there exists a unique  weak solution  $ \overline{v} \in C ([0,+\infty); \mathcal{H}-w)$ to \eqref{Euler11}--\eqref{Solide1ci} in $[0,+\infty)$.
Moreover 
$(\ell ,r)$ is in $C^1 (\mathbb{R}^+; \mathbb{R}^2 \times \mathbb{R})$, $v$ is in 
$L^{\infty}(\mathbb{R}^+ ; \mathcal{LL}({\mathcal F}_0))$ and  $\operatorname{curl} v $ is in $ L^{\infty}(\mathbb{R}^+ ;  L^\infty_c(\overline{{\mathcal F}_0}))$.
\end{thm}

Going back to the original frame Theorem \ref{ThmYudo0} implies Theorem \ref{ThmYudo-intro}. 
Regarding the initial data, let us observe that 
with any $(\ell_0,r_0) \in \mathbb{R}^2 \times \mathbb{R}$, $\omega_0  \in L_c^{\infty}(\overline{{\mathcal F}_0})$, 
one may associate  $ \overline{v}_0 \in \mathcal{H}$ by setting 
$ \overline{v}_0 =  \ell_0 +  r_0 x^\perp$ in $\mathcal{S}_{0}$ and $ \overline{v}_0 =   u_0 $, where  $u_0$ is the compatible initial velocity associated with  $\ell_0$, $r_0$ and $\omega_0$ by  Definition \ref{CompDataY}.

\subsection{Proof of Theorem \ref{ThmYudo0}}
\label{sec-coche}

In order to take into account the velocity contribution due to the vorticity we consider  the Green's function  $G (x,y)$ of $\mathcal{F}_{0}$ with Dirichlet boundary conditions. 
We also introduce the function $K (x,y)=\nabla^\perp G (x,y)$ known as the kernel of the Biot-Savart operator  $K [\omega]$ which therefore acts on $\omega \in L^\infty_c (\overline{\mathcal{F}_{0}})$  through the formula 
\begin{equation*}
  K[\omega](x)= \int_{\mathcal{F}_{0}}K (x,y) \omega(y) \, dy.
\end{equation*}
It is classical that  $K[\omega]$ is
 in $\mathcal{LL}(\mathcal{F}_{0})$, divergence-free, tangent to the boundary and satisfies $\operatorname{curl} K[\omega]=\omega$ and 
$K[\omega](x) = {\mathcal O}\left( |x|^{-2}\right) $ as $ x \rightarrow \infty$,
(so that it  is  square-integrable), and its circulation around $\partial \mathcal{S}_0$
is given by $ \int_{\partial \mathcal{S}_0  }   K [ \omega] \cdot {\tau} \, ds = -  \int_{\mathcal{F}_{0}  }  \omega \, dx .$
Then,   given $\omega$ in $L^\infty_c (\mathcal{F}_{0})$, $\ell$ in $\mathbb{R}^2$, $r$ and $\gamma $ in  $\mathbb{R}$, 
 there is a unique solution $v$ in $\mathcal{LL}({\mathcal F}_{0})$  to the following system:
\begin{gather*}
 \operatorname{div} v = 0   \text{ and }  \operatorname{curl} v  = \omega   \quad   \text{for}  \ x\in  \mathcal{F}_{0} , \\
 v \cdot n = \left(\ell+r x^\perp\right)\cdot n  \quad   \text{for}  \ x\in \partial \mathcal{S}_0    \text{ and }  \int_{ \partial \mathcal{S}_0} v  \cdot  \tau \, ds=  \gamma , \\
 v \longrightarrow 0 \quad  \text{as}  \ x \rightarrow  \infty .
\end{gather*}
Moreover $v$ is given by 
\begin{gather} \label{allo2}
v = \tilde{v} + \beta H  ,   \\ \nonumber \text{ with }  \tilde{v} := K [\omega ]  + \ell_1 \nabla \Phi_1 + \ell_2 \nabla \Phi_2
+ r \nabla \Phi_3    \text{ and  }   \beta := \gamma +   \int_{\mathcal{F}_{0}  }  \omega \, dx .
\end{gather}
We start with looking for some  a priori estimates that is to some estimates satisfied by  smooth solutions  to \eqref{Euler11}--\eqref{Solide1ci}. 
We already mentioned above  the a priori estimates \eqref{ConsOmega} regarding the vorticity. 
From Kelvin's theorem and the vorticity equation \eqref{vorty}, one also has, at least for smooth solutions,  the following: 
\begin{equation*}
\gamma  = \int_{  \partial \mathcal{S}_{0}} v_{0}  \cdot  \tau \, ds  \quad  \text{ and } \quad  \int_{ \mathcal{F}_0} \omega(t,x) \, dx  = \int_{\mathcal{F}_{0}} \omega_{0}(x) \, dx.
\end{equation*}
In particular it follows from these two conservation laws that  the coefficient $ \beta$ in  \eqref{allo2} is constant in time.
Regarding the energy observe that $ \tilde{v}$ is in  $L^2 ({\mathcal F}_0)$ whereas $v$ is not\footnote{It is interesting to compare the decomposition above with the one used in Section \ref{chap1}, cf.  \eqref{allo1}.} unless $\beta =0$.
Still we have the following result.
\begin{prop} \label{PE2}
There exists a constant $C > 0$ (depending only on $\mathcal{S}_0 $, $m $ and ${\mathcal J}$) such that for any smooth solution $(\ell ,r ,v)$ of the problem  \eqref{Euler11}--\eqref{Solide1ci} on the time interval $\lbrack 0, T \rbrack$,  with compactly supported fluid vorticity, the energy-like quantity defined by:
\begin{equation*}
\tilde{\mathcal{E}} (t) := \frac{1}{2} \left( m  |\ell(t)|^2 +   \mathcal{J} r(t)^2 +  \int_{ \mathcal{F}_{0} }  \tilde{v} (t,\cdot )^2 dx \right),
\end{equation*}
satisfies the inequality $ \tilde{\mathcal{E}} (t) \leqslant  \tilde{\mathcal{E}} (0) e^{C|\beta| t} .$
\end{prop}
\begin{rem}
In the case where $\beta=0$, that is, when the solution is of finite energy, the energy is conserved.
\end{rem}
\begin{proof}
We start by recalling that a classical solution satisfies \eqref{SemiW}, and we use the 
decomposition \eqref{allo2} in the left hand side.
One has for all $t$:
\begin{equation} \label{decay}
 \tilde{v} = {\mathcal O}\left( \frac{1}{|x|^{2}}\right) \ \text{ and } \nabla  \tilde{v} = {\mathcal O}\left( \frac{1}{|x|^{3}}\right) \  \text{ as } x \rightarrow \infty .
\end{equation}
We use an integration by parts for the first term of the right hand side of  \eqref{SemiW} to get that 
for any test function $\Psi \in \tilde{\mathcal{H} }_{T}$,
\begin{eqnarray*}
m \, \ell' \cdot \ell_\Psi + \mathcal{J} r' r_\Psi +  \int_{\mathcal{F}_{0} } \partial_{t} \tilde{v}  \cdot \Psi \, dx
 = - \int_{\mathcal{F}_0}  \Psi \cdot\left(\left(v-\ell - r x^\perp \right)\cdot\nabla\right) v \, dx
\\ \quad  -  \int_{\mathcal{F}_0} r v^\perp   \cdot \Psi  \, dx -  m r \ell^\perp   \cdot  \ell_{\Psi}         .
\end{eqnarray*}
Then, using a standard regular truncation process, we obtain that the previous identity is still valid for the test function $\Psi $ defined by $ \Psi (t,x) =  \tilde{v}(t,x)  $ for $(t,x)$ in $ [ 0,T] \times \mathcal{F}_0$ and  $ \Psi (t,x) =   \ell(t)  +  r(t) x^\perp $ for  $(t,x)$ in $ [ 0,T] \times \mathcal{S}_0$. Hence we get:
\begin{eqnarray*}
\tilde{\mathcal{E}}' (t)
&=& - \int_{\mathcal{F} _0 }  \tilde{v} \cdot\left(\left(v-\ell - r x^\perp \right)\cdot\nabla\right)  {v} \, dx
-  \int_{\mathcal{F}_0} r v^\perp   \cdot  \tilde{v}  \, dx \\ 
&=& - \int_{\mathcal{F} _0 }  \tilde{v}  \cdot \Big( (  {v} - \ell-rx^\perp) \cdot\nabla  \tilde{v}   \Big) \, dx
- \beta  \int_{\mathcal{F} _0 }  \tilde{v}  \cdot ( (\tilde{v} \cdot \nabla) H )  \, dx
 \\
&&+ \beta  \int_{\mathcal{F} _0 }  \tilde{v}  \cdot ( (\ell  \cdot\nabla) H )  \, dx - \beta r \int_{\mathcal{F} _0 }  \tilde{v}  \cdot    (H^\perp - (x^\perp  \cdot\nabla) H  ) \, dx  
 \\
&&-  \beta^2 \int_{\mathcal{F} _0 }  \tilde{v}  \cdot ( (H \cdot \nabla) H ) \, dx \\
& =: & I_1 + I_{2}+ I_{3}+ I_4  + I_5 .
\end{eqnarray*}
Integrating by parts we infer that  $I_1 = 0$, since $v - \ell -r x^\perp$ is a divergence free vector field, tangent to the boundary. Let us stress that there is no contribution at infinity because of the decay properties of the various fields involved, see \eqref{decay}. \par 
On the other hand, using the smoothness and decay at infinity of $H$,  we get that there exists $C>0$ depending only on ${\mathcal{F} _0 }$ such that
\begin{equation*}
| I_2 | + | I_3 | + | I_4 | \leqslant C |\beta| \left( \int_{\mathcal{F} _0 } \tilde{v}^2 \,dx  + |\ell|^2 + r^2 \right).
\end{equation*}
Let us now turn our attention to $I_5$.
We first  use  that $H$ being curl free, we have 
\begin{equation*}
 \int_{\mathcal{F} _0 }  \tilde{v}  \cdot ( (H \cdot \nabla) H ) \, dx = \frac{1}{2} \int_{\mathcal{F} _0 }  (\tilde{v}  \cdot \nabla  )| H |^2  \, dx,
\end{equation*}
and then an integration by parts and \eqref{Euler13}
to obtain 
\begin{eqnarray*}
\int_{\mathcal{F} _0 }  \tilde{v}  \cdot ( (H \cdot \nabla) H ) \, dx 
&=& \frac{1}{2} \int_{\partial \mathcal{S} _0 }  (\tilde{v}  \cdot  n )| H |^2 \, ds 
\\
&& = \frac{1}{2} \ell \cdot \int_{\partial \mathcal{S} _0 }  | H |^2 n \, ds  + \frac{1}{2} r \int_{\partial \mathcal{S} _0 }  | H |^2 x^\perp \cdot n \, ds .
\end{eqnarray*}
We make use of  Blasius' lemma and Cauchy's residue Theorem to obtain  that $I_5 = 0$. 

Collecting all these estimates it only remains to use Gronwall's lemma to conclude.
\end{proof}
Proposition \ref{PE2} provides in particular some a priori estimates of the solid velocity. 
We aim now at finding an a priori bound of the body acceleration. 
\begin{prop}
\label{Toascoli}
There exists a constant $C>0$ depending only on $\mathcal{S}_0 $, $m $,  ${\mathcal J}$, $\beta$ and  $ \tilde{\mathcal{E}} (0)$ such that  any classical solution to \eqref{Euler11}--\eqref{Solide1ci} satisfies the estimate $\| (\ell' , r')  \|_{L^{\infty} (0,T)} \leqslant C $.
\end{prop}
\begin{proof}
Again, after a regular truncation procedure, we can use \eqref{SemiW} with, as test functions, the functions $(\Psi_{i} )_{i=1,2,3}$ defined by $\Psi_i  = \nabla \Phi_i $ in $ \mathcal{F}_{0}$ and $\Psi_i = e_i$, for $i=1,2$ and $\Psi_{3 }=  x^\perp $ in $ \mathcal{S}_{0}$.
We observe that the left hand side of \eqref{SemiW} can be recast in terms of the  acceleration of the body only thanks to the added mass phenomenon:
\begin{eqnarray*}
\big( <\partial_{t}  \overline{v} , \Psi_i   > \big)_{i=1,2,3}
&=& \mathcal{M}_g \begin{pmatrix} \ell \\r \end{pmatrix} '(t) + (  \int_{  \mathcal{F}_{0}  }  \partial_{t}  v  \cdot  \nabla   \Phi_i  \, dx   )_{i=1,2,3}
\\ &=& \big(  \mathcal{M}_g+\mathcal{M}_a \big)  \begin{pmatrix} \ell \\r \end{pmatrix} ' ,
\end{eqnarray*}
using \eqref{vieux} (observe that the new contribution in the velocity due to the vorticity does not modify this identity). We recall that $\mathcal{M}_g$ and $\mathcal{M}_a$ were respectively  given by  \eqref{DefMG} and  \eqref{AddedMass}. 
Therefore we infer from  \eqref{SemiW}  that
\begin{eqnarray}
\label{EqL2}
&& (\mathcal{M}_g + \mathcal{M}_a )  \begin{pmatrix} \ell \\ r \end{pmatrix} '
 =  \begin{pmatrix} -m r \ell^\perp \\ 0 \end{pmatrix}
\\  \nonumber  &&
 \quad +  \begin{pmatrix}
   \int_{  \mathcal{F}_{0}  }  v   \cdot    \left[((v-\ell-r x^\perp) \cdot \nabla) \nabla \Phi_i \right] \, dx
   - \int_{\mathcal{F}_{0}} r v^\perp \cdot \nabla \Phi_i \, dx 
  \end{pmatrix}_{i \in \{1,2,3\}}  
  .
  \end{eqnarray}
It then suffices to use the decomposition  \eqref{allo2}, 
Proposition  \ref{PE2} and the decay properties of $\nabla \Phi_i$ and $H$ to obtain a bound of  $\ell'$ and $r'$.
\end{proof}

With these a priori estimates in hand there are several classical ways to infer  the local in time existence of 
 a  weak solution to \eqref{Euler11}-\eqref{Solide1ci} as promised in the statement of Theorem \ref{ThmYudo0}.
Since we do not have mentioned the existence of smooth solutions, one method to produce directly weak solutions is to apply the following Schauder's fixed point theorem.

\begin{thm}\label{Sc}
Let $E$ denotes a Banach space and let $C$ be a nonempty closed
convex set in $E$. Let $F : C\mapsto  C$ be a continuous map such that $F(C)  \subset K$, where
$K$ is a compact subset of $C$. Then $F$ has a fixed point in $K$.
\end{thm}

Theorem  \ref{Sc} is applied to an 
operator $F$ 
which maps
 $(\omega,\ell,r)$ 
to  $(\tilde{\omega},\tilde{\ell},\tilde{r})$
 as follows: 
\begin{gather*}
\partial_t  \tilde{\omega} + \left[(v-\ell-r x^\perp)\cdot\nabla\right]  \tilde{\omega} =0    \text{ in }  {\mathcal F}_{0} ,
\\ (\mathcal{M}_g + \mathcal{M}_a )  \begin{pmatrix} \tilde{\ell} \\ \tilde{r} \end{pmatrix} '
 = 
  \begin{pmatrix}
   \int_{  \mathcal{F}_{0}  }  
 \Big(  v   \cdot    \left[((v-\ell-r x^\perp) \cdot \nabla) \nabla \Phi_i \right] 
   -  r v^\perp \cdot \nabla \Phi_i   
    \Big)
    \, dx 
  \end{pmatrix}_{i \in \{1,2,3\}} 
  \\ +
    \begin{pmatrix} -m r \ell^\perp \\ 0 \end{pmatrix} ,
  \end{gather*}
where $v$ is given by \eqref{allo2}, 
with some appropriate sets $C$ and $K$ of functions $(\omega,\ell,r)$ defined on a time interval $(0,T)$ with $T$ small enough. 
We thus observe that a fixed point of $F$ verifies  \eqref{vorty} and   \eqref{EqL2}.
Moreover the previous a priori bounds can be adapted to the system above and this allows to apply Schauder's fixed point theorem.
In particular the compactness for the $(\ell,r)$-part is given by an appropriate modification of Proposition \ref{Toascoli}.

The global in time existence follows then from global a priori estimates in particular of the vorticity.

On the other hand the uniqueness part of Theorem \ref{ThmYudo0} 
relies on Yudovich's method for the case of a fluid alone. Suppose that we have two solutions $(\ell_{1},r_{1},v_{1})$ and $(\ell_{2},r_{2},v_{2})$ with the same initial data (observe that in this part of the proof the indices do not stand for the components.)  In particular, they share the same circulation $\gamma$ and initial vorticity $w_{0}$. As a consequence, despite the fact that $v_{1}$ and $v_{2}$ are not necessarily in $L^{2}({\mathcal F}_{0})$, their difference $v_{1} - v_{2}$ does belong to $L^{\infty}(0,T;L^{2}({\mathcal F}_{0}))$ with\footnote{Recall that both $v_{1}$ and $v_{2}$ are harmonic for $|x|$ large enough and converge to $0$ at infinity.}
\begin{equation} \label{EstDiffv}
v_{1} - v_{2} = {\mathcal O}\left( \frac{1}{|x|^{2}}\right) \text{ and } \nabla( v_{1} - v_{2}) = {\mathcal O}\left( \frac{1}{|x|^{3}}\right) \ \text{ as } |x| \rightarrow +\infty.
\end{equation}
Moreover $\ell_{1}$,  $\ell_{2}$, $r_{1}$, $r_{2}$ belong to $W^{1,\infty}(0,T)$. As a consequence, one can prove  that $\nabla q_{1}$ and $\nabla q_{2}$ belong to $L^{\infty}(0,T;L^{2}({\mathcal F}_{0}))$. 
Then defining $\breve{\ell}:=\ell_{1}-\ell_{2}$, $\breve{r}:=r_{1}-r_{2}$, $\breve{v}:= v_{1} -v_{2}$ and $\breve{q}=q_{1}-q_{2}$,  we deduce from \eqref{Euler11} that
\begin{equation*}
\frac{\partial \breve{v}}{\partial t}
+ \left[(v_{1} -\ell_{1} -r_{1} x^\perp)\cdot\nabla\right]\breve{v} 
+ \left[(\breve{v} -\breve{\ell} -\breve{r} x^\perp)\cdot\nabla\right] v_{2} 
+ r_{1} \breve{v}^\perp + \breve{r} v_{2}^{\perp}+\nabla \breve{q} =0 .
\end{equation*}
\par
We multiply by $\breve{v}$, integrate over ${\mathcal F}_{0}$ and integrate by parts (which is permitted by \eqref{EstDiffv} and by the regularity of the pressure), and deduce:
\begin{equation*}
\frac{1}{2} \frac{d}{dt} \| \breve{v} \|_{L^{2}}^{2}
+ \int_{{\mathcal F}_{0}}  \breve{v} \cdot \left[ (\breve{v} -\breve{\ell} -\breve{r} x^\perp)\cdot \nabla v_{2} \right]  \, dx
+\breve{r} \int_{{\mathcal F}_{0}}  \breve{v}   \cdot v_{2}^{\perp} \, dx
+ \int_{\partial {\mathcal F}_{0}} \breve{q} \breve{v}\cdot n=0 .
\end{equation*}
For what concerns the last term, 
\begin{eqnarray*}
\int_{\partial {\mathcal F}_{0}} \breve{q} \breve{v}\cdot n &=& \breve{\ell} \cdot \int_{\partial {\mathcal F}_{0}} \breve{q} n + \breve{r}  \int_{\partial {\mathcal F}_{0}} \breve{q} x^{\perp} \cdot n \\
&=& m\breve{\ell} \cdot \big( \breve{\ell}'+ \breve{r}\ell_{1}^{\perp} + r_{2} \breve{\ell}^{\perp} \big) + {\mathcal J} \breve{r}\breve{r}' \\
&=& m \breve{r}\, \breve{\ell} \cdot \ell_{1}^{\perp} + m\breve{\ell} \cdot \breve{\ell}'+ {\mathcal J} \breve{r}\breve{r}'.
\end{eqnarray*}
Using $(x^\perp \cdot \nabla) v_{2} = \nabla (x^{\perp} \cdot v_{2}) - v_{2}^{\perp} - x^{\perp} \omega_{2}$,
and an integration by parts, one has
\begin{equation*}
\int_{{\mathcal F}_{0}} \breve{v} \cdot [(x^\perp \cdot \nabla) v_{2}] \, dx 
= \int_{{\mathcal S}_{0}} (x^{\perp} \cdot v_{2}) [(\breve{\ell} + \breve{r} x^{\perp}) \cdot n ]\, ds+
\int_{{\mathcal F}_{0}} \breve{v} \cdot (- v_{2}^{\perp} - x^{\perp} \omega_{2}) \, dx  .
\end{equation*}
Hence using the boundedness of $v_{2}$ and $\omega_{2}$ in $L^{\infty}(0,T;L^{\infty}({\mathcal F}_{0}))$, the boundedness of $\ell^{1}$ and the one of $\mbox{Supp\,}(\omega_{2})$, we arrive to
\begin{equation*}
\frac{d}{dt} \big( \| \breve{v} \|_{L^{2}}^{2} + \| \breve{\ell} \|^{2} + \| \breve{r} \|^{2} \big) 
\leqslant C  \Big( \| \breve{v} \|_{L^{2}}^{2} + \| \breve{\ell} \|^{2} + \| \breve{r} \|^{2} + \| \nabla v_{2} \|_{L^{p}} \| \breve{v}^{2} \|_{L^{p'}} \Big),
\end{equation*}
for $p>2$.
Here, the various constants $C$ may depend on ${\mathcal S}_{0}$ and on the solutions $(\ell_{1},r_{1},v_{1})$ and $(\ell_{2},r_{2},v_{2})$, but not on $p$. 
Hence using elliptic regularity and interpolation, we obtain that for $p$ large,
\begin{eqnarray*}
\frac{d}{dt} \big( \| \breve{v} \|_{L^{2}}^{2} + \| \breve{\ell} \|^{2} + \| \breve{r} \|^{2} \big) 
&\leqslant& C  \Big( \| \breve{v} \|_{L^{2}}^{2} + \| \breve{\ell} \|^{2} + \| \breve{r} \|^{2} \Big) + C p \| \breve{v}^{2} \|_{L^{p'}} \\
&\leqslant& C  \Big( \| \breve{v} \|_{L^{2}}^{2} + \| \breve{\ell} \|^{2} + \| \breve{r} \|^{2} \Big) 
+ C p \| \breve{v} \|_{L^{2}}^{\frac{2}{p'}}  \| \breve{v}^{2} \|_{L^{\infty}}^{\frac{1}{p}}  .
\end{eqnarray*}
For some constant $C>0$, we have on $[0,T]$:
$\| \breve{v} \|_{L^{2}}^{2} + \| \breve{\ell} \|^{2} + \| \breve{r} \|^{2} \leqslant K$, so for some $C>0$ one has in particular
\begin{equation*}
\frac{d}{dt} \big( \| \breve{v} \|_{L^{2}}^{2} + \| \breve{\ell} \|^{2} + \| \breve{r} \|^{2} \big) 
\leqslant C p \Big( \| \breve{v} \|_{L^{2}}^{2} + \| \breve{\ell} \|^{2} + \| \breve{r} \|^{2} \Big)^{1/p'}.
\end{equation*}
Now the unique solution to $y' = N y^{\delta}$ and $y(0)= \varepsilon>0$ for $\delta \in (0,1)$ and $N >0$ is given by
$y(t) = \Big[ (1-\delta) Nt + \varepsilon^{1-\delta}\Big]^{\frac{1}{1-\delta}}.$
Hence a comparison argument proves that 
$\| \breve{v} \|_{L^{2}}^{2} + \| \breve{\ell} \|^{2} + \| \breve{r} \|^{2} \leqslant (Ct)^{p}.$
We conclude that $\breve{v}=0$ for $t< 1/C$ by letting $p $ converge to $ +\infty$.


\subsection{Energy conservation}
\label{NRJ}

Despite the fact that the energy-type bound obtained in  Proposition \ref{PE2} turned out to be sufficient in order to deal with the Cauchy problem, one may wonder if even in the case where $\beta  \neq 0$ (for which the kinetic fluid energy is infinite, see the discussion above Proposition \ref{PE2}) 
there is a renormalized energy which is exactly conserved at least for regular enough solutions to  the problem  \eqref{Euler11}--\eqref{Solide1ci}. 
Another motivation is that the constant $C$ which appears in  Proposition \ref{PE2}  depends on the body geometry in 
such a way that the corresponding estimate is not uniform in the 
zero radius limit. One may hope that an exactly conserved quantity overcomes this lack of uniformity. 
 For any $p$ in $\mathbb R^3$ and for any $\omega \in L^\infty_c (\overline{\mathcal{F}_{0}})$ we define 
 \begin{gather}
  \label{Hamiltonien1ereForme}
  \mathcal{E} (p,\omega) :=   \frac12 p \cdot (\mathcal{M}_g + \mathcal{M}_a) p - \frac12 \int_{\mathcal{F}_0 \times  \mathcal{F}_0 }  G_{H} (x,y) \omega (x) \omega (y) \, dx \, dy
\\ \nonumber -    \gamma \int_{\mathcal{F}_0  } \omega (x)  \Psi_{H} (x) \, dx ,
\end{gather}
where $ \Psi_{H}$ is defined in Section \ref{HF} and $G_{H}$ is the so-called hydrodynamic  Green function defined by 
\begin{equation}
  \label{GpasH}
G_{H} (x,y) := G (x,y) +  \Psi_{H} (x) +  \Psi_{H} (y),
\end{equation}
where $ G$ is  the standard Dirichlet Green's function defined at the beginning of 
Section \ref{sec-coche}. We recall that $\mathcal{M}_g$ and $\mathcal{M}_a$ were respectively  given by  \eqref{DefMG} and  \eqref{AddedMass}. 
Observe that in the irrotational case where  $\omega$ is vanishing on $\overline{\mathcal{F}_{0}}$ the energy  $\mathcal{E} (p,0)$ is equal to the quantity \eqref{NR} which was proved to be 
 conserved in Proposition \ref{CP}.  
  Indeed the three terms in the right hand side of  \eqref{Hamiltonien1ereForme} can therefore be respectively interpreted as 
 the  kinetic energy of the rigid body with its total inertia included  its genuine inertia and  the added inertia due to the incompressible fluid around, 
 the self-interaction energy of the fluid vorticity and the interaction between the fluid vorticity and the circulation around the body. 
  
The following energy conservation property can therefore be interpreted as an extension of Proposition \ref{CP} to the rotational case. 
\begin{prop} \label{PropKirchoffSueur}
For any smooth solution $(\ell ,r ,v)$ of the problem  \eqref{Euler11}-\eqref{Solide1ci}  with compactly supported vorticity, 
the  quantity $ \mathcal{E} (\ell ,r ,\operatorname{curl} v)$ is conserved along the motion.
\end{prop} 
\begin{proof}
We will proceed in two steps.  We first give another form of \eqref{Hamiltonien1ereForme}. Let us prove that
\begin{equation} \label{Hamiltonien2emeForme}
\mathcal{E} (\ell ,r ,\operatorname{curl} v)= 
   \frac12 p \cdot \mathcal{M}_g \,  p 
+ \frac12 \int_{ \mathcal{F}_{0} } ( |\tilde{v}|^2 + 2  \beta \tilde{v} \cdot    H   ) \, dx ,
\end{equation}
with $ \tilde{v}$ and $\beta$ given by the decomposition \eqref{allo2}.
Observe that the right hand side above can be obtained formally by expanding  the  natural  total kinetic energy of the ``fluid+solid'' system 
$ {\mathcal E}_{g} ( {p} ) 
+\frac12 \int_{ \mathcal{F}_{0} } |{v}(t,\cdot )|^2 \, dx$ thanks to the decomposition   \eqref{allo2} and discarding the infinite term 
$\frac12 \beta^2  \int_{ \mathcal{F}_{0} } |H|^{2} \, dx $ associated with the circulation around the body.
Note in particular that $\tilde{v}(x)= {\mathcal O}(1/|x|^{2}) \ \text{ as }  \ |x| \rightarrow +\infty $, 
so that the last integral in  the right hand side  of  \eqref{Hamiltonien2emeForme} is well-defined. 
 Let us highlight that a  difference with the irrotational case discussed in  Section \ref{sec-decompo} is that 
 $K[\omega] $ and $H$ being not orthogonal in $L^2 (\mathcal{F}_{0})$ there is a crossed term, given by the contribution of the last summand of the last term  of \eqref{Hamiltonien2emeForme}, and which encompasses a dependence on $\gamma$, through $\beta$.
 
In order to simplify the proof of  \eqref{Hamiltonien2emeForme} we introduce a few notations. 
Let us denote  $\Psi (x)  :=  \int_{\mathcal{F}_{0}} G(x,y) \omega(y) dy$
which is a stream function of  $K[\omega]$ vanishing on the boundary $\mathcal{S}_{0}$, so that 
$K[\omega] = \nabla^{\perp} \Psi$.
Let us also  denote  $\nabla \Phi :=  \ell_1 \nabla \Phi_1 +  \ell_2 \nabla \Phi_2 +  r  \nabla \Phi_3$,
so that $\tilde{v} = \nabla^{\perp} \Psi  +  \nabla \Phi$.
Then we compute 
\begin{equation}
\label{A0}
\int_{\mathcal{F}_0}  |\tilde{v}|^2 dx = 
\int_{\mathcal{F}_0} \nabla^\perp  \Psi \cdot   \tilde{v}
+  \int_{\mathcal{F}_0}  \nabla^\perp   \Psi \cdot  \nabla \Phi 
+ \int_{\mathcal{F}_0}  \nabla \Phi  \cdot  \nabla \Phi  .
\end{equation}
First, integrating by parts yields
\begin{gather}
\label{A1}
\int_{\mathcal{F}_0} \nabla^\perp \Psi \cdot \tilde{v}
= - \int_{\mathcal{F}_0 \times \mathcal{F}_0}  G(x,y) \omega (x) \omega (y) \, dx \, dy , \\
\label{A2}
\int_{\mathcal{F}_0}  \nabla^\perp  \Psi \cdot  \nabla \Phi = 0   \text{ and  } \int_{\mathcal{F}_0}  \tilde{v}  \cdot  H = - \int_{\mathcal{F}_0  } \omega (x)  \Psi_H (x) dx .
\end{gather}
There is no boundary terms since  $ \Psi$ and $ \Psi_H$ vanish on the boundary $\mathcal{S}_{0}$, and $ \nabla \Phi $ and $ \tilde{v}$  decrease also   like $1/ | x |^2$  at infinity. 
\par
Also, by definition, we have 
\begin{equation}
\label{A4}
\frac12 \int_{\mathcal{F}_0}  \nabla \Phi  \cdot  \nabla \Phi = \frac12 p \cdot \mathcal{M}_a p .
\end{equation}
Thus combining \eqref{A0}-\eqref{A4} we obtain that 
\begin{eqnarray*}
 \frac12 \int_{ \mathcal{F}_{0} } ( |\tilde{v}|^2 + 2  \beta \tilde{v} \cdot    H   ) \, dx 
 &=&
 \frac12 p \cdot \mathcal{M}_a p
  - \frac12 \int_{\mathcal{F}_0 \times  \mathcal{F}_0 }  G (x,y) \omega (x) \omega (y) \, dx \, dy
\\ &&-    \beta \int_{\mathcal{F}_0  } \omega (x)  \Psi_{H} (x) \, dx .
\end{eqnarray*}
 This entails \eqref{Hamiltonien2emeForme} thanks to  Fubini's theorem, \eqref{allo2} and \eqref{GpasH}.

Now by taking the time derivative of \eqref{Hamiltonien2emeForme}, using the  definition of   $\mathcal{M}_g$  given in  \eqref{DefMG}, \eqref{allo2}  and that  the coefficient $ \beta$ in  \eqref{allo2} is constant in time, 
one obtains:
\begin{align}
\label{hand}
\frac{d}{dt} \Big( \mathcal{E} (\ell ,r ,\operatorname{curl} v) \Big)&=  m \ell  \cdot \ell' (t) + {\mathcal J} r r' (t) +    \int_{\mathcal{F}_0  } \partial_ t  v \cdot  v.  
\end{align}
Using now the fluid equation \eqref{Euler11} one deduces from \eqref{hand} that 
$\frac{d}{dt} \Big( \mathcal{E} (\ell ,r ,\operatorname{curl} v) \Big)=  I_1 + I_2 + I_3$, where 
where
\begin{gather*}
I_1 :=  m \ell   \cdot \ell' (t) + {\mathcal J} r r' (t)  - \int_{\mathcal{F}_0  } \nabla q \cdot  v, \quad  \\ 
I_2 := -\int_{\mathcal{F}_0  } (v-\ell)\cdot\nabla v  \cdot v    \text{ and }
I_3 := - r  \int_{\mathcal{F}_0} [v^{\perp} - (x^{\perp} \cdot \nabla) v] \cdot v .
\end{gather*}
One easily justifies from the decay properties of 
$H$ and $ \tilde{v}$ that 
 each integral above is convergent.
This allows to integrate by parts both  $I_1$ and  $I_2$. 
Using  the interface condition \eqref{Euler13} and then Newton's equations for the body's dynamics, we obtain that  $I_1 = 0$. 
For what concerns $I_{2}$ we get that 
\begin{equation*}
I_2 =  - \frac{1}{2} \int_{\partial \mathcal{S}_0} |v|^2 (v - \ell)\cdot n  .
\end{equation*}
For what concerns $I_{3}$, we consider $R>0$ large in order that ${\mathcal S}_{0} \subset B(0,R)$, and consider the same integral as $I_{3}$, over ${\mathcal F}_{0} \cap B(0,R)$. Integrating by parts we obtain
\begin{equation*}
\int_{\mathcal{F}_0 \cap B(0,R)} [v^{\perp} - (x^{\perp} \cdot \nabla) v] \cdot v
= - \int_{\partial \mathcal{S}_0} (x^{\perp} \cdot n) \frac{|v|^{2}}{2} - \int_{S(0,R)} (x^{\perp} \cdot n) \frac{|v|^{2}}{2},
\end{equation*}
where we denote by $n$ also the unit outward normal on the circle $S(0,R)$. Of course $x^{\perp} \cdot n=0$ on $S(0,R)$, so letting $R \rightarrow +\infty$, we end up with
\begin{equation*}
I_{3} =\frac{1}{2} \int_{\partial \mathcal{S}_0} (rx^{\perp} \cdot n)|v|^{2}.
\end{equation*}
Using \eqref{Euler13} we deduce $I_{2}+I_{3}=0$, so in total we get $\frac{d}{dt} \Big( \mathcal{E} (\ell ,r ,\operatorname{curl} v) \Big)= 0$.
\end{proof}
One difficulty with the  quantity $ \mathcal{E}$  is that  both its form 
\eqref{Hamiltonien1ereForme}
and 
\eqref{Hamiltonien2emeForme} are not the sum of positive terms. 
However one may extirpate some information from the conservation of  $ \mathcal{E}$ thanks to the support of vorticity. 
The basic idea can be exhibited thanks to the 
following technical lemma, having in mind that the hydrodynamic Green function $ G_{H} (x,y)$ behaves like $ \frac{ 1 }{ 2\pi} \ln |x-y| $ at infinity.  
%
%
\begin{lem}\label{intln}
Let $f$ in $L^1 (\mathbb{R}^{2})\cap L^\infty (\mathbb{R}^2)$. We denote by 
$\rho_f :=\inf \, \{ d >1 \ / \  {\rm Supp}(f) \subset B(0,d) \}$.
Then there exists $C>0$ such that or any $y\in B(0,\rho_f)$, 
\begin{equation*}
\int_{\mathbb{R}^2} \Bigl| \ln |x-y| f(x) \Bigl| \, dx \leqslant C\|f \|_{L^\infty} + \ln(2\rho_f) \|f\|_{L^1} .
\end{equation*}
\end{lem}
\begin{proof}
It is sufficient to decompose the integral depending on whether $|x-y|\geqslant 1$ or not. 
\end{proof}
As a consequence we have the following result.
\begin{cor}\label{ProRata}
One has the following estimate for some  positive constant $C$ depending only  on $m$, $ {\mathcal J}$, $ \| \omega_{0} \|_{L^{1} \cap L^{\infty}} $, $ |\ell_{0}|$, $ |r_{0}|$, $|\gamma|, \rho(0)$,  and the geometry: $ | \ell(t) | + |r(t)| \leqslant C [1 + \ln (\rho(t))]$,
where $\rho(t) := \inf \{ d >1 \ / \  {\rm Supp}(\omega(t,\cdot)) \subset B(0,d) \}$.
\end{cor}
%
\subsection{A macroscopic normal form tailored for the zero-radius limit}

We define the set 
\begin{equation*}
\mathcal B:=  \cup_{q \in \mathbb{R}^3} \,  \{  q   \} \times \mathbb{R}^3 \times \mathbb{R} \times L^\infty (\mathcal F(q); \mathbb{R}) .
\end{equation*}
The following result is deduced, by going back to the original frame, 
from the existence and uniqueness result established in Section \ref{sec-coche}
for the div/curl type system satisfied by the velocity in the body frame.
 \begin{prop}
 \label{bs}
For  any $(q,p,\gamma,\omega) $ in $\mathcal B$ with $p=(\ell,r)$ in $\mathbb{R}^2 \times \mathbb{R}$, there exists
a unique  $U(q,p,\gamma,\omega)$ in  the space  $\mathcal{LL}({\mathcal F} (q))$ such that 
\begin{gather*}
 \operatorname{div} U(q,p,\gamma,\omega) = 0   \text{ and }  \operatorname{curl} U(q,p,\gamma,\omega)  = \omega   \quad   \text{for}  \ x\in  \mathcal{F}(q) , \\
 U(q,p,\gamma,\omega) \cdot n = \left(\ell+r x^\perp\right)\cdot n  \quad   \text{for}  \ x\in \partial \mathcal{S}(q)    \text{ and }  \int_{ \partial \mathcal{S}(q)} U(q,p,\gamma,\omega)  \cdot  \tau \, ds=  \gamma , \\
 U(q,p,\gamma,\omega) \longrightarrow 0 \quad  \text{as}  \ x \rightarrow  \infty .
\end{gather*}
 \end{prop}
In order to prepare the asymptotic analysis of the rigid body's dynamics in the zero radius limit, we first establish here an exact 
normal form of the Newton equations for the solid motion for a fixed radius. 
It follows the analysis performed  in Section \ref{chap2} so that we will adopt here  the real-analytic approach developed in \cite{GMS} rather than the complex-analytic approach performed  in \cite{GLS2}. This will simplify the study of  the zero radius limit in the next subsection. 
\begin{thm}
\label{enmarche}
There exists a mapping $F$  in $C (\mathcal B ;  \mathbb{R}^3 )$  depending only on $\gamma$ and $\mathcal{S}_0$ 
such that  the equations  \eqref{Euler1}-\eqref{Euler3b} are equivalent to the following system: 
\begin{gather}
\label{refor1}
({\mathcal M}_{g}   +  {\mathcal M}_{a, \theta} )  \, q''  
  + \langle  \Gamma_{a,\theta} ,q',q'\rangle
 =  F(q,q',\gamma,\omega) ,
\\  \label{refor2}
  \frac{\partial \omega }{\partial t}+ \operatorname{div}
\big(   \omega   U(q,q',\gamma,\omega)  \big) = 0   \text{ for } x \in  \mathcal{F}(q(t)) ,
\end{gather}
where ${\mathcal M}_{a, \theta} $ is given by Theorem \ref{pasdenom}  and $ \Gamma_{a,\theta}$ denotes the  a-connection associated with ${\mathcal M}_{a, \theta}$ by Definition \ref{Christ}.
\end{thm}
Above it is understood that the equivalence concerns Yudovich type solutions.

One observes that the left hand side of  \eqref{refor1}  is the same than the one of 
 \eqref{ODE_ext}.  On the other hand the right hand side of \eqref{refor1} is more intricate. 
Indeed we are going to provide a rather explicit definition of  the force term $F$.
In order to do so we  split, for $(q,p,\gamma,\omega) $ in $\mathcal B$, the vector field  $U(q,p,\gamma,\omega)$ into 
\begin{equation}
  \label{decomposition}
U(q,p,\gamma,\omega) = U_1 (q,p) + U_2 (q,\gamma,\omega),
\end{equation}
 where 
$U_1(q,p) $ denotes the potential part that is the unique solution in  the space  $\mathcal{LL}({\mathcal F} (q))$ to the following system:
\begin{gather*}
 \operatorname{div} U_1(q,p) = 0   \text{ and }  \operatorname{curl} U_1(q,p)  = 0   \quad   \text{for}  \ x\in  \mathcal{F}(q) , \\
 U_1(q,p) \cdot n = \left(\ell+r (x-h)^\perp\right)\cdot n    \quad   \text{for}  \ x\in \partial \mathcal{S}(q)    \text{ and }  \int_{ \partial \mathcal{S}(q)} U_1(q,p)  \cdot  \tau \, ds=  0 , \\
 U_1(q,p) \longrightarrow 0 \quad  \text{as}  \ x \rightarrow  \infty ,
\end{gather*}
where $q=(h,\theta)$ and $p=(\ell,r)$, 
and $U_2 (q,\gamma,\omega)$ therefore denotes the unique solution in  the space  $\mathcal{LL}({\mathcal F} (q))$ to the following system:
\begin{gather*}
 \operatorname{div} U_2 (q,\gamma,\omega) = 0   \text{ and }  \operatorname{curl} U_2 (q,\gamma,\omega)  = \omega   \quad   \text{for}  \ x\in  \mathcal{F}(q) , \\
 U_2 (q,\gamma,\omega) \cdot n =  0  \quad   \text{for}  \ x\in \partial \mathcal{S}(q)    \text{ and }  \int_{ \partial \mathcal{S}(q)} U_2 (q,\gamma,\omega)  \cdot  \tau \, ds=  \gamma , \\
 U_2 (q,\gamma,\omega) \longrightarrow 0 \quad  \text{as}  \ x \rightarrow  \infty .
\end{gather*}
Observe that the vector fields  $U_1(q,p) $ and $U_2 (q,\gamma,\omega)$ are respectively linear with respect to $p$ and $(\gamma,\omega)$ whereas their dependence on $q$ is encoded into the change of variable \eqref{chgtvar} (since their counterpart in the body frame do not depend on $q$).

Eventually, we define for $(q,p,\gamma,\omega)$ in $\mathcal B$, three vector in $\mathbb R^3$, by
\begin{align}
\label{fifi1}
B(q,\gamma,\omega) & := - \int_{\partial\mathcal S(q)}U_2 (q,\gamma,\omega) \cdot \tau \Big(   (  U_1(q,e_i)  \cdot n)_{i}   \times 
(  U_1(q,e_i)  \cdot \tau)_{i}   \Big)\, {\rm d}s,\\
\label{fifi2}
 E(q,\gamma,\omega) &:= - \frac{1}{2} \big( \int_{\partial\mathcal S(q)}  
 \left|  U_2 (q,\gamma,\omega)  \right|^2           U_1 (q,e_i) \cdot n  \,  {\rm d}s  \big)_{i} ,\\
\label{fifi3}
D(q,p,\gamma,\omega) &:= - \big( \int_{\mathcal F(q)}      \omega U(q,p,\gamma,\omega)^\perp  \cdot  U_1(q,e_i)   \,  {\rm d}x \big)_{i} ,
\end{align}
where the index $i$ runs over $1,2,3$ and the $e_i $ stands for the canonical basis of $\mathbb R^3$. 
We will prove Theorem \ref{enmarche} with the mapping $F$ given, for $(q,p,\gamma,\omega)$ in $\mathcal B$, 
by 
\begin{equation} \label{dim}
F (q,p,\gamma,\omega) :=  \gamma^2 E(q,\gamma,\omega) +  \gamma \, p \times B(q,\gamma,\omega) + D(q,p,\gamma,\omega) .
\end{equation}
Observe that the vector fields $B(q,\gamma,\omega)$ and $E(q,\gamma,\omega)$ above have the same form than 
the vector fields $B(q)$ and $E(q)$ used in Section \ref{chap2}, see \eqref{B-def} and \eqref{E-def} except that they also encompass a contribution due to  the vorticity through the vector fields $  U_2  $.
The last term in \eqref{dim} is a direct contribution of the vorticity,  in the sense that it intervenes explicitly  inside an integral over the fluid domain. In particular this term may be non vanishing even if $\gamma=0$ unlike the two first terms.  

\begin{proof}

First of all  \eqref{refor2} simply recasts the transport of the fluid vorticity
 by the fluid velocity characterized by Proposition  \ref{bs}.
 The proof of Theorem \ref{enmarche} therefore reduces to prove the equivalence of Newton's equations for the solid motion with \eqref{refor1}.
In a perhaps surprising way it seems more convenient not to use the  reformulation \eqref{Euler11}-\eqref{Solide11}  of the system  in the body frame.
Instead we rather  proceed as in  Section \ref{chap2} with a few modifications due to the fact that we now deal with a non-vanishing vorticity $\omega = \operatorname{curl}  u$.
In particular one has to modify \eqref{EQ_bernoulli} into 
\begin{equation*}
\label{EQ_bernoulliVort}
\nabla \pi =-\left( \frac{\partial u}{\partial t}+\frac{1}{2}\nabla|u^2 | +  \omega  u^\perp   \right)               \quad\text{in }\mathcal F(q).
\end{equation*}
and therefore \eqref{bern_2} becomes 
\begin{gather*}
\label{bern_2Vort}
mh''\cdot\ell^\ast+\mathcal J\theta''r^\ast = - \int_{\mathcal F(q)}\left(\frac{\partial u}{\partial t}+
\frac{1}{2}\nabla (u^2) \right)\cdot U_1 (q,p^\ast ){\rm d}x
\\  -  \int_{\mathcal F(q)}    \omega    u^\perp \cdot U_1 (q,p^\ast ){\rm d}x ,
\quad    \text{ for all } p^\ast=(\ell^\ast ,r^\ast)\in\mathbb R^3 .
\end{gather*}
Then we use that $u = U(q,q',\gamma,\omega)  = U_1 (q,q') + U_2 (q,\gamma,\omega)$  to obtain, for all $p^\ast :=(\ell^\ast ,r^\ast) \in \mathbb R^3$, 
\begin{gather}
 \nonumber
m\ell'\cdot\ell^\ast+\mathcal Jr'r^\ast 
+ \int_{\mathcal F(q)} \Big( \frac{\partial U_1 (q,q')}{\partial t}+\frac{1}{2}\nabla|U_1 (q,q')|^2 \Big) \cdot U_1 (q,p^\ast ){\rm d}x
\\  \nonumber =  - \int_{\mathcal F(q)} \big( \frac{1}{2}\nabla|U_2 (q,\gamma,\omega)|^2 \big) \cdot U_1 (q,p^\ast ){\rm d}x
 \\  \nonumber  \qquad-
\int_{\mathcal F(q)}\big(\frac{\partial U_2 (q,\gamma,\omega)}{\partial t} + \frac{1}{2}\nabla (U_1 (q,q') \cdot U_2 (q,\gamma,\omega)) \big) \cdot U_1 (q,p^\ast ){\rm d}x
\\ -  \int_{\mathcal F(q)} \omega   U(q,q',\gamma,\omega) ^\perp \cdot U_1 (q,p^\ast ){\rm d}x . \label{bern_2bis}
\end{gather}
Using Theorem \ref{pasdenom} in the case where $\gamma=0$  yields that the left hand side of \eqref{bern_2bis} is equal to inner product of  the  left hand side of \eqref{refor1} with $p^\ast$.
By integration by parts one obtains that the first term in the right hand side of \eqref{bern_2bis} is equal to inner product of $ \gamma^2 E(q,\gamma,\omega) $ with $p^\ast$.
By adapting the proof of \eqref{last_one} one proves that the second term in the right hand side of \eqref{bern_2bis} is equal to inner product of $ \gamma \, q' \times B(q,\gamma,\omega) $ with $p^\ast$.
It follows from the linearity of $ U_1 (q,p)$ with respect to $p$ that the last term of  \eqref{bern_2bis} is equal to inner product of $D(q,q',\gamma,\omega) $  with $p^\ast$, and this concludes   the proof of Theorem \ref{enmarche}.
\end{proof}

\subsection{Zero radius limit}

We now investigate the zero radius limit and therefore  assume that, for every $\varepsilon \in (0,1]$, the solid domain occupies 
\eqref{DomInit} where  $ \mathcal{S}_0$ is  a fixed simply connected smooth compact subset of $\mathbb{R}^2$.
We consider $p_0  = ( \ell_0 , r_0  ) \in  \mathbb R^3$,   $m>0$, $\mathcal J>0$,  
 $\gamma $ in $ \mathbb R$ (respectively in $ \mathbb R^*$)  in the case of a massive (resp. massless) particle.
 Let  $ \omega_0 $ in  $L_c^{\infty}(\mathbb{R}^{2}\setminus\{0\})$
 Then   for every $\varepsilon \in (0,1]$,  combining Theorem \ref{ThmYudo-intro} and Theorem \ref{enmarche}, we obtain that there exists 
  a  unique global solution 
   $(h^{\varepsilon}, \theta^{\varepsilon}, \omega^\varepsilon)$ with Yudovich regularity (in particular with bounded vorticity) to the equations
   \eqref{refor1}-\eqref{refor2}   with some  coefficients $\mathcal{M}_{g}^\varepsilon$, $\mathcal{M}_{a}^\varepsilon$, $\Gamma_{a}^\varepsilon$ and ${F}^\varepsilon$ 
   associated with $\mathcal S_0^\varepsilon$, $m^\varepsilon ,\mathcal{J}^\varepsilon $  given  in Definition \ref{massiveP} and $\gamma $, 
 and with the initial data $(q^\varepsilon , (q^\varepsilon)')(0)= (0, p_0 )$ and $\omega^{\varepsilon} |_{t=0} = \omega_{0|{\mathcal F}_{0}^{\varepsilon}}$.
  In the massless case with  $ \alpha \geq 2$, we will consider only here the case where  $\mathcal S_0$ is not a  disk. 
 As already mentioned in  the fifth remark after Definition \ref{massiveP} 
 the case where $\mathcal S_0$ is  a non-homogeneous disk requires a few adaptations and can be tackled as in  \cite{GMS} for the irrotational bounded case
(whereas this case was actually  omitted in \cite{GLS,GLS2}).
Our results then read as follows.

\begin{thm}
 \label{yudo-VO}
 Let  be given
 a  circulation  $\gamma $ in $ \mathbb R$ in the case of a  massive particle and in  $ \mathbb R^*$ in the  case of a massless particle.
Let  be given  $(\ell_0 ,r_0 ) \in \mathbb{R}^3$,  $ \omega_0 $ in  $L_c^{\infty}(\mathbb{R}^{2}\setminus\{0\})$.  
  For any $\varepsilon \in (0,1]$,  let us denote $(h^{\varepsilon}, \theta^{\varepsilon}, \omega^\varepsilon)$ the  solution to the system associated with with $ \mathcal{S}_0^\varepsilon$,  $m^\varepsilon ,\mathcal{J}^\varepsilon $, $ \ell_0$,  $  r_0$,  $\gamma$ and $\omega_{0|{\mathcal F}_{0}^{\varepsilon}}$ as above.
 Then in the zero radius limit  $\varepsilon \rightarrow 0$, with the inertia scaling  described in Definition \ref{massiveP},  one has, in the case of a massive (respectively massless) particle, 
 that  for any  $T>0$, up to a subsequence (resp. for the whole sequence), 
   $h^\varepsilon $ converges to $h$ weakly-$*$ in $W^{2,\infty} (0,T;\mathbb{R}^{2})$ (resp. in $W^{1,\infty} (0,T;\mathbb{R}^{2})$), 
 $\varepsilon \theta^{\varepsilon}$ converges to $0$ weakly-$*$ in $W^{2,\infty} (0,T;\mathbb{R})$ 
 $\omega^{\varepsilon}$ (extended by $0$ inside the solid) converges to $\omega$ in $C^{0} ([0,T]; L^{\infty}(\mathbb{R}^{2})-\text{weak-}\star)$.
Moreover one has  \eqref{EulerPoint}  in $[0,T] \times \mathbb{R}^{2}$, respectively  
   \eqref{mls1}  in the massive limit  and \eqref{mls2}  in the massless limit,
with the initial conditions 
$\omega |_{t= 0}=  \omega_0 ,\ h(0) = 0, \ h' (0) = \ell_0$ (resp.  $ \omega |_{t= 0}=  \omega_0 ,\ h(0) = 0 )$.
\end{thm}
%

%
\begin{rem}
Note that  the convergence of $h^{\varepsilon}$ cannot be strong in $W^{1,\infty} (0,T;\mathbb{R}^{2})$, in general, as this would entail that
$$
 \ell_0 = \frac{1}{2 \pi} \int_{ \mathbb{R}^{2}} \frac{(h_0 -y)^{\perp}}{| h_0 -y|^{2}}  \omega_0  (y) \, dy .
 $$
\end{rem}
Theorem \ref{thm-intro-sec3-VO} is a consequence of  Theorem \ref{yudo-VO}.

\begin{proof}
We will proceed as in the proof of Theorem  \ref{Bounded-VO} with a few modifications. 
First using $ p^\varepsilon = ( (h^\varepsilon)'  ,  \varepsilon (\theta^\varepsilon)'  )^t $ 
we obtain that  the solid equations  are of the form
\begin{gather}
\label{refor1eps}
 ( \varepsilon^\alpha  \,  \mathcal{M}_g   + \varepsilon^2 \mathcal{M}_{a,\,  \theta^\varepsilon} )  \, (p^\varepsilon)'
  + \varepsilon \langle  \Gamma_{a,\theta^\varepsilon} ,p^\varepsilon,p^\varepsilon\rangle
\\ \nonumber =  
  \gamma^2 \tilde{E}^\varepsilon (q^\varepsilon,\gamma,\omega^\varepsilon) +  \gamma \, p^\varepsilon \times \tilde{B}^\varepsilon (q^\varepsilon,\gamma,\omega^\varepsilon) + \tilde{D}^\varepsilon (q^\varepsilon,p^\varepsilon,\gamma,\omega^\varepsilon) ,
  \end{gather}
where $ \tilde{E}^\varepsilon$, $\tilde{B}^\varepsilon $ and $ \tilde{D}^\varepsilon$ are respectively deduced from   $E$, $B$ and $D$ defined in 
\eqref{fifi1}, \eqref{fifi2} and \eqref{fifi3} by some appropriate scalings. 
Here again the crucial  issue is to obtain some bounds uniformly in $\varepsilon$ in order to 
pass to the limit in \eqref{refor1eps}. 
First we look for an appropriate modification of   Corollary \ref{ProRata}  in the zero radius limit thanks to an appropriate renormalization  of 
the energy   \eqref{Hamiltonien1ereForme}   as $\varepsilon \rightarrow 0^{+}$ by discarding
some  terms which are logarithmically divergent  in the limit  but which do not bear any information on the state of the system.\footnote{Observe that the quantity  \eqref{Hamiltonien1ereForme}  was already obtained from    the  natural  total kinetic energy of the ``fluid+solid'' system by a renormalization at infinity. 
   Here the renormalization rather tackles some undesired concentrations at the center of mass of the shrinking particle.} 
This provides an uniform estimate of  $\varepsilon^{\min(1,\frac{ \alpha }{2 })} \, | (h^\varepsilon)' , \varepsilon (\theta^\varepsilon)' )    | _{\mathbb{R}^{3}} $ at least till 
the vorticity is neither too far from the solid nor too close. 
Unfortunately  in the massless case the coefficient $ \alpha$ satisfies $ \alpha >0$ and the previous estimate is not sufficient.
In order to get some improved estimates,  we expand the coefficients in \eqref{refor1eps}  as $\varepsilon \rightarrow 0^{+}$ 
 using in particular an irrotational approximation of the fluid velocity on the body's boundary in order to use 
Lamb's lemma. Some cancellations similar to 
\eqref{mesaoulent} and \eqref{grav}  allow in particular to absorb the leading orders of  the term $ \tilde{E}^\varepsilon$ into the leading part of the expansions of the terms involving $ \Gamma_{a,\theta^\varepsilon}$ and 
 $\tilde{B}^\varepsilon$ thanks to the following modulation of the velocity:
\begin{equation*} \label{DefLTilde}
\tilde{\ell}^{\varepsilon}(t):= (h^{\varepsilon})' (t) -  K_{\mathbb{R}^{2}}[\omega^{\varepsilon}(t,\cdot)](h^{\varepsilon} )
- \varepsilon \nabla K_{\mathbb{R}^{2}}[\omega^{\varepsilon}(t,\cdot)](h^{\varepsilon}) \cdot R(\theta^\varepsilon) \xi ,
\end{equation*}
where $\xi$ is the conformal center of  $\mathcal{S}_0 $, cf. \eqref{DefXi}. On the other hand the term  $ \tilde{D}^\varepsilon$ turns out to be smaller  at least till 
the vorticity stays supported at distance of order $1$ of the solid. 
  Next we introduce the notation  $\tilde{p}^{\varepsilon}:=(\tilde{\ell}^{\varepsilon},\varepsilon (\theta^{\varepsilon})')$.
We thus obtain the following asymptotic normal form.
\begin{prop}\label{Pro:NormalForm}
Let us fix $\rho>0$. There exists $C>0$ such that if for a given $T>0$ and an $\varepsilon \in (0,1)$ one has for all $t \in [0,T]$:
\begin{equation} \label{CondDistVorticiteSolide}
d(h^{\varepsilon}(t), \mbox{Supp\,} (\omega^{\varepsilon}(t))) \geqslant 1/\rho
\ \text{ and } \ 
\mbox{Supp\,} (\omega^{\varepsilon}(t)) \subset B(h^{\varepsilon}(t),\rho),
\end{equation}
then there exist
a function $G=G(\varepsilon,t): (0,1) \times [0,T] \rightarrow \mathbb{R}^{3}$ satisfying
\begin{equation} \label{IneqWG}
\left| \int_0^t \tilde{p}^{\varepsilon}(s) \cdot G (\varepsilon,s) \, ds \right| 
\leqslant \varepsilon C \left( 1 + t + \int_0^t |\tilde{p}^{\varepsilon}(s)|^{2}\, ds \right) ,
\end{equation}
and a function $F=F(\varepsilon,t):(0,1) \times [0,T] \rightarrow \mathbb{R}^{3}$ satisfying
\begin{equation} \label{IneqWNL}
|F(\varepsilon,t)| \leqslant C \left( 1+ |\tilde{p}^{\varepsilon}(t)| + \varepsilon |\tilde{p}^{\varepsilon}(t)|^{2} \right),
\end{equation}
such that one has on $[0,T]$:
\begin{eqnarray} \label{Eq:NormalForm}
 && \big( \varepsilon^\alpha {\mathcal M}_{g} + \varepsilon^2 {\mathcal M}_{a,\theta^{\varepsilon}} \big) (\tilde{p}^{\varepsilon})' 
+   \varepsilon  \langle\Gamma_{a,\theta^{\varepsilon}} , \tilde{p}^{\varepsilon}, \tilde{p}^{\varepsilon} \rangle
\\ &&\nonumber \quad  = \gamma\, \tilde{p}^{\varepsilon} \times B_{\theta^{\varepsilon}}
 + \varepsilon \gamma G(\varepsilon,t) 
+ \varepsilon^{\min(\alpha,2)} F(\varepsilon,t).
\end{eqnarray}
\end{prop}
From this normal form, we  deduce the following modulated energy estimates. 
\begin{lem} \label{Pro:ModulatedEnergy}
Let  $\rho>0$. There exists $C>0$ such that if for a given $T>0$ and an $\varepsilon \in (0,1)$ one has that \eqref{CondDistVorticiteSolide} is valid on $[0,T]$, then one has $| (h^{\varepsilon})' | + \varepsilon| (\theta^{\varepsilon})'| \leqslant C$ on $ [0,T].$
\end{lem}
\begin{proof}
Let  $\rho>0$ and let  $C>0$  be given by Proposition \ref{Pro:NormalForm}. 
Let   $T>0$ and an $\varepsilon \in (0,1)$ one has that \eqref{CondDistVorticiteSolide} is valid on $[0,T]$.
Then according to Proposition \ref{Pro:NormalForm} one has  \eqref{Eq:NormalForm} on $[0,T]$. 
It is then sufficient to multiply  \eqref{Eq:NormalForm} by $ \tilde{p}^{\varepsilon}$, to deal with the right hand side as in  Proposition 
 \ref{CP}, to use  the assumption on the initial data and finally to apply two Gr\"onwall type estimates\footnote{The second one being devoted to deduce some estimates for $(h^{\varepsilon})'$ from the ones on the modulated velocities.} to conclude. 
%
%

\end{proof}
Let us now tackle the passage to the limit.
In a first time, we obtain the convergence stated in Theorem~\ref{yudo-VO} on a small interval $[0,T]$, and only in a second time obtain this convergence on any time interval. 
We consider $T^{\varepsilon}$ the supremum of the positive real number $ \tau $ for which one has for any $t\in [0,\tau]$, 
$ {\rm d}(h^\varepsilon(t), \mbox{Supp\,} \omega^\varepsilon(t))>1/(2\rho_{T})$
and $\mbox{Supp\,} \omega^\varepsilon(t) \subset B(h^\varepsilon(t),2\rho_{T}).$
For any $\varepsilon>0$ small enough such that ${\rm d}(\mbox{Supp\,} \omega_{0}, \mathcal{S}_{0}^\varepsilon)>2\rho_{T}/3$, we have of course $T^{\varepsilon}>0$. Using Proposition~\ref{Pro:ModulatedEnergy}, we deduce that 
there exists $\varepsilon_{0}>0$ and $\underline{T}>0$ such that $\inf_{\varepsilon \in (0,\varepsilon_{0})} T^{\varepsilon} \geqslant \underline{T}.$

Thanks to  a compactness argument using these estimates, the uniqueness of the solutions in the limit and Proposition~\ref{Pro:NormalForm}
 this allows to prove the convergence claimed in Theorem~\ref{yudo-VO} locally in time, that is 
 $h^\varepsilon $ converges to $h$ weakly-$\star$ in $W^{1,\infty} (0,\underline{T};\mathbb{R}^{2})$ and  $\omega^{\varepsilon}$ converges to $\omega$ in $C^{0} ([0,\underline{T}]; L^{\infty}(\mathbb{R}^{2})-\text{weak-}\star)$.
Finally we obtain the solid part of Theorem~\ref{yudo-VO} by a sort of continuous induction argument. 
Moreover, with the previous uniform estimates, passing to the limit in the fluid equation is routine.
\end{proof}

%

%


%


\subsection*{Acknowledgment}
The author thanks  the Agence Nationale de la Recherche, Project DYFICOLTI, grant ANR-13-BS01-0003-01 and Project IFSMACS, grant ANR-15-CE40-0010  for their financial support.


\newpage
\begin{changemargin}{-1cm}{-1cm}

\tableofcontents

\end{changemargin}


\end{document}